\documentclass[a4paper,11pt]{amsart}

\usepackage{graphicx}
\usepackage{amsmath}
\usepackage{amssymb}
\usepackage[active]{srcltx}
\usepackage{hyperref}

\newtheorem{thm}{Theorem}[section]
\newtheorem{lem}[thm]{Lemma}
\newtheorem{prop}[thm]{Proposition}
\newtheorem{cor}[thm]{Corollary}
\theoremstyle{definition}

\theoremstyle{remark}
\newtheorem{remark}{Remark} %
\theoremstyle{plain}

\def\CC{{\mathbb C}}

\def\HH{{\mathbb H}}

\def\RR{{\mathbb R}}

\def\ZZ{{\mathbb Z}}

\def\veca{{\text{\boldmath$a$}}}
\def\vecb{{\text{\boldmath$b$}}}

\def\vece{{\text{\boldmath$e$}}}

\def\veck{{\text{\boldmath$k$}}}

\def\vecm{{\text{\boldmath$m$}}}
\def\vecn{{\text{\boldmath$n$}}}
\def\vecq{{\text{\boldmath$q$}}}

\def\vecr{{\text{\boldmath$r$}}}
\def\vecs{{\text{\boldmath$s$}}}

\def\vecv{{\text{\boldmath$v$}}}

\def\vecw{{\text{\boldmath$w$}}}
\def\vecx{{\text{\boldmath$x$}}}

\def\vecy{{\text{\boldmath$y$}}}
\def\vecz{{\text{\boldmath$z$}}}
\def\vecalf{{\text{\boldmath$\alpha$}}}

\def\veceta{{\text{\boldmath$\eta$}}}

\def\vecxi{{\text{\boldmath$\xi$}}}

\def\vecnull{{\text{\boldmath$0$}}}

\def\scrF{{\mathcal F}}

\def\scrS{{\mathcal S}}

\def\scrX{{\mathcal X}}
\def\scrY{{\mathcal Y}}

\def\tim{\operatorname{Im}}

\def\C{\operatorname{C{}}}

\def\L{\operatorname{L{}}}

\def\Sw{{\mathcal S}}
\def\SL{\operatorname{SL}}

\def\T{\mathbb T}

\def\supp{\operatorname{supp}}

\def\ord{\operatorname{ord}}

\def\SLZ{\SL(2,\ZZ)}
\def\SLR{\SL(2,\RR)}

\def\trans{\,^\mathrm{t}\!}

\def\Onder#1#2#3#4#5{#1 \setbox0=\hbox{$#1$}\setbox1=\hbox{$#2$}
       \dimen0=.5\wd0 \dimen1=\dimen0 \dimen2=\dp0 \dimen3=\dimen2
       \advance\dimen0 by .5\wd1 \advance\dimen0 by -#4
       \advance\dimen1 by -.5\wd1 \advance\dimen1 by -#4
       \advance\dimen2 by -#3 \advance\dimen2 by \ht1
       \advance\dimen2 by 0.3ex \advance\dimen3 by #5
        \kern-\dimen0\raisebox{-\dimen2}[0ex][\dimen3]{\box1}
       \kern\dimen1}

\newcommand{\sumone}{\sideset{}{^{(1)}}\sum_{\smatr abcd}}
\newcommand{\sumtone}{\sideset{}{^{(\widetilde 1)}}\sum_{\smatr abcd}}
\newcommand{\sumtwo}{\sideset{}{^{(2)}}\sum_{\smatr abcd}}
\renewcommand{\mod}{\:\text{mod}\:}
\newcommand{\tGamma}{\Gamma'}

\newcommand{\tscrX}{\widetilde{\scrX}}
\newcommand{\hg}{\widehat{g}}
\newcommand{\hh}{\widehat{h}}
\newcommand{\hF}{\widehat{F}}

\newcommand{\og}{\overline{g}}
\newcommand{\hpsi}{\widehat{\psi}}
\newcommand{\tF}{\widetilde{F}}
\newcommand{\tQ}{\widetilde{Q}}
\newcommand{\oGamma}{\overline{\Gamma}}
\newcommand{\bs}{\backslash}

\newcommand{\lsl}{\mathfrak{sl}}

\newcommand{\ig}{\mathfrak{g}}

\newcommand{\GaG}{\Gamma\backslash G}
\newcommand{\wh}{\widehat}
\newcommand{\tf}{\widetilde f}
\newcommand{\wdelta}{\widetilde\delta}

\newcommand{\Q}{\mathbb{Q}}
\newcommand{\R}{\mathbb{R}}
\newcommand{\Z}{\mathbb{Z}}

\newcommand{\col}{\: : \:}

\newcommand{\bn}{\mathbf{0}}

\newcommand{\tg}{\widetilde{g}}

\newcommand{\ve}{\varepsilon}

\newcommand{\cmatr}[2]{\left( \begin{matrix} #1 \\ #2 \end{matrix} \right) }
\newcommand{\scmatr}[2]{\left( \begin{smallmatrix} #1 \\ #2 \end{smallmatrix} \right) }
\newcommand{\matr}[4]{\left( \begin{matrix} #1 & #2 \\ #3 & #4 \end{matrix} \right) }
\newcommand{\smatr}[4]{\bigl( \begin{smallmatrix} #1 & #2 \\ #3 & #4 \end{smallmatrix} \bigr) }

\title[Effective equidistribution and application to quadratic forms]{An effective equidistribution result for $\SL(2,\R)\ltimes(\R^2)^{\oplus k}$ and application to inhomogeneous quadratic forms}
\author{Andreas Str\"ombergsson}
\author{Pankaj Vishe}
\address{Department of Mathematics, Box 480, Uppsala University,
SE-75106 Uppsala, Sweden\newline
\rule[0ex]{0ex}{0ex} \hspace{8pt}{\tt astrombe@math.uu.se}}
\address{Department of Mathematical Sciences, Durham University, Durham DH1 3LE, U.K.\newline
\rule[0ex]{0ex}{0ex} \hspace{8pt}{\tt  pankaj.vishe@durham.ac.uk}}
\thanks{A.S.\ was %
supported by the Swedish Research Council Grant 2016-03360. 
P.V.\ was partly supported by the G\"oran Gustafsson Foundation (KVA) at KTH and 
by the EPSRC programme grant EP/J018260/1.}

\begin{document}

\begin{abstract}
Let $G=\SL(2,\R)\ltimes(\R^2)^{\oplus k}$ and let $\Gamma$ be a congruence subgroup of $\SL(2,\Z)\ltimes(\Z^2)^{\oplus
k}$.
We prove a polynomially effective asymptotic equidistribution result for 
special types of unipotent orbits in $\GaG$
which project to pieces of closed horocycles in $\SL(2,\Z)\backslash\SL(2,\R)$.
As an application, we prove an effective quantitative Oppenheim type result for the 
quadratic form $(m_1-\alpha)^2+(m_2-\beta)^2-(m_3-\alpha)^2-(m_4-\beta)^2$,
for $(\alpha,\beta)\in\R^2$ of Diophantine type,
following the approach by Marklof \cite{MarklofpaircorrI} using theta sums.
\end{abstract}

\maketitle

\section{Introduction}
\label{intro}

The results of M.\ Ratner on measure rigidity and equidistribution of orbits
of a unipotent flow \cite{Ratner91}, \cite{mR91}, play a fundamental role in homogeneous dynamics.
These results also have   %
many applications outside of dynamics,    %
ranging from problems in number theory to mathematical physics.    %
In recent years %
there has been an increased interest in obtaining %
\textit{effective} versions of Ratner's results in special cases, 
i.e., to provide an explicit rate of density or equidistribution
for the orbits of a (non-horospherical) 
unipotent flow; cf.\ \cite{GreenTao}, \cite{EMV}, \cite{mohammadi2012}, \cite{LM}, \cite{SASL},
\cite{BV}, \cite{Prinyasart}.
In particular, in \cite{SASL} and \cite{BV}, effective equidistribution results were obtained for orbits of a 
1-parameter unipotent flow on $\SL(2,\Z)\ltimes\Z^2 \backslash \SL(2,\R)\ltimes\R^2$,
using Fourier analysis and methods of from analytic number theory,
and in the very recent paper \cite{Prinyasart},
building on similar methods,
effective equidistribution of diagonal translates of certain orbits in
$\SL(3,\Z)\ltimes\Z^3 \backslash \SL(3,\R)\ltimes\R^3$
was established.
Our purpose in the present paper is to prove results of a similar nature for
homogeneous spaces of the group 
$G=\SL(2,\R)\ltimes(\R^2)^{\oplus k}$
for $k\geq2$,
and to apply these to derive an
effective quantitative Oppenheim type result for a certain family of inhomogeneous quadratic forms
of signature $(2,2)$.
Here $(\R^2)^{\oplus k}$ denotes the direct sum of $k$ copies of $\R^2$, each provided with the standard action of
$\SL(2,\R)$.

We now turn to a precise description of our setting.
We represent vectors %
by column matrices. 
Throughout the paper we will identify $(\R^2)^{\oplus k}$ with $\R^{2k}$
so that the action of $G':=\SL(2,\R)$ is given by
\begin{align*}
\matr abcd\cmatr{\vecx}{\vecx'}=\cmatr{a\vecx+b\vecx'}{c\vecx+d\vecx'}
\qquad \text{for }\:\matr abcd\in G',\: \vecx,\vecx'\in\R^k.
\end{align*}
The elements of 
\begin{align*}
G=\SL(2,\R)\ltimes(\R^2)^{\oplus k}
\end{align*}
are then represented by pairs $(M,\vecv)\in G'\times\R^{2k}$, with a multiplication law
\begin{align*}
(M,\vecv)(M',\vecv')=(MM',\vecv+M\vecv').
\end{align*}
Let 
\begin{align*}
a(y)=\matr{\sqrt y}00{1/\sqrt y}\qquad
\text{and} \qquad
u(x)=\matr 1x01
\qquad (y>0,\: x\in\R).
\end{align*}
We will always view $G'=\SL(2,\R)$ as a subgroup of $G$ through $M\mapsto(M,\bn)$;
in particular, $a(y)$ and $u(x)$ are also elements of $G$.
We set
\begin{align*}
\overline\Gamma=\SL(2,\Z)\ltimes(\Z^2)^{\oplus k}.
\end{align*}
In our notation, this is the subgroup of all $(M,\vecv)\in G$ with $M\in\SL(2,\Z)$ and $\vecv\in\Z^{2k}$.
Given a subgroup $\Gamma$ of $\overline\Gamma$ of finite index,
we consider the homogeneous space 
\begin{align*}
X=\GaG.
\end{align*}
As we will detail %
below, this space is a torus bundle %
over a finite cover of the familiar 
3-dimensional homogeneous space $\SL(2,\Z)\backslash\SL(2,\R)$ classifying unimodular lattices in $\R^2$.
We fix $\mu$ to be the (left and right invariant) Haar measure on $G$, normalized so as to induce a probability measure
on $X$,
which we also denote by $\mu$.

The following %
equidistribution result is a
special case of \cite[Thm.\ 3]{DMS}\footnote{Apply \cite[Thm.\ 3]{DMS} with $d=2$ and $M=1_2$
and use the anti-automorphism $(M,\scmatr{\vecx}{\vecx'})\mapsto (\trans M,\trans M\scmatr{\vecx'}{-\vecx})$
of $G$ to translate from the setting with $G/\Gamma$ in \cite{DMS} into our setting with $X=\GaG$.
As noted in \cite[Remark 7.2]{DMS}, the proof of \cite[Thm.\ 3]{DMS} extends trivially to the case when $\Gamma$ is an 
arbitrary subgroup of $\SL(2,\Z)\ltimes(\Z^2)^{\oplus k}$ of finite index.};
alternatively it may be deduced (with some work) as a consequence of
\cite[Thm.\ 1.4]{Shah}.
Note that both \cite[Thm.\ 3]{DMS} and \cite[Thm.\ 1.4]{Shah}  depend crucially on Ratner's classification of invariant measures.

For any $\veca,\vecb\in\R^k$ we denote by $\veca\vecb$ the standard scalar product,
$\veca\vecb=a_1b_1+\ldots+a_kb_k$.

\begin{thm}\label{INEFFECTIVETHM}
Let $\Gamma$ be a subgroup of $\overline\Gamma=\SL(2,\Z)\ltimes(\Z^2)^{\oplus k}$ of finite index.
Fix $\vecxi=\cmatr{\vecxi_1}{\vecxi_2}$ in $\R^{2k}$ %
subject to the condition that 
there does not exist any $\vecm\in\Z^k\setminus\{\bn\}$ for which both 
$\vecm\vecxi_1$ and $\vecm\vecxi_2$ are integers.
Then for any Borel probability measure $\lambda$ on $\R$ which is absolutely continuous with respect to the Lebesgue
measure,
and any bounded continuous function $f$ on $X=\Gamma\backslash G$, %
\begin{align}\label{INEFFECTIVETHMRES}
\lim_{y\to0^+}\int_\R f\left(\Gamma\bigl(1_2,\vecxi\bigr) %
u(x)a(y)\right)\,d\lambda(x)
=%
\int_X f\,d\mu.
\end{align}
\end{thm}

In view of the relation 
\begin{align*}
u(x)a(y)=a(y)u(y^{-1}x),
\end{align*}
the integration in the left hand side of \eqref{INEFFECTIVETHMRES} is along an orbit of
the unipotent flow 
\begin{align*}
U^t:\Gamma g\mapsto \Gamma g u(t)\qquad (t\in\R)
\end{align*}
on $X$. Let $D:G\to G'$ be the natural projection sending $(M,\vecv)$ to $M$;
then $D(\Gamma)$ is a finite index subgroup of $\SL(2,\Z)$, 
and $D$ induces a projection map from $X$ to $X':=D(\Gamma)\backslash G'$,
which we also call $D$;
this realizes $X$ as a torus bundle %
over the space $X'$, which in turn
is a finite cover of $\SL(2,\Z)\backslash\SL(2,\R)$. 
The orbits which appear in \eqref{INEFFECTIVETHMRES} are exactly those orbits of the flow $U^t$
which project to a closed horocycle in $X'$ around its cusp at $\infty$.
Letting $y$ decrease towards zero means that we are considering expanding translates
of the initial orbit $x\mapsto\Gamma(1_2,\vecxi)u(x)$.

Let us note that the condition imposed on $\vecxi$ in Theorem \ref{INEFFECTIVETHM} cannot be weakened. %
Indeed, for any $\vecm\in\Z^k\setminus\{\bn\}$, set
\begin{align*}
X_\vecm:=\left\{\Gamma \left(M,\cmatr{\vecv_1}{\vecv_2}\right)\col M\in\SL(2,\R),\:\vecm\vecv_1\in\Z
,\:\vecm\vecv_2\in\Z\right\}.
\end{align*}
This is a closed embedded submanifold of codimension 2 in $X$.
If both $\vecm\vecxi_1$ and $\vecm\vecxi_2$ are integers then 
\begin{align}\label{ORBITCONFINED}
\Gamma\bigl(1_2,\vecxi\bigr)u(x)a(y)\in X_\vecm
\qquad
\text{for all }\: x\in\R,\: y>0,
\end{align}
and therefore the curve certainly cannot become equidistributed in $X$,
i.e.\ \eqref{INEFFECTIVETHMRES} fails for some $f$.
(For example, consider any bounded continuous $f\geq0$ such that $f_{|X_\vecm}\equiv0$
while $\int_X f\,d\mu>0$.)

Marklof in \cite[Thm.\ 5.7]{MarklofpaircorrI} proved
Theorem \ref{INEFFECTIVETHM} in the special case of $\vecxi_1=\bn$,
and then in \cite[Thm.\ 3.1]{Marklofmeansquare} in the special case of $\vecxi_2=\bn$.
Note that if $\vecxi_1=\bn$, the condition on $\vecxi_2$ in the theorem 
becomes that $1$ together with the $k$ components of $\vecxi_2$ 
should be linearly independent over $\Q$
(and vice versa if $\vecxi_2=\bn$).
Our main results in the present paper are 
Theorem \ref{MAINTHM1} and Theorem \ref{MAINTHM2} below, which give
\textit{effective} versions of 
these two special cases of Theorem \ref{INEFFECTIVETHM},
under the further requirement that $\Gamma$ is a congruence subgroup of $\overline\Gamma$.

To prepare for the statement of the main theorems we introduce some further notation.
For a positive integer $N$, $\Gamma(N)$ denotes the principal congruence subgroup of $\SL(2,\Z)$ of level
$N$: %
\begin{align*}
\Gamma(N)=\biggl\{\matr abcd\in\SL(2,\Z)\col \matr abcd\equiv \matr 1001 \mod N\biggr\}.
\end{align*}
We will consider $X=\GaG$ where $\Gamma$ is a subgroup of $\overline\Gamma$ of the form
$\Gamma=\Gamma(N)\ltimes\Z^{2k}$.\label{GAMMAdef}
(The case of an arbitrary congruence subgroup of $\overline\Gamma$ can easily be reduced to %
the case of $\Gamma=\Gamma(N)\ltimes\Z^{2k}$,
by using the fact that for any $q\in\Z^+$, the map $(M,\vecv)\mapsto(M,q\vecv)$ is an automorphism of $G$.)

We introduce the following cuspidal height function, for $(M,\vecv)\in G$:
\begin{align}\label{CHdef}
\scrY(M,\vecv)=\scrY(M)=\sup \{\tim \gamma M(i)\col\gamma\in\SL(2,\Z)\},
\end{align}
where in the right hand side we use the standard action of $G'=\SL(2,\R)$ on the
Poincar\'e upper half plane $\HH=\{\tau=u+iv\in\CC\col v>0\}$.
Then $\scrY(M,\vecv)\geq\sqrt3/2$ for all $(M,\vecv)\in G$.
Note that $\scrY(M,\vecv)$ depends only on the coset $\oGamma(M,\vecv)$, 
and in particular $\scrY$ can be viewed as a function on $X$.
Given $p_1,p_2,\ldots\in X$, we have $\scrY(p_j)\to\infty$ if and only if the sequence $p_1,p_2,\ldots$
leaves all compact subsets of $X$.

For $m\geq0$ %
and $a\in\R$, we let $\C_a^m(X)$ be the space of all $m$ times continuously differentiable
functions
on $X$, all of whose derivatives up to order $m$ are $\ll\scrY^{-a}$ throughout $X$.
In more precise terms, let $\ig$ be the Lie algebra of $G$, and fix a basis $X_1,\ldots,X_{2k+3}$ of $\ig$
(we make a definite choice of this basis; cf.\ \eqref{FIXEDBASIS} below).
Each $Y\in\ig$ can be realised as a left invariant differential operator
on functions on $G$, and thus also a differential operator on $X=\GaG$,
which we will also denote by $Y$.
For any $f\in\C^m(X)$, set
\begin{align}\label{CMNORMDEF}
\|f\|_{\C^m_a}:=\sum_{\ord(D)\leq m}\,\sup_{p\in X}\,\bigl|\scrY(p)^a(Df)(p)\bigr|, %
\end{align}
where the sum is taken %
over all monomials in $X_1,\ldots,X_{3+2k}$ of degree $\leq m$.
In particular, $\|\cdot\|_{\C^0_0}$ is the supremum norm.
Then $\C_a^m(X)$ is the space of all $f\in\C^m(X)$ with $\|f\|_{\C^m_a}<\infty$.

For any integer $n\geq0$ %
and real numbers $a\geq0$ and $p\in[1,+\infty]$,
we introduce the weighted Sobolev norm $S_{p,a,n}(h)$ 
on functions $h\in\C^n(\R)$ 
through
\begin{align*}
S_{p,a,n}(h)=\sum_{j=0}^n\|(1+|x|)^a\,\partial^j h(x) \|_{\L^p}.
\end{align*}

For $x\in\R$ let $\langle x\rangle$ denote the distance %
to the nearest integer; $\langle x\rangle=\min_{n\in\Z}|x-n|$.
Given $\beta>k$, $\vecxi \in \RR^{k}$, $T>0$  we define
\begin{align}\label{MMXIDEF}
\delta_{\beta,\vecxi}(T)=\sum_{\vecr\in\Z^{k}\setminus\{\bn\}}\|\vecr\|^{-\beta}\sum_{j=1}^\infty
\frac{1+\log^+\bigl(\frac{T\langle j\vecr\vecxi\rangle}j\bigr)}{j^2+Tj\langle j\vecr\vecxi\rangle}.
\end{align}
Since $\log^+(x)<x$ %
($\forall x>0$), one has
\begin{align}\label{deltabetaxiBASICBOUND}
\delta_{\beta,\vecxi}(T)\leq C_{k,\beta}:=\sum_{\vecr\in\Z^{k}\setminus\{\bn\}}\|\vecr\|^{-\beta}\sum_{j=1}^\infty j^{-2}<\infty,
\end{align}
for all $\vecxi$ and $T$.
\pagebreak

We now state our two main theorems:
\begin{thm}\label{MAINTHM1}
\emph{[Effective version of Theorem \ref{INEFFECTIVETHM} when $\vecxi_1=\bn$.]}
Let $k\geq2$ and $\Gamma=\Gamma(N)\ltimes\Z^{2k}$.
Fix $\ve>0$ and an integer $\beta\geq\max(8-k,1+k)$,
and set $m=3(\beta+k+1)$ and $a=(\beta-1)/2$.
Then for any $f\in\C^m_a(X)$,
$h\in\C^2(\R)$ with $S_{\infty,2+\ve,2}(h)<\infty$, $\vecxi_2\in\R^k$ and $y>0$, we have
\begin{align}\notag
\left|\int_\R f %
\Bigl(\Gamma\Bigl(1_2,\Bigl(\begin{matrix} \bn \\ \vecxi_2 \end{matrix}\Bigr)\Big)u(x)a(y)\Bigr)
h(x)\,dx-\int_{X}f\,d\mu\int_\R h \,dx\right |
\hspace{150pt}
\\\label{MAINTHM1RES}
\ll \|f\|_{\C^m_a} S_{\infty,2+\ve,2}(h) 
\Bigl(\delta_{\beta,\vecxi_2}(y^{-\frac12})+y^{\frac14-\ve}\Bigr),
\end{align}
where the implied constant depends only on $k$, $N$, $\ve$, $\beta$.
\end{thm}

\begin{thm}\label{MAINTHM2}
\emph{[Effective version of Theorem \ref{INEFFECTIVETHM} when $\vecxi_2=\bn$.]}
Let $k\geq2$ and $\Gamma=\Gamma(N)\ltimes\Z^{2k}$.
Fix $\ve>0$ and an integer $\beta\geq\max(7-k,1+k)$,
and set $m=3(\beta+k)+2$ and $a=(\beta-1)/2$.
Then for any $f\in\C^m_a(X)$,
$h\in\C^2(\R)$ with $S_{1,0,2}(h)<\infty$, $\vecxi_1\in\R^k$ and $y>0$, we have
\begin{align}\notag
\left|\int_\R f %
\Bigl(\Gamma\Bigl(1_2,\Bigl(\begin{matrix} \vecxi_1 \\ \bn \end{matrix}\Bigr)\Big)u(x)a(y)\Bigr)
h(x)\,dx-\int_{X}f\,d\mu\int_\R h \,dx\right |
\hspace{150pt}
\\\label{MAINTHM2RES}
\ll \|f\|_{\C^m_a} S_{1,0,2}(h) \Bigl(\delta_{\beta,\vecxi_1}(y^{-\frac12})+y^{\frac14-\ve}\Bigr),
\end{align}
where the implied constant depends only on $k$, $N$, $\ve$, $\beta$.
\end{thm}

Let us make some comments on these results.
Firstly, note that for any fixed $\vecxi_2\in\R^k$ and $\beta>k$, 
one has $\delta_{\beta,\vecxi_2}(y^{-1/2})\to0$ as $y\to0$ if and only if
$\vecr\vecxi_2\notin\Q$ for all $\vecr\in\Z^{k}\setminus\{\bn\}$.
Hence Theorem \ref{MAINTHM1} indeed gives an effective version of Theorem \ref{INEFFECTIVETHM}
in the special case when $\Gamma$ is a congruence subgroup of $\overline\Gamma$ and $\vecxi_1=\bn$.
Similarly Theorem \ref{MAINTHM2} gives an effective version of Theorem \ref{INEFFECTIVETHM}
when %
$\vecxi_2=\bn$.
Secondly, as we will explain in Section \ref{DIOPHSEC} below
(see especially Lemmata \ref{DIOFHDIMLEM} and \ref{MXIBOUNDLEM2}, and the relation \eqref{wMMcomp}),
for a sufficiently large $\beta$ and $\vecxi_2$ subject to a Diophantine condition,
the majorant function $\delta_{\beta,\vecxi_2}(T)$ has a power rate decay in $T$ as $T\to\infty$.
In particular for any $\ve>0$,
$\delta_{\beta,\vecxi_2}(T)\ll T^{\ve-1}$ holds for all $\vecxi_2\in\R^k$ outside a set of 
Hausdorff dimension $<k$.
Note that for any such $\beta$ and $\vecxi_2$, the bound in Theorem \ref{MAINTHM1}
decays like $y^{\frac14-\ve}$ as $y\to0$.
An analogous statement holds for Theorem \ref{MAINTHM2}.

One should also note that the integral in Theorem \ref{MAINTHM2} 
(but not the one in Theorem \ref{MAINTHM1}) 
runs over a \textit{closed}
orbit in $X$;
indeed the point $\Gamma\bigl(1_2,\scmatr{\vecxi_1}{\bn}\bigr)u(x)$ is invariant under
$x\mapsto x+N$, since $u(t)\scmatr{\vecxi_1}{\bn}=\scmatr{\vecxi_1}{\bn}$, $\forall t$,
and $u(N)\in\Gamma(N)$.
Hence it is only natural that the bound obtained 
in Theorem \ref{MAINTHM2} is invariant under translations of $h$.

We have made no effort to optimize the dependence on the test functions $f$ and $h$ in the theorems;
rather, we have simply imposed as much smoothness and decay of these %
as needed to comfortably reach the best decay rate with respect to $y$ that our method can give.

\vspace{2pt}

The proofs of Theorems \ref{MAINTHM1} and \ref{MAINTHM2} are given in Sections \ref{FOURIERDECSEC}--\ref{BKSEC};
the basic approach is to use Fourier decomposition with respect to the torus fiber variable,
just as in \cite{SASL};
however there are several new difficulties that have to be tackled.
In particular, the $\oGamma'$-orbits in $\Z^{2k}$,
which are used to partition the Fourier decomposition,
are more complicated for $k\geq2$ than for $k=1$:
There are two types of orbits, which we call ``A-orbits'' and ``B-orbits'',
where B-orbits only appear for $k\geq2$;
cf.\ Sec.\ \ref{FOURIERDECSEC}.
Establishing cancelation in the contribution from the B-orbits
requires a novel %
treatment, which we give in    %
Sec.\ \ref{BKSEC}.
The treatment of the A-orbits (cf.\ Sec.~ \ref{A sum}) becomes more delicate for $k\geq2$
than for $k=1$,
and %
this is where we need to require that the test function $f$
decays sufficiently rapidly in the cusp
(cf.\ the parameter ``$a$'' in Theorems~ \ref{MAINTHM1} and \ref{MAINTHM2});
this is not needed for $k=1$.
Other differences versus \cite{SASL}
are that we consider congruence subgroups and not just
$\overline\Gamma=\SL(2,\Z)\ltimes(\Z^2)^{\oplus k}$ itself,
and the fact that the Diophantine conditions are more complicated
in the present paper, 
as they concern vectors in $\R^k$.

As will be seen, %
in the present paper we make crucial use of the assumptions
in Theorems~ \ref{MAINTHM1} and \ref{MAINTHM2} that either $\vecxi_1=\bn$ or $\vecxi_2=\bn$.
It is an interesting problem to seek a more general treatment so as to obtain 
an effective version of Theorem \ref{INEFFECTIVETHM} for general $\vecxi_1,\vecxi_2$. %
We have some preliminary results on this problem and hope to return to it in a later paper.

\vspace{5pt}

We next turn to an application of Theorem \ref{MAINTHM1}:
Following an approach introduced by Marklof in \cite{MarklofpaircorrI}
using theta series,
we will prove an 
effective quantitative Oppenheim type result for the inhomogeneous quadratic form
\begin{align}\label{Qdef}
Q(x_1,x_2,x_3,x_4)=(x_1-\alpha)^2+(x_2-\beta)^2-(x_3-\alpha)^2-(x_4-\beta)^2
\end{align}
for a fixed vector $(\alpha,\beta)\in\R^2$ subject to Diophantine conditions.
Recall that the original Oppenheim conjecture states that for any 
indefinite nondegenerate
homogeneous quadratic form $\tQ$
in $n\geq3$ variables, not proportional to a rational form,
$\tQ(\Z^n)$ is dense in $\R$.
This was proved in celebrated work by Margulis \cite{Margulis89}.
An effective version of this result has more recently been obtained by
Lindenstrauss and Margulis, \cite{LM}.
A quantitative (but non-effective) version of the Oppenheim conjecture 
for forms of signature $(p,q)$ with $p\geq3$ and $q\geq1$
was proved by Eskin, Margulis and Mozes, 
\cite{EMM98}, and extended to forms of signature $(2,2)$
subject to a Diophantine condition in \cite{EMM05}.
Similar quantitative results were later proved also for
inhomogeneous quadratic forms by Margulis and Mohammadi \cite{MM};
in particular the result proved by Marklof \cite{MarklofpaircorrI}
for the form $Q$ in \eqref{Qdef} is a special case of the results in 
\cite{MM};
however the method of proof in \cite{MM}
is different and does not involve theta series.

Effective quantitative results for indefinite 
forms in $n\geq5$ variables 
have been proved by
G\"otze and Margulis \cite{GM}.
However we are not aware of any previous effective quantitative results for forms in $3$ or $4$ variables.

Returning to the form $Q$ in \eqref{Qdef},
for $f\in\C_c(\R^4)$, $g\in\C(\R)\cap \L^1(\R)$ and $T>0$, set
\begin{align}\label{NabfgTdef}
N_{\alpha,\beta}(f,g,T):=\frac1{T^2}\sum_{\vecm\in\Z^4\setminus\Delta}f(T^{-1}\vecm)g(Q(\vecm))
\end{align}
where $\Delta:=\{(\vecm_1,\vecm_1)\col\vecm_1\in\Z^2\}$.
We also set
\begin{align}\label{lambdafdef}
\lambda_f:=\frac12\int_0^\infty\int_0^{2\pi}\int_0^{2\pi}
f\bigl(r\cos\zeta_1,r\sin\zeta_1,r\cos\zeta_2,r\sin\zeta_2\bigr)\,d\zeta_1\,d\zeta_2\,r\,dr.
\end{align}
One verifies easily that
\begin{align*}
\lim_{T\to\infty}\frac1{T^2}\int_{\R^4}f\bigl(T^{-1}\vecx\bigr) g(Q(\vecx))\,d\vecx
=\lambda_f\int_{\R}g(r)\,dr.
\end{align*}

We say that $\vecxi\in\R^k$ is \textit{$\kappa$-Diophantine} if there exists a constant
$c>0$ such that $\|q\vecxi-\vecm\|\geq cq^{-\kappa}$ for all $q\in\Z^+$ and $\vecm\in\Z^k$
(cf.\ \cite[Sec.\ 1.5]{MarklofpaircorrII}\footnote{Note\label{KAPPAfootnote} that our $\kappa$ corresponds to ``$\kappa-1$'' in
\cite[Sec.\ 1.5]{MarklofpaircorrII}. Both of these conventions are common in the literature,
and we made our choice so as to make the statement of Theorem \ref{MAINAPPLTHM} and later results
as simple as possible.}).
We also say that $\vecxi$ is \textit{$[\kappa;c]$-Diophantine} in this case.
The smallest possible value for $\kappa$ is $\kappa=k^{-1}$,
and on the other hand Lebesgue-almost every $\vecxi\in\R^k$ is $(k^{-1}+\ve)$-Diophantine
for every $\ve>0$.
In Section \ref{DIOPHSEC} we will also discuss a different (also standard) Diophantine condition,
which is more directly connected to the decay properties of $\delta_{\beta,\vecxi}(T)$.

In Section \ref{APPLSEC} we prove the following 
effective quantitative Oppenheim result for the form $Q$: %
\begin{thm}\label{MAINAPPLTHM}
There exists an absolute constant $B>0$ such that 
for any $[\kappa;c]$-Diophantine vector $(\alpha,\beta)\in\R^2$ with $|\alpha|,|\beta|\leq1$, %
any $f\in\C_c^1(\R^4)$ with support contained in the unit ball centered at the origin,
any $g\in\C^3(\R)$ with $S_{1,2,3}(g)<\infty$,
and any $T\geq1$,
\begin{align}\label{MAINAPPLTHMres}
\biggl|N_{\alpha,\beta}(f,g,T)-\lambda_f\int_{\R}g(s)\,ds\biggr|
\ll \sum_{j=1}^4\Bigl\|\frac{\partial}{\partial x_j}f\Bigr\|_{\L^\infty}\,S_{1,2,3}(g)\,\kappa \,c^{-\frac1{\kappa}}
\,\delta_{6,(\alpha,\beta)}(T)^{1/(B\kappa)},
\end{align}
where the implied constant is absolute.
\end{thm}
The assumption in Theorem \ref{MAINAPPLTHM} that $\supp(f)$ is contained in the unit ball 
simplifies the statement of the theorem,
but can easily be weakened by an aposteriori scaling argument;
furthermore %
one can remove the assumption that $(\alpha,\beta)\in[-1,1]^2$,
as long as $T$ is large compared to $\|(\alpha,\beta)\|$.
Cf.\ Corollary \ref{MAINAPPLTHMfnoncptsuppCOR} in Sec.\ \ref{MAINAPPLTHMcorproofsec}.

As we will show in Section \ref{MAINAPPLTHMcorproofsec},
by a standard approximation argument, %
Theorem \ref{MAINAPPLTHM} %
implies the following effective counting result.
For real numbers $a<b$ and $T>0$, set
\begin{align}\label{Nalfbetdef}
N_{\alpha,\beta}(a,b,T):=\frac1{T^2}\#\bigl\{\vecx\in\Z^4\setminus\Delta\col \|\vecx\|<T,\: a<Q(\vecx)<b\bigr\}.
\end{align}
(One could also replace
the ball $\{\|\vecx\|<T\}$ in \eqref{Nalfbetdef} by a more general expanding region in $\R^4$;
however in order to keep the presentation simple  we will not elaborate on this.)
\begin{cor}\label{MAINAPPLTHMCOR1}
There exists an absolute constant $B'>0$ such that 
for any $[\kappa;c]$-Diophantine vector $(\alpha,\beta)\in[-1,1]^2$
and any real numbers $a<b$ and $T\geq1$,
\begin{align}\label{MAINAPPLTHMCOR1RES}
\Bigl|N_{\alpha,\beta}(a,b,T)-\tfrac{\pi^2}2(b-a)\Bigr|\ll 
\bigl(1+|a|+|b|\bigr)^3
\kappa c^{-\frac1{\kappa}}\,\delta_{6,(\alpha,\beta)}(T)^{1/(B'\kappa)},
\end{align}
where the implied constant is absolute.
\end{cor}
Note that 
the right hand sides of \eqref{MAINAPPLTHMres} and \eqref{MAINAPPLTHMCOR1RES} tend to
zero as $T\to\infty$ 
(keeping all other data fixed)
whenever 
$1,\alpha,\beta$ are linearly independent over $\Q$
and the vector $(\alpha,\beta)$ is $\kappa$-Diophantine for some $\kappa$.
If $(\alpha,\beta)$ furthermore satisfies a Diophantine condition of the type discussed in
Section \ref{DIOPHSEC}
then we even have a power rate decay with respect to $T$ in \eqref{MAINAPPLTHMres}
and \eqref{MAINAPPLTHMCOR1RES}.
In particular, by a result of Schmidt \cite{wS70} (or \cite{wS67}), we have a power rate 
decay with respect to $T$ whenever 
$\alpha,\beta$ are algebraic numbers
such that $1,\alpha,\beta$ are linearly independent over $\Q$;
cf.\ Remark \ref{SCHMIDT70REM} below.

\begin{remark}
The actual powers for the decay with respect to $T$
which we obtain in Theorem~\ref{MAINAPPLTHM} and Corollary \ref{MAINAPPLTHMCOR1}
are quite small
and depend strongly on the $a$ and $m$ appearing in the $\C_a^m$-norm in 
Theorem~\ref{MAINTHM1} (which, as we remarked above, we have not attempted to optimize).
Cf.\ Lemma~\ref{|v|>} and Remark \ref{|v|>REM} below.
It is an interesting problem to seek the \textit{maximal} power $\eta$
such that the difference in 
\eqref{MAINAPPLTHMres}
decays like $T^{-\eta}$,
for any fixed $(\alpha,\beta)$
subject to an appropriate Diophantine condition
and any sufficiently nice test functions $f$ and $g$.
\end{remark}

\begin{remark}\label{LINDEPrem}
The relation
$\lim_{T\to\infty} N_{\alpha,\beta}(a,b,T)=\frac{\pi^2}2(b-a)$
also holds for $(\alpha,\beta)$ $\kappa$-Diophantine
with $1,\alpha,\beta$ linearly \textit{de}pendent over $\Q$,
except that for certain such pairs $\alpha,\beta$ the definition of $N_{\alpha,\beta}(a,b,T)$
in \eqref{Nalfbetdef} has to be modified by removing one more exceptional subspace besides $\Delta$.
This follows as a %
special case of the (ineffective) result of 
Margulis and Mohammadi \cite[Theorem 1.9]{MM}\footnote{The notion \label{MMDEF1p7footnote}
of $\vecxi\in\R^k$ being ``$\kappa$-Diophantine''
in \cite{MM} %
is different from the one which we have defined;
however it is easy to verify that
if $\vecxi$ is $\kappa$-Diophantine in the sense of 
\cite[Def.\ 1.7]{MM} then $\vecxi$ is $(\kappa-1)$-Diophantine in our sense,
and %
if $\vecxi$ is $\kappa$-Diophantine in our sense then
$\vecxi$ is $k(\kappa+1)$-Diophantine in the sense of \cite[Def.~1.7]{MM}.
One also verifies by a direct computation 
that the form $Q$ in \eqref{Qdef} with $(\alpha,\beta)\notin\Q^2$
admits at most one more exceptional
subspace in the sense of \cite[p.\ 124(bottom)]{MM}
besides $\Delta=\{(\vecm_1,\vecm_1)\col\vecm_1\in\Z^2\}$, and 
such an exceptional subspace
can only occur %
when $1,\alpha,\beta$ are linearly dependent over $\Q$.}.
The reason why Theorem \ref{MAINAPPLTHM} and Cor.\ \ref{MAINAPPLTHMCOR1} fail to 
give the desired limiting result 
in the case when $1,\alpha,\beta$ are linearly dependent over $\Q$ is that 
as a crucial step in the proof, 
Theorem~\ref{MAINTHM1} is applied with $\vecxi_2=\scmatr{\alpha}{\beta}$,
and as we discussed in connection with Theorem \ref{INEFFECTIVETHM}
(cf.\ \eqref{ORBITCONFINED}),
the asymptotic equidistribution therein fails 
when $1,\alpha,\beta$ are $\Q$-linearly dependent.
This situation is discussed in \cite[Appendix A]{MarklofpaircorrI},
and as indicated there, and carried out in some special cases,
it is possible to extend the proof method of \cite{MarklofpaircorrI}
to the case of $\Q$-linear dependence, %
by utilizing equidistribution in the appropriate homogeneous submanifold of $\Gamma\bs G$.
It would be interesting to make this approach effective,
i.e.\ to seek a satisfactory effective version of the statement that
$\lim_{T\to\infty} N_{\alpha,\beta}(a,b,T)=\frac{\pi^2}2(b-a)$
for \textit{all} $\kappa$-Diophantine vectors $(\alpha,\beta)\in\R^2$.

It should be noted that \textit{some}
Diophantine condition on $(\alpha,\beta)$ is certainly necessary
in order for 
$\lim_{T\to\infty} N_{\alpha,\beta}(a,b,T)=\frac{\pi^2}2(b-a)$
to hold;
cf.\ \cite[Thm.\ 1.13 and Sec.\ 9]{MarklofpaircorrI}.
By contrast,
the \textit{non-}quantitative result that $Q(\Z^4)$ is dense in $\R$,
and in fact even
$\liminf_{T\to\infty} N_{\alpha,\beta}(a,b,T)\geq\frac{\pi^2}2(b-a)$
for all $a<b$,
is known to hold for \textit{all} irrational vectors $(\alpha,\beta)$,
that is, for all $(\alpha,\beta)\in\R^2\setminus\Q^2$.
This is a special case of \cite[Thm.\ 1.4]{MM}.
\end{remark}

\vspace{5pt}

Finally let us note that Theorem \ref{MAINAPPLTHM} implies an effective version
of the main theorem of \cite{MarklofpaircorrI},
which says that under explicit Diophantine conditions on $(\alpha,\beta)\in\R^2$,
the local two-point correlations of the sequence given by the values of 
$Q_1(m,n)=(m-\alpha)^2+(n-\beta)^2$, with $(m,n)\in\Z^2$,
are those of a Poisson process
--- a result which %
partly confirms a conjecture of Berry and Tabor \cite{BT}
on quantized integrable systems.
For fixed $(\alpha,\beta)\in\R^2$,
denote by $0\leq\lambda_1\leq\lambda_2\leq\cdots\to\infty$
the sequence of %
values of $Q_1(m,n)$ for $(m,n)\in\Z^2$, counted with multiplicity.
One easily verifies %
that the asymptotic density of this sequence is $\pi$:
\begin{align*}
\#\{j\col\lambda_j\leq\Lambda\}
=\#\{(m,n)\in\Z^2\col (m-\alpha)^2+(n-\beta)^2<\Lambda\}
\sim\pi \Lambda\qquad\text{as }\:\Lambda\to\infty.
\end{align*}
For a given interval $[a,b]\subset\R$,
the \textit{pair correlation function} is then defined as
\begin{align}\label{R2def}
R_2[a,b](\Lambda)=\frac1{\pi\Lambda}\#\bigl\{(j,k)\in(\Z^+)^2\col
j\neq k,\: \lambda_j,\lambda_k<\Lambda;\: \lambda_j-\lambda_k\in(a,b)\bigr\}.
\end{align}
In Section \ref{MAINAPPLTHMcorproofsec} we will prove:
\begin{cor}\label{BTcor}
There exists an absolute constant $B''>0$ such that
for any $[\kappa;c]$-Diophantine vector $(\alpha,\beta)\in\R^2$,
and any real numbers $a<b$ and $\Lambda\geq1$,
\begin{align}\label{BTcorres}
\bigl|R_2[a,b](\Lambda)-\pi(b-a)\bigr|\ll (1+|a|+|b|)^3\kappa c^{-\frac1{\kappa}}
\delta_{6,(\alpha,\beta)}(T)^{1/(B''\kappa)},
\end{align}
where the implied constant is absolute.
\end{cor}
This corollary %
indeed gives an effective version of Marklof \cite[Theorem 1.8]{MarklofpaircorrI},
as well as of 
\cite[Theorem 1.6]{MarklofpaircorrII} in the case $k=2$,
since the right hand side of \eqref{BTcorres} tends to zero as $T\to\infty$
for any fixed $\kappa$-Diophantine vector $(\alpha,\beta)$ (any $\kappa$) %
such that $1,\alpha,\beta$ are linearly independent over $\Q$.

The main result in Marklof \cite[Theorem 1.6]{MarklofpaircorrII}
generalizes \cite[Theorem 1.8]{MarklofpaircorrI}
to the case of the local pair correlation density of the sequence $\|\vecm-\vecalf\|^k$ ($\vecm\in\Z^k$)
for any $k\geq2$
(and also for $k=2$ it is a stronger result,
since the Diophantine condition imposed on the vector $(\alpha,\beta)\in\R^2$ is weaker).
Unfortunately it seems that
Theorem \ref{MAINTHM1} above cannot be used to prove an effective version of
this more general result when $k\geq3$.
The reason is that the key equidistribution result required,
\cite[Thm.\ 5.1]{MarklofpaircorrII},
concerns the integral
\begin{align}\label{MpcIIintegral}
y^{\sigma}\int_{\R} f \Bigl(\Gamma\Bigl(1_2,\Bigl(\begin{matrix} \bn \\ \vecxi_2 \end{matrix}\Bigr)\Bigr)u(x)a(y)\Bigr)
h(y^{\sigma}x)\,dx
\end{align}
with $\sigma=\frac k2-1$,
that is, the integral which appears in Theorem~\ref{MAINTHM1} but with the function $h$ replaced by
$x\mapsto y^{\sigma} h(y^\sigma x)$.
With this choice,
the %
$S_{\infty,2+\ve,2}$-norm in the right hand side of \eqref{MAINTHM1RES}
grows rapidly as $y\to0$, making the bound useless. 
This failure may at first seem surprising, since the factor $y^\sigma$ means, when $\sigma>0$,
that we are considering a unipotent orbit expanding at a \textit{faster} rate than for $\sigma=0$,
so the result can be expected to be easier (or at least not more difficult) to prove.
However, there is a genuine difference between
$x$ near zero and $x$ far from zero in the integrand in \eqref{MpcIIintegral};
for example, for any %
$u(n)\in\Gamma$,
using $u(n)\bigl(1_2,\scmatr{\bn}{\vecxi_2}\bigr)=\bigl(1_2,\scmatr{n\vecxi_2}{\vecxi_2}\bigr)u(n)$
 we have
\begin{align*}
f \Bigl(\Gamma\Bigl(1_2,\Bigl(\begin{matrix} \bn \\ \vecxi_2 \end{matrix}\Bigr)\Bigr)u(x)a(y)\Bigr)
=f \Bigl(\Gamma \Bigl(1_2,\Bigl(\begin{matrix} n\vecxi_2 \\ \vecxi_2 \end{matrix}\Bigr)\Bigr)u(x+n)a(y)\Bigr).
\end{align*}
It is clear from this that if one would solve the aforementioned
problem of proving an effective version of 
Theorem \ref{INEFFECTIVETHM} in the general case with both $\vecxi_1,\vecxi_2$ allowed to be  non-zero,
this can be expected to also lead to an effective version of 
\cite[Thm.\ 5.1]{MarklofpaircorrII}, and so, with further work,
should also lead to an effective version of \cite[Theorem 1.6]{MarklofpaircorrII} for general $k\geq2$.

\vspace{5pt}

\textit{Acknowledgments.}
We would like to thank 
Sanju Velani for helpful discussions regarding Lemma \ref{DIOFHDIMLEM}.

\section{Some notation}
\label{NOTATION_SEC}

We use the standard notation $A=O(B)$ or $A\ll B$ meaning $|A|\leq CB$ for some constant $C>0$.
We shall also use $A\asymp B$ as a substitute for $A\ll B\ll A$.
The implicit constant $C$ will always be allowed to depend on $k$ and $N$ without any explicit mention.
If we wish to indicate that $C$ also depends on some other quantities $f,g,h$, we will use the notation
$A\ll_{f,g,h} B$ or $A=O_{f,g,h}(B)$.

Recall from Section \ref{intro} that 
$G'=\SL(2,\R)$ and $G=G'\ltimes\R^{2k}$.
Let $\ig$ be the Lie algebra of $G$;
it may be naturally identified with the space $\lsl(2,\R)\oplus\R^{2k}$,
with Lie bracket %
$[(X,\vecv),(Y,\vecw)]=(XY-YX,X\vecw-Y\vecv)$
(cf., e.g., \cite[Prop.\ 1.124]{Knapp}).
Using this notation, we fix the following basis of $\ig$:
\begin{align}\label{FIXEDBASIS}
&X_1=\left(\matr 0100,\bn\right);\qquad
X_2=\left(\matr 0010,\bn\right);\qquad
X_3=\left(\matr 100{-1},\bn\right);\qquad
\\\notag
& X_{3+\ell}=\left(0,\cmatr{\vece_\ell}{\bn}\right);\qquad
X_{3+k+\ell}=\left(0,\cmatr{\bn}{\vece_\ell}\right)\qquad(\ell=1,\ldots,k).
\end{align}
Here $\vece_1=(1,0,\ldots,0)$, $\vece_2=(0,1,0,\ldots,0)$, $\ldots$, $\vece_k=(0,\ldots,0,1)$ are the
standard basis vectors of $\R^k$.

We set 
\begin{align*}
\overline\Gamma'=\SL(2,\Z)\quad\text{and}\quad\Gamma'=\Gamma(N),
\quad\text{so that }\quad
\overline\Gamma=\overline\Gamma'\ltimes\Z^{2k}\quad\text{and}\quad\Gamma=\Gamma'\ltimes\Z^{2k}
\end{align*}
(cf.\ Section \ref{intro}).
Given a function $f$ on $X=\GaG$, we will often view $f$ as a function on $G$ through $f(g)=f(\Gamma g)$,
and we will write $f(M,\vecv)$ in place of $f((M,\vecv))$, for $(M,\vecv)\in G$. %
Furthermore, given any $R\in\overline\Gamma'$, we set
\begin{align}\label{FRDEF}
f_R(M,\vecv):=f(R^{-1}(M,\vecv))=f( R^{-1}M, R^{-1}\vecv).
\end{align}
Since $\Gamma'$ is normal in $\overline\Gamma'$, $f_R$ is also left $\Gamma$-invariant, i.e.\ $f_R$ can be
viewed as a function on $X$.
Note also that $\|f_R\|_{\C_a^m}=\|f\|_{\C_a^m}$ for all $m\geq0$, $a\in\R$.

\section{Linear form Diophantine conditions}\label{DIOPHSEC}

Given real numbers $\kappa\geq k$ and $\alpha\geq1$, we say that a vector $\vecxi\in\R^k$ is
\textit{$\kappa$-LFD}
(short for $\kappa$-linear form Diophantine)
if there is a constant $c>0$ such that
\begin{align}\label{DIOPH0}
\langle\vecr\vecxi\rangle\geq c\|\vecr\|^{-\kappa}\qquad
\text{for all }\:\vecr\in\Z^k\setminus\{\bn\},
\end{align}
and we say that $\vecxi$ is \textit{$(\kappa,\alpha)$-LFD}
if there is a constant $c>0$ such that 
\begin{align}\label{DIOPH1}
\langle j\vecr\vecxi\rangle\geq cj^{-\alpha}\|\vecr\|^{-\kappa}\qquad
\text{for all }\: j\in\Z^+,\:\vecr\in\Z^k\setminus\{\bn\}.
\end{align}
Recall here that for $x\in\R$, $\langle x\rangle$ denotes the distance to the nearest integer,
and $\vecr\vecxi$ is the scalar product, $\vecr\vecxi=r_1\xi_1+\ldots+r_k\xi_k$.
The condition in \eqref{DIOPH0} is very standard in the Diophantine approximation literature;
however we are not aware of any discussion %
of the more general condition in 
\eqref{DIOPH1}.
When \eqref{DIOPH0} holds, we will say that $\vecxi$ is $[\kappa;c]$-LFD,  %
and similarly when \eqref{DIOPH1} holds, we will say that $\vecxi$ is $[(\kappa,\alpha);c]$-LFD.
Note that being $[\kappa;c]$-LFD is equivalent to being $[(\kappa,\alpha);c]$-LFD for any
$\alpha\geq\kappa$.
Hence the notion of being $[(\kappa,\alpha);c]$-LFD is mainly relevant when $1\leq\alpha<\kappa$,
and in this case the condition \eqref{DIOPH1} is equivalent
to the same condition with $\vecr$ restricted to being a \textit{primitive} vector in $\Z^k$
(viz., a vector with $\gcd(r_1,\ldots,r_k)=1$).

Note that if $\vecxi$ is $(\kappa,\alpha)$-LFD, then $\vecxi$ is also $\kappa$-LFD
and furthermore 
each co-ordinate $\xi_\ell$ of $\vecxi$ is an 
$\alpha$-LFD ($\Leftrightarrow$ $\alpha$-Diophantine) real number
(apply \eqref{DIOPH1} with $\vecr=\vece_{\ell}$).
Hence 
if either $\kappa=k$ or $\alpha=1$, then the set
of $(\kappa,\alpha)$-LFD $\vecxi\in\R^k$ has Lebesgue measure zero
\cite{Khintchine26}; \cite{Khintchine25,Perron}.
On the other hand if both $\kappa>k$ and $\alpha>1$, then the \textit{complement} of that set has 
Lebesgue measure zero, and moreover, it has Hausdorff dimension strictly less than $k$:
\begin{lem}\label{DIOFHDIMLEM}
If $\kappa>k$ and $\alpha>1$ then the Hausdorff dimension of the set of all 
$\vecxi\in\R^k$ which are not $(\kappa,\alpha)$-LFD
equals $k-1+\max\bigl(\frac{k+1}{\kappa+1},\frac2{\alpha+1}\bigr)$.
\end{lem}
\begin{proof}
The set in the statement of the Lemma contains the set of all $\vecxi\in\R^k$ which are not
$\kappa$-LFD, and the latter set has (Hausdorff) dimension
$k-1+\frac{k+1}{\kappa+1}$,
cf.\ Bovey and Dodson, \cite{BD}.
Furthermore, taking $\vecr=\vece_1$ in \eqref{DIOPH1} we see that 
the set in the statement of the lemma contains the set of all $\vecxi\in\R^k$
for which $\xi_1$ is not $\alpha$-LFD, 
and this set has dimension $k-1+\frac2{\alpha+1}$.

Hence it remains to prove that the dimension in the statement of the lemma is bounded \textit{above}
by $k-1+\max\bigl(\frac{k+1}{\kappa+1},\frac2{\alpha+1}\bigr)$.
It suffices to consider $\vecxi\in[0,1)^k$.
Set
\begin{align*}
\Delta_{j,\vecr,m}=\bigl\{\vecxi\in[0,1)^k\col|j\vecr\vecxi-m|<j^{-\alpha}\|\vecr\|^{-\kappa}\bigr\}.
\end{align*}
Then every non-$(\kappa,\alpha)$-LFD $\vecxi$ in $[0,1)^k$ belongs to
$\Delta_{j,\vecr,m}$ for infinitely many $(j,\vecr,m)\in\Z^+\times(\Z^k\setminus\{\bn\})\times\Z$.
Note also that $\Delta_{j,\vecr,m}=\emptyset$ unless $|m|\ll j\|\vecr\|$,
and for any $(j,\vecr,m)\in\Z^+\times(\Z^k\setminus\{\bn\})\times\Z$,
if we set $\ell=\ell_{j,\vecr}=j^{-\alpha-1}\|\vecr\|^{-\kappa-1}$
then the set $\Delta_{j,\vecr,m}$ can be covered by
$\ll \ell^{1-k}$ open hypercubes each having sides of length $\ll\ell$,
with the normal to each face being parallel to a co-ordinate axis.
If $s>k-1+\max\bigl(\frac{k+1}{\kappa+1},\frac2{\alpha+1}\bigr)$ then the total $s$-volume of the family
of hypercubes obtained as $(j,\vecr,m)$ runs through $\Z^+\times(\Z^k\setminus\{\bn\})\times\Z$
(subject to $\Delta_{j,\vecr,m}\neq\emptyset$) is
\begin{align*}
\ll\sum_{j=1}^\infty\sum_{\vecr\in\Z^k\setminus\{\bn\}}
j\|\vecr\|\cdot \bigl(j^{-\alpha-1}\|\vecr\|^{-\kappa-1}\bigr)^{1-k+s}<\infty.
\end{align*}
Note also that for any $\delta>0$ there are only a finite number of 
non-empty sets $\Delta_{j,\vecr,m}$ satisfying $\ell_{j,\vecr}\geq\delta$;
hence every non-$(\kappa,\alpha)$-LFD $\vecxi\in [0,1)^k$
is contained in the union of hypercubes in the above family restricted by $\ell_{j,\vecr}<\delta$.
It follows that for every $s>k-1+\max\bigl(\frac{k+1}{\kappa+1},\frac2{\alpha+1}\bigr)$,
the $s$-dimensional outer Hausdorff measure of the set of all
non-$(\kappa,\alpha)$-LFD $\vecxi$ in $[0,1)^k$ equals zero.
This completes the proof.
\end{proof}

We will need the following auxiliary result.
\begin{lem}\label{MXIBOUNDLEM}
Let $\eta\in\R$, $c>0$, $\kappa\geq1$, and assume that $\langle j\eta\rangle\geq cj^{-\kappa}$ for all
$j\in\Z^+$.
Then 
\begin{align}\label{MXIBOUNDLEMRES1}
\sum_{j=1}^\infty \frac1{j^2+Tj\langle j\eta\rangle}\ll(cT)^{-\frac2{1+\kappa}}\log^2(2+T)
\qquad\text{for all }\: T>0.
\end{align}
\end{lem}
(The bound is essentially optimal. Indeed,
if $\langle j\eta\rangle\leq cj^{-\kappa}$ holds for some $j$ then for $T=j^{1+\kappa}/c$, already the term
$\frac1{j^2+Tj\langle j\eta\rangle}$ is bounded below by $\frac12(cT)^{-\frac2{1+\kappa}}$.)
\begin{proof}
We assume $cT>1$ since otherwise the bound is trivial.
Note that the assumptions of the lemma imply that $\eta$ is irrational, and $0<c\leq\langle\eta\rangle\leq\frac12$.
Thus %
$T>2$.

Let $p_k/q_k$ %
be the $k$th convergent of the
(simple) continued fraction expansion of $\eta$ (cf., e.g., \cite[Ch.\ X]{HW}; in particular $1=q_0\leq q_1<q_2<\cdots$).
For any $\ell\geq1$ we have
\begin{align*}
\sum_{1\leq j\leq q_\ell/2}\frac{1}{j\langle j\eta\rangle}
=\sum_{k=1}^\ell\sum_{q_{k-1}/2<j\leq q_k/2}\frac1{j\langle j\eta\rangle}
\ll \sum_{k=1}^\ell q_{k-1}^{-1}\sum_{1\leq j\leq q_k/2}
\frac1{\langle j\eta\rangle}
\ll \sum_{k=1}^\ell \frac{q_k\log q_k}{q_{k-1}},
\end{align*}
where the last bound follows from \cite[Lemma 4.8]{Nathanson},
since $|\eta-\frac{p_k}{q_k}|<\frac1{q_kq_{k+1}}$ \cite[Thm.\ 171]{HW}.
But for every $k\geq1$ we have 
$cq_{k-1}^{-{\kappa}}\leq\langle q_{k-1}\eta\rangle<q_k^{-1}$, i.e.\ $q_k<c^{-1}q_{k-1}^{{\kappa}}$;
hence we get
\begin{align}\label{MXIBOUNDLEMPF4}
\sum_{1\leq j\leq q_\ell/2}\frac{1}{j\langle j\eta\rangle}
\ll c^{-1}(\log q_\ell)\sum_{k=1}^\ell q_{k-1}^{{\kappa}-1}
\ll c^{-1}(\log q_\ell)^2q_{\ell-1}^{\kappa-1},
\end{align}
where we used the fact that $q_\ell$ is bounded below by the $\ell$th Fibonacci number.

Next note that for any $\ell\geq1$ and $h\geq1$,
by \cite[Lemma 4.9]{Nathanson},
\begin{align}\label{MXIBOUNDLEMPF1}
\sum_{hq_{\ell}+1\leq j\leq (h+1)q_{\ell}}\frac1{j^2+Tj\langle j\eta\rangle}
\leq \frac1{Thq_{\ell}}
\sum_{r=1}^{q_{\ell}}\min\Bigl(\frac T{hq_{\ell}},\frac1{\langle (hq_{\ell}+r)\eta\rangle}\Bigr)
\ll \frac1{(hq_\ell)^2}+\frac{\log q_{\ell}}{Th}.
\end{align}
Similarly 
\begin{align}\label{MXIBOUNDLEMPF2}
\sum_{q_\ell/2<j\leq q_\ell}
\frac1{j^2+Tj\langle j\eta\rangle}
\ll \frac1{Tq_\ell}\sum_{r=1}^{q_{\ell}}\min\Bigl(\frac T{q_{\ell}},\frac1{\langle r\eta\rangle}\Bigr)
\ll \frac1{q_{\ell}^2}+\frac{\log q_{\ell}}{T}.
\end{align}
Adding \eqref{MXIBOUNDLEMPF2} and \eqref{MXIBOUNDLEMPF1} for all $h\leq T/q_\ell$ we obtain
\begin{align}\label{MXIBOUNDLEMPF3}
\sum_{q_\ell/2<j\leq T}\frac1{j^2+Tj\langle j\eta\rangle}
\ll \frac1{q_\ell^2}+\frac{\log q_\ell \log(1+\frac T{q_\ell})}T.
\end{align}
Now choose $\ell\geq1$ so that $q_{\ell-1}\leq(cT)^{\frac1{1+\kappa}}<q_\ell$.
Then
$q_\ell<c^{-1}q_{\ell-1}^{{\kappa}}\leq (cT)^{-\frac1{\kappa+1}}T<T$.
Now \eqref{MXIBOUNDLEMRES1} follows from
\eqref{MXIBOUNDLEMPF4}, \eqref{MXIBOUNDLEMPF3}
and the bound $\sum_{j>T}j^{-2}\ll T^{-1}$.
\end{proof}

We now give a result on the rate of decay of the majorant function $\delta_{\beta,\vecxi}(T)$ 
(cf.\ \eqref{MMXIDEF}),
assuming that $\vecxi$ is of an appropriate LFD type.
In fact we consider the following slightly simpler majorant:
\begin{align}\label{WMdef}
\wdelta_{\beta,\vecxi}(T)=\sum_{\vecr\in\Z^{k}\setminus\{\bn\}}\|\vecr\|^{-\beta}\sum_{j=1}^\infty
\frac1{j^2+Tj\langle j\vecr\vecxi\rangle}.
\end{align}
Note that $\delta_{\beta,\vecxi}(T)$ and $\wdelta_{\beta,\vecxi}(T)$ decay with very similar rates,
since 
\begin{align}\label{wMMcomp}
\wdelta_{\beta,\vecxi}(T)\leq \delta_{\beta,\vecxi}(T)\leq (2\log T)\wdelta_{\beta,\vecxi}(T),\qquad \forall T\geq e.
\end{align}
\begin{lem}\label{MXIBOUNDLEM2}
For any $\kappa\geq k$, $\alpha\geq1$ and $\beta>k+\frac{2\kappa}{1+\alpha}$,
if $\vecxi\in\R^k$ is $[(\kappa,\alpha);c]$-LFD then
\begin{align*}
\delta_{\beta,\vecxi}(T)\ll_{\beta,\kappa,\alpha} 
(cT)^{-\frac2{1+\alpha}}\log^2(2+T)
\qquad
\text{for all }\: T>0.
\end{align*}
\end{lem}
\begin{proof}
Using Lemma \ref{MXIBOUNDLEM} and the assumption that 
$\vecxi$ is $[(\kappa,\alpha);c]$-LFD, we have
\begin{align*}
\sum_{j=1}^\infty \frac1{j^2+Tj\langle j\vecr\vecxi\rangle}\ll
(c\|\vecr\|^{-\kappa}T)^{-\frac2{1+\alpha}}\log^2(2+T),
\qquad\text{for each }\:\vecr\in\Z^k\setminus\{\bn\}.
\end{align*}
Multiplying by $\|\vecr\|^{-\beta}$ and adding over all $\vecr\in\Z^k\setminus\{\bn\}$,
we obtain the stated bound.
\end{proof}

\begin{remark}\label{deltadecayREM}
A standard argument also shows that 
given any $\beta>k$ and $\ve>0$,
the bound $\delta_{\beta,\vecxi}(T)\ll T^{\ve-1}$ as $T\to\infty$ holds for \textit{Lebesgue almost all} $\vecxi\in\R^k$.
We here give an outline of the proof: One verifies that
for $T$ large, $\int_{[0,1]^k}\wdelta_{\beta,\vecxi}(T)\,d\vecxi\ll T^{\frac13{\ve}-1}$,
and hence the set of $\vecxi\in[0,1]^k$ satisfying
$\wdelta_{\beta,\vecxi}(T)\geq T^{\frac23\ve-1}$
has Lebesgue measure $\ll T^{-\frac13\ve}$.
The sum of these measures over $T=2^1,2^2,2^3,\ldots$ is finite,
and so, by Borel-Cantelli,
for almost every $\vecxi\in[0,1]^k$ there is some $M\in\Z^+$ such that
$\wdelta_{\beta,\vecxi}(2^m)<(2^m)^{\frac23\ve-1}$ for all integers $m\geq M$,
and thus
$\wdelta_{\beta,\vecxi}(T)<2T^{\frac23\ve-1}$ for all (real) $T\geq 2^M$.
The desired claim then follows using
\eqref{wMMcomp} and the fact that $\delta_{\beta,\vecxi}(T)$ is invariant under
$\vecxi\mapsto\vecxi+\vecm$, $\vecm\in\Z^k$.
\end{remark}

\begin{remark}\label{SCHMIDT70REM}
By Schmidt, \cite{wS70},
if $\xi_1,\ldots,\xi_k$ are (real) \textit{algebraic} numbers such that
$1,\xi_1,\ldots,\xi_k$ are linearly independent over $\Q$,
then $\vecxi$ is $\kappa$-LFD (and thus $[\kappa,\kappa]$-LFD) for every $\kappa>k$.
Hence for such a $\vecxi$,
Lemma \ref{MXIBOUNDLEM2} implies that
for any $\beta>k\bigl(1+\frac2{1+k}\bigr)$ and $\ve>0$
we have
$\delta_{\beta,\vecxi}(T)\ll_{\vecxi,\beta,\ve} T^{\ve-\frac2{k+1}}$ for all $T>0$.
In connection with Theorem \ref{MAINAPPLTHM}
it should be noted that any such $\vecxi$
is also $\kappa$-Diophantine for every $\kappa>k^{-1}$;
again cf.\ \cite{wS70}.
\end{remark}

\section{Fourier decomposition with respect to the torus variable}
\label{FOURIERDECSEC}

We now start with the proof of Theorems \ref{MAINTHM1} and \ref{MAINTHM2}.
In this section, which generalizes \cite[Sec.\ 4]{SASL}, we consider the Fourier decomposition of a given
test function on $X$ with respect to the torus variable,
and prove bounds on the resulting Fourier coefficients.
Some parts of our discussion is a close mimic of \cite[Sec.\ 4]{SASL},
but there are also some new aspects that have to be considered; 
cf.\ in particular all of Section \ref{BKBOUNDS} below.

To start with, we consider an arbitrary function $f\in\C(\Z^{2k}\backslash G)$,
where $\Z^{2k}$ is viewed as a subgroup of $G$ through $\vecn\mapsto(1_2,\vecn)$.
We view $f$ as a function on $G$ by composing with the projection $G\mapsto\Z^{2k}\backslash G$.
Then $f(M,\vecxi)=f((1_2,\vecn)(M,\vecxi))=f(M,\vecxi+\vecn)$ for all $\vecn\in\Z^{2k}$,
which means that for any fixed $M\in G'$, %
$\vecxi\mapsto f(M,\vecxi)$ is a function on the torus $\T^{2k}=\Z^{2k}\backslash\R^{2k}$.
We write $\widehat f(M,\vecm)$ for the Fourier coefficients in the torus variable;
\begin{align}\label{WHFDEF}
\widehat f(M,\vecm)=\int_{\Z^{2k}\backslash\R^{2k}} f(M,\vecxi)e(-\vecm\vecxi)\,d\vecxi,
\qquad M\in G',\:\vecm\in\Z^{2k}.
\end{align}
Here $d\vecxi$ denotes Lebesgue measure on $\R^{2k}$.
Thus for $f\in\C^2(\Z^{2k}\backslash G)$ %
we have
\begin{align}\label{FOURIERSERIES}
f(M,\vecxi)=\sum_{\vecm\in\Z^{2k}}\widehat f(M,\vecm)e(\vecm\vecxi),
\end{align}
with absolute convergence uniformly\footnote{For any fixed ordering of $\Z^{2k}$.}
over $(M,\vecxi)$ in any compact subset of $G$.
(Indeed, the function $\vecxi\mapsto f(M,\vecxi)$ is in $\C^2(\T^{2k})$,
with $\|f(M,\cdot)\|_{\C^2(\T^{2k})}$ depending continuously on $M\in G'$.)

If $f$ is also invariant under some $T\in\overline\Gamma'=\SL(2,\Z)$,
this leads to a corresponding invariance relation for $\widehat f(M,\vecm)$:
\begin{lem}\label{BASICFOURIERLEM}
For any $T\in\overline\Gamma'$,
if $f\in\C(\Z^{2k}\backslash G)$ is left $T$-invariant, then 
\begin{align}\label{BASICFOURIERLEMRES1}
\wh f(TM,\vecm)=\wh f(M,\trans T\vecm),\qquad\forall M\in G',\:\vecm\in\Z^{2k},
\end{align}
where $\trans T$ is the transpose of $T$.
\end{lem}
\begin{proof}
We have
\begin{align*}
\widehat f(TM,\vecm)
& =\int_{\T^{2k}} f(TM,\vecxi)e(-\vecm\vecxi)\,d\vecxi
=\int_{\T^{2k}} f(TM,T\vecxi)e(-\vecm(T\vecxi))\,d\vecxi
\\
&=\int_{\T^{2k}} f(T(M,\vecxi))e(-\vecm(T\vecxi))\,d\vecxi
=\int_{\T^{2k}} f(M,\vecxi)e(-\vecm(T\vecxi))\,d\vecxi,
\end{align*}
where in the second equality we used the fact that $\vecxi\mapsto T\vecxi$ is a diffeomorphism of
$\T^{2k}$ preserving $d\vecxi$,
and in the last equality we used the fact that $f$ is left $T$-invariant.
Using $\vecm(T\vecxi)=(\trans T\vecm)\vecxi$ we obtain \eqref{BASICFOURIERLEMRES1}.
\end{proof}

Because of Lemma \ref{BASICFOURIERLEM},
if $f\in\C^2(\overline\Gamma\backslash G)$, then it is convenient to group the terms in
\eqref{FOURIERSERIES} together according to the orbits for the action of $\overline\Gamma'$ on $\Z^{2k}$.
We call an orbit for this action an \textit{A-orbit} if it contains some element
of the form $\cmatr\bn\vecr$, where $\vecr\in\Z^k\setminus\{\bn\}$.
Every other non-zero orbit is called a \textit{B-orbit}.
\begin{lem}\label{GOODBPROPLEM}
Every B-orbit contains an element $\veceta=\cmatr\vecq\vecr$ 
($\vecq=\trans(q_1,\ldots,q_k)$, $\vecr=\trans(r_1,\ldots,r_k)$)
with the property that there are some $1\leq\ell_1<\ell_2\leq k$ such that
$r_j=0$ for all $j<\ell_1$, $q_j=0$ for all $j<\ell_2$,
and $r_{\ell_1}>0$, $0\leq r_{\ell_2}<|q_{\ell_2}|$.
\end{lem}
\begin{proof}
Let $\veceta=\cmatr\vecq\vecr$ be an element in a B-orbit.
Then $\veceta\neq\bn$, and we may take $\ell_1$ to be the smallest index for which
$\Bigl(\begin{matrix}{q_{\ell_1}}\\{r_{\ell_1}}\end{matrix}\Bigr)\neq
\Bigl(\begin{matrix}0\\0\end{matrix}\Bigr)$. %
After replacing $\veceta$ by $T\veceta$ for an appropriate $T\in\overline\Gamma'$
we can ensure that $q_{\ell_1}=0$ and $r_{\ell_1}>0$, while clearly still $q_j=r_j=0$ for all
$j<\ell_1$.
Now since $\veceta$ is not in an A-orbit we cannot have $q_j=0$ for all $j$, and we take
$\ell_2>\ell_1$ to be the smallest index for which $q_{\ell_2}\neq0$.
Finally by replacing $\veceta$ by $\matr10x1\veceta$ for an appropriate $x\in\Z$ we can make
$0\leq r_{\ell_2}<|q_{\ell_2}|$ hold, while $q_j$ and $r_j$ for $j<\ell_2$ remain unchanged.
\end{proof}

Let us fix, once and for all, a set of representatives $A_k,B_k\subset\Z^{2k}$ such that
$A_k$ contains exactly one element from each A-orbit and $B_k$ contains exactly one element from
each B-orbit,
and furthermore each $\veceta\in A_k$ is of the form $\veceta=\cmatr\bn\vecr$
and each $\veceta\in B_k$ has the property described in Lemma \ref{GOODBPROPLEM}.

\begin{lem}\label{STABLEM}
The stabilizer  in $\overline\Gamma'$ of any $\veceta\in A_k$ equals 
$\Bigl\{\matr 10n1\col n\in\Z\Bigr\}$.
The stabilizer in $\overline\Gamma'$ of any $\veceta\in B_k$ is trivial.
\end{lem}
\begin{proof}
Immediate verification.
\end{proof}

The lemma implies that we can decompose $\Z^{2k}$ as a disjoint union of singleton sets as follows:
\begin{align}\label{Z2KDEC}
\Z^{2k}=\{\bn\}\:\:\bigsqcup\:\:\biggl(\bigsqcup_{\veceta\in
A_k}\bigsqcup_{T\in\overline\Gamma'_\infty\backslash\overline\Gamma'}
\{\trans T\veceta\}\biggr)\:\:\bigsqcup\:\:\
\biggl(\bigsqcup_{\veceta\in B_k}\bigsqcup_{T\in\overline\Gamma'} \{\trans T\veceta\}\biggr),
\end{align}
where $\overline\Gamma'_\infty\backslash\overline\Gamma'$ denotes any set of representatives for the right cosets inside
$\overline\Gamma'$ of the subgroup
\begin{align}\label{GAMMAinfdef}
\overline\Gamma'_\infty:=\left\{\matr 1n01\col n\in\Z\right\}.
\end{align}
Grouping together the terms in \eqref{FOURIERSERIES} according to \eqref{Z2KDEC},
and then applying Lemma \ref{BASICFOURIERLEM}, we get,
for any $f\in\C^2(\overline\Gamma\backslash G)$:
\begin{align}\label{BASICFOURIER_SPEC}
f(M,\vecxi)=\widehat f(M,\bn)+\sum_{\veceta\in A_k}
\sum_{T\in\overline\Gamma_\infty'\backslash\overline\Gamma'}
\widehat f(TM,\veceta)e((\trans T\veceta)\vecxi)
+\sum_{\veceta\in B_k}\sum_{T\in \overline\Gamma'} 
\widehat f(TM,\veceta)e((\trans T\veceta)\vecxi).
\end{align}
If $k=1$ then $B_k=\emptyset$ and \eqref{BASICFOURIER_SPEC} can be seen to agree with \cite[Lemma 4.1]{SASL}.
However $B_k$ is easily seen to be nonempty for every $k\geq2$.

We now wish to give a similar decomposition of a general function $f\in\C^2(X)$.
Recall that $X=\GaG$ and $\Gamma=\Gamma'\ltimes\Z^{2k}$
with $\Gamma'=\Gamma(N)$, a normal subgroup of $\overline\Gamma'=\SL(2,\Z)$.
For any subgroup $H$ of $G'$ and any subset $A\subset G'$ satisfying $HA=A$,
we denote by $H\backslash A$ a set of representatives for the distinct cosets $Ha$ ($a\in A$).
We also write 
$\overline\Gamma'_\infty \backslash\overline\Gamma'/\Gamma'$ for a set of representatives for the
double cosets of the form $\overline\Gamma'_\infty R\Gamma'$ with $R\in\overline\Gamma'$.
Let
\begin{align*}
\Gamma'_\infty:=\Gamma'\cap\overline\Gamma'_\infty=\left\{\matr 1{Nn}01\col n\in\Z\right\}.
\end{align*}
One then verifies that
$\bigsqcup_{R\in\overline\Gamma'_\infty \backslash\overline\Gamma'/\Gamma'} 
\bigsqcup_{T\in\Gamma'_\infty\backslash \Gamma'R}\{T\}$
is a set of representatives for $\overline\Gamma'_\infty\backslash\overline\Gamma'$.
Hence from \eqref{Z2KDEC} we get
\begin{align}\label{Z2KDEC2}
\Z^{2k}=\{\bn\}\:\:\bigsqcup\:\:\biggl(\bigsqcup_{\veceta\in A_k}
\bigsqcup_{R\in\overline\Gamma'_\infty \backslash\overline\Gamma'/\Gamma'} 
\bigsqcup_{T\in\Gamma'_\infty\backslash \Gamma'R}
\{\trans T\veceta\}\biggr)\:\:\bigsqcup\:\:\
\biggl(\bigsqcup_{\veceta\in B_k}\bigsqcup_{R\in\overline\Gamma'/\Gamma'}\bigsqcup_{T\in
\Gamma'R} 
\{\trans T\veceta\}\biggr).
\end{align}
Using $\Gamma'R=R\Gamma'$ and $\trans(R\gamma)\veceta=\trans\gamma(\trans R\veceta)$ for $\gamma\in\Gamma'$,
this formula is seen to provide a decomposition of $\Z^{2k}$ into orbits for the action of
$\trans\Gamma'=\Gamma'$.
In order to get a convenient corresponding partition of the sum in \eqref{FOURIERSERIES},
recall \eqref{FRDEF}, and note that for any $R\in\overline\Gamma'$, $M\in G'$, $\vecm\in\Z^{2k}$ we have
\begin{align*}
\widehat{f_R}(M,\vecm)
=\widehat f(R^{-1}M,\trans R\vecm).
\end{align*}
This is proved by a computation similar to the proof of Lemma \ref{BASICFOURIERLEM}.
Using  %
Lemma \ref{BASICFOURIERLEM} we get
$\widehat f(M,\trans\gamma\trans R\veceta)
=\widehat{f_R}(R\gamma M,\veceta)$ for all $\gamma\in\Gamma'$,
or in other words:
\begin{align}\label{FRKEYTRANS}
\widehat f(M,\trans T\veceta)=\widehat{f_R}(TM,\veceta),\qquad\forall R\in\overline\Gamma',\:T\in \Gamma'R,\:
M\in G',\:\veceta\in\Z^{2k}.
\end{align}
Now from \eqref{FOURIERSERIES}, \eqref{Z2KDEC2} and \eqref{FRKEYTRANS} we get:
\begin{align}\notag
f(M,\vecxi)=\widehat f(M,\bn)+\sum_{\veceta\in A_k}
\sum_{R\in\overline\Gamma'_\infty \backslash\overline\Gamma'/\Gamma'} \,
\sum_{T\in\Gamma'_\infty\backslash \Gamma'R}
\widehat{f_R}(TM,\veceta)e((\trans T\veceta)\vecxi)
\\\label{BASICFOURIER}
+\sum_{\veceta\in B_k}\sum_{R\in\overline\Gamma'/\Gamma'}\, \sum_{T\in \Gamma'R} 
\widehat{f_R}(TM,\veceta)e((\trans T\veceta)\vecxi).
\end{align}
Note here that for any $\veceta\in A_k$ and $R\in\overline\Gamma'$, the function
$M\mapsto\widehat{f_R}(M,\veceta)$ is left $\Gamma'_\infty$-invariant,
by \eqref{FRKEYTRANS} and Lemma \ref{STABLEM}.
However for $\veceta\in B_k$ there is no such invariance present.

\vspace{10pt}

\subsection{Bounds when $\veceta\in A_k$}

We now give bounds on $\widehat{f}(T,\veceta)$ for $\veceta\in A_k$.
\begin{lem}\label{DECAYTFNFROMCMNORMLEM}
For any $m\geq0$, %
$\alpha\in\R_{\geq0}$, $\vecr\in\Z^k\setminus\{\bn\}$ and $f\in
\C^m_\alpha(X)$, we have
\begin{align}\label{DECAYTFNFROMCMNORMLEMRES}
\left|\wh
f\left(\matr abcd,\cmatr\bn\vecr\right)\right|\ll_{m,\alpha}\|f\|_{\C^m_\alpha}\|\vecr\|^{-m}
(c^2+d^2)^{-\frac m2}\min\bigl(1,(c^2+d^2)^{\alpha}\bigr).
\end{align}
\end{lem}
\begin{proof}
The left invariant differential operator corresponding to $Y\in\ig$ is 
given by $Yf(g)=\lim_{t\to0} (f(g\exp(tY))-f(g))/t$.
In particular, 
if we parametrize $G$ as
$\bigl(\smatr abcd,\scmatr{\vecy}{\vecz}\bigr)$, where $\vecy=\trans(y_1,\ldots,y_k)$ and
$\vecz=\trans(z_1,\ldots,z_k)$, then (cf.\ \eqref{FIXEDBASIS})
\begin{align}\label{DECAYTFNFROMCMNORMLEMpf1}
X_{3+\ell}=a\frac{\partial}{\partial y_\ell}+c\frac{\partial}{\partial z_\ell}
\qquad\text{and}\qquad
X_{3+k+\ell}=b\frac{\partial}{\partial y_\ell}+d\frac{\partial}{\partial z_\ell},
\qquad
\ell\in\{1,\ldots,k\}.
\end{align}
Now 
\begin{align*}
\wh f\left(\matr
abcd,\cmatr\bn\vecr\right)=\int_{\T^k}\int_{\T^k}
f\left(\matr abcd,\cmatr\vecy\vecz\right)e(-\vecr\vecz)\,d\vecy\,d\vecz,
\end{align*}
and hence by repeated integration by parts we have
\begin{align*}
(2\pi i r_\ell c)^m\cdot\wh f\left(\matr abcd,\cmatr\bn\vecr\right)
=\int_{\T^k}\int_{\T^k}
[X_{3+\ell}^mf]\left(\matr abcd,\cmatr\vecy\vecz\right)e(-\vecr\vecz)\,d\vecy\,d\vecz
\end{align*}
and
\begin{align*}
(2\pi i r_\ell d)^m\cdot\wh f\left(\matr abcd,\cmatr\bn\vecr\right)
=\int_{\T^k}\int_{\T^k}
[X_{3+k+\ell}^mf]\left(\matr abcd,\cmatr\vecy\vecz\right)e(-\vecr\vecz)\,d\vecy\,d\vecz.
\end{align*}
Hence
\begin{align*}
|r_\ell|^m\max(|c|^m,|d|^m)\cdot\left|\wh f\left(\matr abcd,\cmatr\bn\vecr\right)\right|
\leq(2\pi)^{-m}\|f\|_{\C^m_\alpha}\scrY\left(\matr abcd\right)^{-\alpha},
\end{align*}
for each $\ell\in\{1,\ldots,k\}$.
Using $\scrY(\smatr abcd)\geq\max(\sqrt3/2,(c^2+d^2)^{-1})$, we get \eqref{DECAYTFNFROMCMNORMLEMRES}.
\end{proof}

Using Lemma \ref{DECAYTFNFROMCMNORMLEM} we immediately obtain bounds
on \textit{derivatives} of $\wh f(\cdot,\cdot)$ with respect to the first variable.
We express these in terms of Iwasawa co-ordinates,
that is we write (by a slight abuse of notation)
\begin{align}\label{TFNIWASAWA}
\wh f(u,v,\theta;\veceta):=\wh f\left(\matr 1u01\matr{\sqrt v}00{1/\sqrt v}
\matr{\cos\theta}{-\sin\theta}{\sin\theta}{\cos\theta},\veceta\right),
\end{align}
for $u\in\R$, $v>0$, $\theta\in\R/2\pi\Z$, $\veceta\in\Z^{2k}$.

\label{SL2RGLIEALGDISCSEC}

\begin{lem}\label{DERDECAYTFNFROMCMNORMLEM2}
For any %
$\alpha\in\R_{\geq0}$, $\vecr\in\Z^k\setminus\{\bn\}$, 
integers $m,\ell_1,\ell_2,\ell_3\geq0$ and $f\in\C^{m+\ell}_\alpha(X)$,
where $\ell=\ell_1+\ell_2+\ell_3$, we have
\begin{align}\label{DERDECAYTFNFROMCMNORMLEM2RES}
&\left|\Bigl(\frac{\partial}{\partial u}\Bigr)^{\ell_1}
\Bigl(\frac{\partial}{\partial v}\Bigr)^{\ell_2}
\Bigl(\frac{\partial}{\partial\theta}\Bigr)^{\ell_3}
\wh f\left(u,v,\theta;\cmatr\bn\vecr\right)\right|
\ll_{m,\ell,\alpha}\|f\|_{\C_\alpha^{m+\ell}}\|\vecr\|^{-m}v^{\frac m2-\ell_1-\ell_2}\min(1,v^{-\alpha}).
\end{align}
\end{lem}
\begin{proof}
This is just as in \cite[Lemmas 4.3, 4.4]{SASL}.
\end{proof}

\subsection{Bounds when $\veceta\in B_k$}
\label{BKBOUNDS}
We now give bounds on $\widehat{f}(T,\veceta)$ when $\veceta\in B_k$.
We will use the Frobenius matrix norm, 
\begin{align*}
\left\|\matr abcd\right\|:=\sqrt{a^2+b^2+c^2+d^2},\qquad\matr abcd \in G'.
\end{align*}
\begin{remark}\label{FROBINIWASAWAREM}
In Iwasawa co-ordinates, for any $u\in\R$, $v>0$, $\theta\in\R/2\pi\Z$,
we have,
\begin{align*}
\left\|\matr 1u01\matr{\sqrt v}00{1/\sqrt
v}\matr{\cos\theta}{-\sin\theta}{\sin\theta}{\cos\theta}\right\|
\asymp\sqrt{\frac{u^2+v^2+1}v}.
\end{align*}
Indeed, 
\begin{align*}
\left\|\matr 1u01\matr{\sqrt v}00{1/\sqrt v}\right\|
=\left\|\matr{\sqrt v}{u/\sqrt v}0{1/\sqrt v}\right\|
=\sqrt{v+\frac{u^2}v+\frac1v}
=\sqrt{\frac{u^2+v^2+1}v};
\end{align*}
hence the stated relation follows using the compactness of
$\{\smatr{\cos\theta}{-\sin\theta}{\sin\theta}{\cos\theta}\col\theta\in\R/2\pi\Z\}$.
\end{remark}

We have the following analogue of Lemma \ref{DECAYTFNFROMCMNORMLEM}.
\begin{lem}\label{PARTBDECAYFNLEM}
For any $\veceta\in B_k$, $m\geq0$,
$f\in \C^m_0(X)$ and $T\in G'$, %
\begin{align}\label{PARTBDECAYFNLEMRES}
\left|\wh
f\bigl(T,\veceta\bigr)\right|\ll_m\frac{\|f\|_{\C^m_0}}{\bigl(\|T\|+\|\veceta\|/\|T\|\bigr)^m}.
\end{align}
\end{lem}
\begin{remark}\label{PARTBDECAYFNLEMREM}
As a consequence, for any $0<\beta<\frac12$ we have 
\begin{align*}
\left|\wh
f\bigl(T,\veceta\bigr)\right|\ll_{m}\frac{\|f\|_{\C^m_0}}{\|T\|^{m(1-2\beta)}\|\veceta\|^{m\beta}
}.
\end{align*}
\end{remark}
\begin{proof}
We write $\veceta=\cmatr\vecq\vecr$ and $T=\matr abcd$.
Repeated integration by parts gives (cf.\ the proof of Lemma \ref{DECAYTFNFROMCMNORMLEM}):
\begin{align*}
(2\pi i (q_\ell a+r_\ell c))^m\cdot\wh f\left(T,\veceta %
\right)
=\int_{\T^k}\int_{\T^k}
[X_{3+\ell}^mf]\biggl(T,\cmatr\vecy\vecz\biggr)e\biggl(-\veceta\cmatr\vecy\vecz\biggr)\,d\vecy\,
d\vecz
\end{align*}
and
\begin{align*}
(2\pi i (q_\ell b+r_\ell d))^m\cdot\wh f\left(T,\veceta %
\right)
=\int_{\T^k}\int_{\T^k}
[X_{3+k+\ell}^mf]\biggl(T,\cmatr\vecy\vecz\biggr)e\biggl(-\veceta\cmatr\vecy\vecz\biggr)\,d\vecy\,
d\vecz.
\end{align*}
Hence if we write $\veceta^{(\ell)}:=\cmatr{q_\ell}{r_\ell}\in\R^2$ then we conclude that for each
$\ell\in\{1,\ldots,k\}$ and for each column vector $\vecv$ of $T$, we have
\begin{align}\label{PARTBDECAYFNLEMPF1}
\bigl|\wh
f\bigl(T,\veceta\bigr)\bigr|\leq\frac{\|f\|_{\C^m_0}}{(2\pi)^m|\veceta^{(\ell)}\vecv|^m}.
\end{align}

Now fix a column vector $\vecv$ of $T$ with the largest norm.
Then $\|T\|\leq\sqrt2\|\vecv\|$.
By our definition of $B_k$, $\veceta$ has the property described in Lemma
\ref{GOODBPROPLEM},
i.e.\ there are $1\leq\ell_1<\ell_2\leq k$ such that
$r_j=0$ for all $j<\ell_1$, $q_j=0$ for all $j<\ell_2$,
and $r_{\ell_1}>0$, $0\leq r_{\ell_2}<|q_{\ell_2}|$.
In particular the vectors $\veceta^{(\ell_1)}$ and $\veceta^{(\ell_2)}$ are non-zero, hence both
have length
$\geq1$, 
and the angle between the lines
$\R\veceta^{(\ell_1)}$ and $\R\veceta^{(\ell_2)}$ in $\R^2$ is $>\frac\pi4$.
Hence the normal line to $\vecv$ in $\R^2$ has an angle $\geq\frac\pi8$ to at least one of the lines
$\R\veceta^{(\ell_1)}$ and $\R\veceta^{(\ell_2)}$,
and it follows that at least one of the scalar products 
$\veceta^{(\ell_1)}\vecv$ and $\veceta^{(\ell_2)}\vecv$
has an absolute value %
$\geq\sin(\frac\pi8)\|\vecv\|$. Hence using \eqref{PARTBDECAYFNLEMPF1} we get
\begin{align}\label{PARTBDECAYFNLEMPF2}
\bigl|\wh f\bigl(T,\veceta\bigr)\bigr|\ll_m\frac{\|f\|_{\C^m_0}}{\|T\|^m}.
\end{align}

Next let $\vecv'$ be the \textit{other} column vector of $T$, and let $\alpha\in(0,\frac\pi2]$ be
the angle between the lines
$\R\vecv$ and $\R\vecv'$; then $\|\vecv\|\|\vecv'\|\sin\alpha=1$, since $\det T=1$.
Let $\ell\in\{1,\ldots,k\}$ be the index for which $\|\veceta^{(\ell)}\|$ is maximal;
then $\|\veceta\|\leq\sqrt k\|\veceta^{(\ell)}\|$.
Now the normal line to $\veceta^{(\ell)}$ in $\R^2$ must have an angle $\geq\frac\alpha2$ to at
least one of the lines
$\R\vecv$ and $\R\vecv'$.
Hence either 
\begin{align*}
|\veceta^{(\ell)}\vecv|\geq\|\veceta^{(\ell)}\|\|\vecv\|\sin(\tfrac\alpha2)
>\tfrac12\|\veceta^{(\ell)}\|\|\vecv\|\sin\alpha=\frac{\|\veceta^{(\ell)}\|}{2\|\vecv'\|}\geq
\frac{\|\veceta\|}{2\sqrt k\|T\|}
\end{align*}
or else
\begin{align*}
|\veceta^{(\ell)}\vecv'|\geq\|\veceta^{(\ell)}\|\|\vecv'\|\sin(\tfrac\alpha2)
>\tfrac12\|\veceta^{(\ell)}\|\|\vecv'\|\sin\alpha=\frac{\|\veceta^{(\ell)}\|}{2\|\vecv\|}\geq
\frac{\|\veceta\|}{2\sqrt k\|T\|}.
\end{align*}
Applying \eqref{PARTBDECAYFNLEMPF1} for the appropriate column vector of $T$ we get
\begin{align}\label{PARTBDECAYFNLEMPF3}
\bigl|\wh f\bigl(T,\veceta\bigr)\bigr|\ll_m\frac{\|f\|_{\C^m_0}}{\bigl(\|\veceta\|/\|T\|\bigr)^m}.
\end{align}
Together, \eqref{PARTBDECAYFNLEMPF2} and \eqref{PARTBDECAYFNLEMPF3} imply \eqref{PARTBDECAYFNLEMRES}.
\end{proof}

Using Iwasawa co-ordinates, %
the bound in Remark \ref{PARTBDECAYFNLEMREM} can be expressed as follows,
for any $0<\beta<\frac12$ (cf.\ Remark \ref{FROBINIWASAWAREM}):
\begin{align}\label{PARTBDECAYFNLEMREMEXPL}
\left|\wh f\bigl(u,v,\theta;\veceta\bigr)\right|\ll_{m}
\|f\|_{\C^m_0}\Bigl(\frac v{u^2+v^2+1}\Bigr)^{m(\frac12-\beta)}\|\veceta\|^{-m\beta}.
\end{align}
Arguing again as in \cite[Lemmas 4.3, 4.4]{SASL} we now obtain the following bound on derivatives.
\begin{lem}\label{PARTBDERDECAYTFNFROMCMNORMLEM2}
Fix $0<\beta<\frac12$ and integers $m,\ell_1,\ell_2,\ell_3\geq0$.
For any $\veceta\in B_k$ and $f\in \C^{m+\ell}_0(X)$, 
where $\ell=\ell_1+\ell_2+\ell_3$, we have
\begin{align*}
&\left|\Bigl(\frac{\partial}{\partial u}\Bigr)^{\ell_1}
\Bigl(\frac{\partial}{\partial v}\Bigr)^{\ell_2}
\Bigl(\frac{\partial}{\partial\theta}\Bigr)^{\ell_3}
\wh f\left(u,v,\theta;\veceta\right)\right|
\ll_{m,\ell}\|f\|_{\C^{m+\ell}_0}\, v^{-\ell_1-\ell_2}
\Bigl(\frac v{u^2+v^2+1}\Bigr)^{m(\frac12-\beta)}\|\veceta\|^{-m\beta}.
\end{align*}
\end{lem}

\section{Obtaining the leading term}
\label{LEADINGTERMSEC}

Our task is to study the integral
\begin{align}\label{OURINTEGRAL}
\int_\R f\Bigl(\Gamma\,(1_2,\vecxi)u(x)a(y)\Bigr)h(x)\,dx
=\int_\R f\left(\Gamma
\left(\matr{\sqrt y}{x/\sqrt y}0{1/\sqrt y},\vecxi\right)\right)h(x)\,dx.
\end{align}
We may assume $0<y\leq1$ from the start, since \eqref{MAINTHM1RES} and \eqref{MAINTHM2RES} are otherwise trivial
(indeed, the left hand sides of \eqref{MAINTHM1RES}, \eqref{MAINTHM2RES} are always
$\ll\|f\|_{\C_0^0} S_{\infty,0,2}(h)$).
Decomposing $f$ as in \eqref{BASICFOURIER}, we get that \eqref{OURINTEGRAL} is
\begin{align}\label{MAINSTEP1}
=\int_\R \wh f\left(\matr{\sqrt y}{x/\sqrt y}0{1/\sqrt y},\bn\right)h(x)\,dx
\hspace{220pt}
\\\notag
+\sum_{\veceta\in A_k}\sum_{R\in \overline\Gamma'_\infty\backslash\overline\Gamma'/\Gamma'}
\sum_{\substack{T\in \Gamma'_\infty \backslash\overline\Gamma'\\T\equiv R\bmod{N}}}
e((\trans T\veceta)\vecxi)\int_\R
\wh f_R\left(T\matr{\sqrt y}{x/\sqrt y}0{1/\sqrt y},\veceta\right)h(x)\,dx
\\\notag
+\sum_{\veceta\in B_k}\sum_{R\in \overline\Gamma'/\Gamma'} \sum_{\substack{T\in\overline\Gamma'\\T\equiv R\bmod{N} }}
e((\trans T\veceta)\vecxi)\int_\R
\wh f_R\left(T\matr{\sqrt y}{x/\sqrt y}0{1/\sqrt y},\veceta\right)h(x)\,dx.
\end{align}
Here the change of order of summation and integration will be justified by 
an absolute convergence which holds for any $f$ and $h$ as in
Theorem \ref{MAINTHM1} or Theorem \ref{MAINTHM2};
cf.\ Lemmata \ref{ASUMABSCONVlem} and 
\ref{BPARTTRIVBOUNDLEM}
as well as \eqref{PARTBc0pf1}, \eqref{PARTBc0pf2} below.

Recall that $M\mapsto\wh f(M,\bn)$ is invariant under $\Gamma'=\Gamma(N)$;
hence the first integral in \eqref{MAINSTEP1}
is simply a weighted average along a closed horocycle
in $\Gamma'\bs G'$, %
a case which has been thoroughly studied in the literature (for arbitrary lattices in $G'=\SL(2,\R)$);
cf.\ in particular \cite{Burger}, \cite{FF}, \cite{iha}.
By the bound by Kim and Sarnak \cite{KS} towards the Ramanujan conjecture,
the smallest non-zero eigenvalue of the Laplace operator on the hyperbolic surface $\Gamma(N)\backslash\mathbb{H}$
satisfies $\lambda_1\geq\frac14-(\frac 7{64})^2$.
Using this in \cite[Thm.\ 1, Rem.\ 3.4]{iha}, we obtain
\begin{align}\label{MAINTERM}
\int_\R \wh f\left(\matr{\sqrt y}{x/\sqrt y}0{1/\sqrt y},\bn\right)h(x)\,dx
=\int_{\GaG} f\,d\mu\int_\R h\,dx
+O\Bigl(\|f\|_{\C_{0}^4}\,S_{1,0,1}(h)\,y^{\frac12-\frac7{64}}\Bigr).
\end{align}

\begin{remark}
Note that in the more general setting of Theorem \ref{INEFFECTIVETHM},
we could have e.g.\ $\Gamma=\Lambda\ltimes(\Z^2)^{\oplus k}$
with $\Lambda$ being a \textit{non-congruence} subgroup of $\SL(2,\Z)$.
If we would seek to extend the present methods to that case,
when carrying out this first step of using equidistribution on $\Lambda\backslash\SL(2,\R)$,
we would obtain an analogue of \eqref{MAINTERM} with an error term decaying as $O(y^{c(\Lambda)})$
for some $0<c(\Lambda)\leq\frac12$.   %
However in this case it is known that for certain choices of $\Lambda$ the spectral gap for
$\Lambda\backslash\SL(2,\R)$ can be made arbitrarily small \cite{Selberg65},
meaning that there is no uniform lower bound on the exponent $c(\Lambda)$.
\end{remark}

\section{Cancellation in an exponential sum}
\label{cancellationsec}

In this section, we derive bounds on certain exponential sums which give nontrivial cancellations in various sums that arise frequently in our arguments in the rest of the paper. Recall that $\overline\Gamma'=\SL(2,\Z)$ and $\Gamma'=\Gamma(N)$.
Let $R=\smatr{a_0}{b_0}{c_0}{d_0}\in\overline\Gamma'$ be given.
We set $[R]:=\Gamma'R=R\Gamma'$;
this is the set of all matrices in $\overline\Gamma'$ which are congruent to $R$ modulo $N$.
We let $\Gamma_\infty'\backslash [R]$ be a set of representatives for the right cosets of $\Gamma_\infty'$ 
contained in $[R]$,
and let $\Gamma_\infty'\backslash [R]/\Gamma_\infty'$ be a set of representatives for the double cosets of the
form $\Gamma_\infty' T\Gamma_\infty'$ with $T\in[R]$.  
For any given integer $c\equiv c_0\mod N$, we consider the following subsets:
\begin{align*}
[\Gamma_\infty'\backslash [R]\: ;\:c]:=\left\{\matr {a_1}{b_1}{c_1}{d_1}\in\Gamma_\infty'\backslash [R]\col c_1=c\right\}
\end{align*}
and
\begin{align*}
[\Gamma_\infty'\backslash [R]/\Gamma_\infty'\: ;\:c]
:=\left\{\matr {a_1}{b_1}{c_1}{d_1}\in\Gamma_\infty'\backslash [R]/\Gamma_\infty'\col c_1=c\right\}.
\end{align*}
Note that $[\Gamma_\infty'\backslash [R]/\Gamma_\infty'\: ;\:c]$ is a finite set.
We introduce the symbol $\sum^{(1)}$ \label{SUMONEDEF}
to denote summation over all matrices in $[\Gamma_\infty'\backslash [R]\: ;\:c]$,
and $\sum^{(2)}$ to denote summation over all matrices in $[\Gamma_\infty'\backslash [R]/\Gamma_\infty'\: ;\:c]$.
Note that the summation range in both $\sum^{(1)}$ and $\sum^{(2)}$ depend implicitly on $c$, $N$ and $R$.

\begin{remark}\label{cnegREM}
In the rest of this section we will assume $c\neq0$.
Note that we have an obvious bijection, $T\mapsto -T$, between the two sets
$[\Gamma_\infty'\backslash [R]\: ;\:c]$ and $[\Gamma_\infty'\backslash [-R]\: ;\:-c]$.
Hence without loss of generality we may assume $c>0$.
\end{remark}

For any $N,R,c$ as above with $c>0$, and $m,n\in\Z$,
we introduce the following generalized Kloosterman sum:
\begin{align}\label{SGENdef}
S(m,n;c;R,N)={\sum_{\smatr abcd}}^{\hspace{-5pt}(2)}\: e\Bigl(m\frac d{cN}+n\frac a{cN}\Bigr).
\end{align}
This sum is well-defined, since, for $\smatr abcd\in[\Gamma_\infty'\backslash [R]/\Gamma_\infty'\: ;\:c]$, 
both $d\mod cN$ and $a\mod cN$ are independent of the choice of coset representative. We begin by deriving bounds for the sums $S(m,n;c;R,N)$.

\begin{lem}\label{multiplicative}
Let $c$ and $N$ be positive integers, let $R=\smatr{a_0}{b_0}{c_0}{d_0}\in\overline\Gamma'$ with $c_0\equiv c\mod N$,
and let $M_1,M_2$ be coprime positive integers such that $cN=M_1M_2$. Then
\begin{align}\label{multiplicativeRES}
S(m,n;c;R,N)=S(m,\overline{M}_2^2n;K_3;R_1,K_1)S(m,\overline{M}_1^2n;K_4;R_2,K_2),
\end{align}
where $K_1=(N,M_1),K_2=(N,M_2), K_3=(c,M_1),K_4=(c,M_2)$,
and $\overline{M}_1\in\Z$ is a multiplicative inverse of $M_1\mod M_2$,
$\overline{M}_2\in\Z$ is a multiplicative inverse of $M_2\mod M_1$,
and
\begin{align}\label{multiplicativeR1R2}
R_1\equiv\matr{M_2a_0}{K_4b_0}{K_3}{\overline{M}_2d_0}\mod K_1,\qquad
R_2\equiv\matr{M_1a_0}{K_3b_0}{K_4}{\overline{M}_1d_0}\mod K_2.
\end{align}
\end{lem}
Note that the existence of matrices $R_1,R_2\in\overline\Gamma'$ satisfying \eqref{multiplicativeR1R2}
is guaranteed; cf. e.g., \cite[Thm.\ 4.2.1]{tM89}. 

\begin{proof}
By a straightforward analysis one verifies that
the map $\smatr abcd\mapsto\langle a\mod cN,d\mod cN\rangle$ gives a bijection from 
$[\Gamma_\infty'\backslash [R]/\Gamma_\infty'\: ;\:c]$ onto the set
$U[c,N;a_0,b_0,d_0]$
consisting of all pairs
$\langle a,d\rangle$ in $(\Z/cN\Z)^2$ satisfying
$a\equiv a_0\mod N$, $d\equiv d_0\mod N$
and $ad\equiv 1+b_0c\mod cN$.
Hence 
\begin{align}\label{SALTFORMULA}
S(m,n;c;R,N)=\sum_{\langle a,d\rangle\in U[c,N;a_0,b_0,d_0]} e\Bigl(m\frac d{cN}+n\frac a{cN}\Bigr).
\end{align}
The formula \eqref{multiplicativeRES} now follows
since the map
taking $\langle\langle a,d\rangle,\langle a',d'\rangle\rangle$ 
to $\langle M_2\overline M_2^2a+M_1\overline M_1^2a',M_2d+M_1d'\rangle$
is a bijection from
$U\bigl[K_3,K_1;M_2a_0,K_4b_0,\overline{M_2}d_0\bigr]
\times U\bigl[K_4,K_2;M_1a_0,K_3b_0,\overline{M_1}d_0\bigr]$
onto $U[c,N;a_0,b_0,d_0]$,
by the Chinese Remainder Theorem.
\end{proof}

For $n$ a positive integer, we write $\sigma(n)$ for the number of (positive) divisors of $n$, and $\sigma_1(n)$ for
their sum: $\sigma(n)=\sum_{d\mid n}1$ and $\sigma_1(n)=\sum_{d\mid n}d$.

We now use the multiplicativity relation to 
prove that the generalized Kloosterman sums satisfy a Weil type bound
(cf.\ \eqref{final decompRES1}),
and to give an explicit formula in the case $n=0$.
\begin{lem}
\label{final decomp}
For any $m,n\in\Z$, $c,N\in\Z^+$ and %
$R=\smatr{a_0}{b_0}{c_0}{d_0}\in\overline\Gamma'$ with $c_0\equiv c\mod N$,
\begin{align}\label{final decompRES1}
|S(m,n;c;R,N) |\ll_N \sigma(c)(m,n,c)^{1/2}c^{1/2},
\end{align}
where $\sigma(c)$ is the number of (positive) divisors of $c$.
Moreover, in the case $n=0$, if we write $c=c_1c_2$, where $c_1\mid N^\infty$ and
$(c_2,N)=1$, then
\begin{align}\label{final decompRES2}
\mathrm S(m,0;c;R,N)=
I(c_1\,|\, m)\,\,\mu\Bigl(\frac{c}{(c,m)}\Bigr)\frac{\phi(c_2)c_1}{ \phi(c/(c, m))}
e\Bigl(m\frac{\overline{c_2}d_0}{c_1N}\Bigr),
\end{align}
where $I(\cdot)$ is the indicator function and $\overline{c_2} $ is a multiplicative inverse of $c_2\bmod{N}$.
In particular,
\begin{align}\label{final decompRES3}
|\mathrm S(m,0;c;R,N)|\leq (c,m).
\end{align}
\end{lem}
\begin{proof}
Let $c_1,c_2$ be as in the statement of the lemma, and set $N'=(c_1^\infty,N)$ and $N''=N/N'$.
Applying Lemma \ref{multiplicative} twice gives
\begin{align}\label{fdecomplempf1}
S(m,n;c;R,N)=S(m,n';c_1;R',N')S(m,n'';1;R'',N'')S(m,n_2;c_2;R_2,1),
\end{align}
for some $n_1,n_2,n',n''\in\Z$ and $R_1,R_2,R',R''\in\overline\Gamma'$.
Here $|S(m,n'';1;R'',N'')|=1$,
and the third factor is a standard Kloosterman sum;
$S(m,n_2;c_2;R_2,1)=S(m,n_2;c_2)$.
Regarding the first factor, elementary arguments give, with $R'=\smatr{a'}{b'}{c'}{d'}$:
\begin{align*}
S(m,n';c_1;R',N')&=
e\Bigl(n'\frac{\overline{d'}b'}{N'}\Bigr)\sum_{\substack{d\in\Z/c_1N'\Z\\ d\equiv d'\mod N'}}
e\Bigl(m\frac d{c_1N'}+n'\frac{\overline{d}}{c_1N'}\Bigr),
\\
&=\frac{e(n'\overline{d'}b'/N')}{N'}\sum_{j\in\Z/N'\Z}e\Bigl(-\frac{jd'}{N'}\Bigr)S(m+jc_1,n';c_1N').
\end{align*}
Now \eqref{final decompRES1} follows 
using Weil's bound on the standard Kloosterman sum \cite{Weil}, \cite[Ch.\ 11.7]{IK}.
Also \eqref{final decompRES2} and \eqref{final decompRES3} follow, using
basic facts about Ramanujan sums (cf., e.g., \cite[Ch.\ 3.2]{IK}).
\end{proof}

We are now set to state and prove the main lemma in this section.

\begin{lem}\label{lem:exp sum}
Let $N,c\in\Z^+$ and $R=\smatr{a_0}{b_0}{c_0}{d_0}\in\overline\Gamma'$, with $c_0\equiv c\mod N$.
Write $c=c_1c_2$, where $c_1\mid N^\infty$ and $(c_2,N)=1$.
Let $F(x_1,x_2)$ be a function in $\C^4(\R\times(\R/N\Z))$
such that $F$ and its derivatives $\partial_{x_1}^{j}\partial_{x_2}^{k}F$ for $j,k\leq2$ 
are in $\L^1(\R\times(\R/N\Z))$.
Then for any subset $K\subset\Z$ and any $\alpha\in \RR$,
\begin{align}
\label{eq:exp sum 1}
&\sumone e(d\alpha)F\Bigl(\frac dc,\frac ac\Bigr)
\\\notag
&=\sum_{\substack{m\in K\\ c_1\mid m}}
\biggl(\int_{ \R\times(\R/\Z) } F(Nx_1 ,Nx_2)e((cN\alpha-m)x_1)\,dx_1\,dx_2\biggr)
\mu\Bigl(\frac{c}{(c,m)}\Bigr)\frac{\phi(c_2)c_1}{\phi(c/(c,m))}e\Bigl(\frac{m\overline{c_2}d_0}{c_1N}\Bigr)
\\\notag
&+O\Bigl(\|F\|_{\L^1}+\|\partial_{x_1}^2F\|_{\L^1}\Bigr)\sum_{m\in
\Z\setminus K}\frac{(c,m)}{1+|m-cN\alpha |^{2} }
+O\Bigl(\|\partial_{x_2}^2F\|_{\L^1}+\|\partial_{x_1}^2\partial_{x_2}^2F\|_{\L^1}\Bigr)\sigma(c)\sqrt{c},
\end{align}
where $\overline{c_2}$ is a multiplicative inverse of $c_2\bmod{N}$.
\end{lem}
We remark that the sum in the left hand side of \eqref{eq:exp sum 1} is well-defined,
since, for $\smatr abcd\in[\Gamma_\infty'\backslash [R]\: ;\:c]$, both $d$ and the congruence
class of $a$ modulo $cN$ are independent of the choice of a coset representative.
\begin{proof}
Set 
\begin{align}\label{expsumlemPF4}
H(x_1,x_2)=\sum_{\ell\in \Z}f(x_1+\ell,x_2),\qquad
\text{where }\:
f(x_1,x_2):=F(Nx_1,Nx_2)e(cN\alpha x_1 ).
\end{align}
Note that since $f\in\C^4\cap\L^1(\R\times(\R/\Z))$;
the sum defining $H(x_1,x_2)$ is absolutely convergent for almost all
$(x_1,x_2)\in\R\times(\R/\Z)$, and $H\in\L^1(\R^2/\Z^2)$.
We will use the notation $F_{j,k}=\partial_{x_1}^{j}\partial_{x_2}^{k}F$
and $f_{j,k}=\partial_{x_1}^{j}\partial_{x_2}^{k}f$.
In order to get a stronger convergence statement, we note that
\begin{align}\label{expsumlemPFn1}
|f(x_1,x_2)|\leq\int_{x_1-\frac12}^{x_1+\frac12}\Bigl(|f(r,x_2)|+|f_{1,0}(r,x_2)|\Bigr)\,dr.
\end{align}
This follows by integrating the inequality
$|f(x_1,x_2)|\leq |f(r,x_2)|+\int_{r}^{x_1}|f_{1,0}(t,x_2)|\,|dt|$
over $r\in(x_1-\frac12,x_1+\frac12)$. %
Similarly, we have 
$|f_{j,0}(r,x_2)|\leq\int_{\R/\Z}(|f_{j,0}(r,s)|+|f_{j,1}(r,s)|)\,ds$,
and using this in \eqref{expsumlemPFn1}, we obtain the following
elementary Sobolev embedding type inequality:
\begin{align}\label{expsumlemPFn2}
|f(x_1,x_2)|\leq\int_{x_1-\frac12}^{x_1+\frac12}\int_{\R/\Z}
\Bigl(|f(r,s)|+|f_{1,0}(r,s)|+|f_{0,1}(r,s)|+|f_{1,1}(r,s)|\Bigr)\,ds\,dr.
\end{align}
Using \eqref{expsumlemPFn2} and the fact that $f_{j,k}\in\L^1(\R\times(\R/\Z))$ for $j,k\leq1$,
we conclude that the sum in \eqref{expsumlemPF4} is absolutely convergent
for \textit{all} $(x_1,x_2)$, uniformly over $(x_1,x_2)$ in any compact set. %
In particular, the function $H(x_1,x_2)$ is defined everywhere on $\R^2/\Z^2$, and is continuous.

Consider the Fourier coefficients of $H$,
\begin{align}\notag
a_{m,n}&=\int_{\R^2/\Z^2}H(x_1,x_2)e(-mx_1-nx_2)\,dx_1\,dx_2
\\\label{expsumlemPF2}
&=\int_{\R\times(\R/\Z)}F(Nx_1,Nx_2)e((cN\alpha-m)x_1-nx_2)\,dx_1\,dx_2.
\end{align}
Note that for any $j\leq1$ and $k\leq2$,
\begin{align}\label{expsumlemPF1}
\int_{\R/\Z}\bigl|F_{j,k}(x_1,Nx_2)\bigr|\,dx_2\to0\qquad\text{as }\: x_1\to\pm\infty.
\end{align}
This follows by applying \eqref{expsumlemPFn1} to $F_{j,k}(x_1,Nx_2)$
and using $F_{j,k},F_{j+1,k}\in\L^1(\R\times(\R/N\Z))$.
We may now integrate by parts repeatedly in \eqref{expsumlemPF2},
using \eqref{expsumlemPF1} to justify convergence,
to obtain
\begin{align}\label{expsumlemPF10}
a_{m,n}=\frac{N^{j+k}}{(2\pi i)^{j+k}(m-cN\alpha)^{j}n^{k}}
\int_{\R\times\R/\Z}F_{j,k}(Nx_1,Nx_2)e((cN\alpha-m)x_1-nx_2)\,dx_1\,dx_2,
\end{align}
for any $0\leq j,k\leq2$ and any integers $m,n$ subject to $m\neq cN\alpha$ if $j>0$ and $n\neq0$ if $k>0$.
Using this formula for $j\in\{0,2\}$ and $k=2$ gives
\begin{align}\label{expsumlemPF6}
|a_{m,n}|\ll_N(\|F_{0,2}\|_{\L^1}+\|F_{2,2}\|_{\L^1}
)\min(1,|m-cN\alpha
|^{-2}
)n^{-2},\qquad \forall
m\in \Z,\: n\in\Z\setminus\{0\}.
\end{align}
Similarly, using \eqref{expsumlemPF10} for $j\in\{0,2\}$ and $k=0$,
\begin{align}\label{expsumlemPF7}
|a_{m,0}|\ll_N (\|F\|_{\L^1}+\|F_{2,0}\|_{\L^1})\min(1,|m-cN\alpha|^{-2} ),\qquad \forall m\in \Z.
\end{align}
These bounds imply that the Fourier series of $H$ is absolutely convergent;
and since $H$ is continuous, $H$ is in fact equal to its Fourier series at every point 
(cf., e.g., \cite[Prop.\ 3.1.14]{grafakos}):
\begin{align}\label{expsumlemPF5}
H(x_1,x_2)=\sum_{m,n\in\Z}a_{m,n}e(mx_1+nx_2).
\end{align}

Now we consider the sum in the left hand side of \eqref{eq:exp sum 1}.
We have
\begin{align}\label{expsumlemPF20}
\sumone e(d\alpha) F\Bigl(\frac dc,\frac ac\Bigr)
=\sumtwo\sum_{\ell \in\Z}  F\Bigl(\frac {d+\ell cN}c,\frac ac\Bigr)e\bigl(\alpha(d+\ell cN)\bigr)
=\sumtwo H\Bigl(\frac d{cN},\frac a{cN}\Bigr).
\end{align}
Here all sums are absolutely convergent, since the sum in
\eqref{expsumlemPF4} is absolutely convergent and $\sum^{(2)}$ runs over a finite set.
Substituting \eqref{expsumlemPF5} in the last sum, and using \eqref{SGENdef}, we obtain
\begin{align}\label{expsumlemPF8}
\sumone e(d\alpha) F\Bigl(\frac dc,\frac ac\Bigr)
=\sum_{m,n\in\Z}a_{m,n}S(m,n;c;R,N).
\end{align}
Now we bound the contribution from all terms with $n\neq0$ in \eqref{expsumlemPF8} using
\eqref{expsumlemPF6}, \eqref{final decompRES1} and %
$\sum_{n\neq0}(m,n,c)^{1/2}n^{-2}\leq\sum_{n\neq0}|n|^{-3/2}\ll1$,
while the terms with $n=0$ are handled using
\eqref{expsumlemPF2} and \eqref{final decompRES2} when $m\in K$,
and using \eqref{expsumlemPF7} and \eqref{final decompRES3} when $m\notin K$.
In this way we obtain \eqref{eq:exp sum 1}. %
\end{proof}

\begin{remark}\label{cnegREM2}
If $c<0$ and $R=\smatr{a_0}{b_0}{c_0}{d_0}\in\overline\Gamma'$, $c_0\equiv c\mod N$,
then we see from Remark \ref{cnegREM}
that the sum $\sum^{(1)} e(d\alpha)F\bigl(\frac dc,\frac ac\bigr)$
remains the same if we replace $\langle c,R,\alpha\rangle$ by $\langle -c,-R,-\alpha\rangle$;
after this replacement, Lemma \ref{lem:exp sum} applies to the sum.
\end{remark}

Lemma \ref{lem:exp sum} %
will suffice for most parts of our discussion.
However, at one step in the treatment of the sum over $B_k$ in \eqref{MAINSTEP1}, we will need a more delicate estimate.
The point here is to obtain a bound which only involves derivatives 
$\partial_{x_1}^{\ell_1}\partial_{x_2}^{\ell_2}F$ with $\ell_2$ as small as possible.
Lemma \ref{lem:exp sum} requires using $\ell_2=2$ but the following lemma will effectively allow us to take
$\ell_2=\frac12+\ve$.
Cf.\ also Remark \ref{L12NORMrem} below.
We define a mixed $\L^1, \L^2$ norm for a functions $F$ on $\RR\times\RR/N\ZZ$ as follows:
$$\|F\|_{\L^{1,2}}=\left(\int_{\RR/N\ZZ}\left(\int_\RR|F(x_1,x_2)|\,dx_1\right)^2\,dx_2\right)^{1/2}.   $$
\begin{lem}\label{expsumlem2}
Let $0<\ve<1$ and let $N,c,R$ be as before. 
Let $F(x_1,x_2)$ be a function in $\C^3(\R\times(\R/N\Z))$
such that $\|\partial_{x_1}^{j}\partial_{x_2}^{k}F\|_{\L^{1,2}}<\infty$ for $j\leq2$, $k\leq1$.
Then for any $\alpha\in \RR$,
 \begin{align}\notag
\Biggl|\:\:\sumone e(d\alpha)F\Bigl(\frac dc,\frac ac\Bigr)\Biggr|
\ll_{\ve} (\|F\|_{\L^1}+\|\partial_{x_1}^2F\|_{\L^1})\sum_{\ell\in
\Z}\frac{(c,\lfloor cN\alpha+\ell\rfloor)}{1+\ell^2}
\hspace{100pt}
\\\label{eq:exp sum 3}
+\bigl(\|F\|_{\L^{1,2}}+\|\partial_{x_1}^2F\|_{\L^{1,2}}\bigr)^{\frac{1-\ve}2}
\bigl(\|\partial_{x_2}F\|_{\L^{1,2}}+\|\partial_{x_1}^2\partial_{x_2}F\|_{\L^{1,2}}\bigr)^{\frac{1+\ve}2}
\sigma(c)^{3/2}\sqrt{c}.
\end{align}
\end{lem}
\begin{proof}
Note that $\|F_{j,k}\|_{\L^1}\leq\sqrt N\|F_{j,k}\|_{\L^{1,2}}$
by Cauchy-Schwarz. Hence as in the proof of Lemma \ref{lem:exp sum},
$H(x_1,x_2)$ in \eqref{expsumlemPF4} is a well-defined continuous function on $\R^2/\Z^2$,
and its Fourier coefficients $a_{m,n}$ satisfy \eqref{expsumlemPF10} for any
$j\leq2$, $k\leq1$; that is,
\begin{align*}
a_{m,n}=\frac{N^{j+k}}{(2\pi i)^{j+k}(m-cN\alpha)^{j}n^{k}}\int_{\R/\Z}F_{m,j,k}(x_2)e(-nx_2)\,dx_2,
\end{align*}
where $F_{m,j,k}(x_2)=\int_{\R}F_{j,k}(Nx_1,Nx_2)e((cN\alpha-m)x_1)\,dx_1$
is a function on $\R/\Z$.
This gives a relation between $a_{m,n}$ and the $n$-th Fourier coefficient of $F_{m,j,k}$.
Using this relation for $j\in\{0,2\}$ and applying Parseval's identity, 
for any $k\geq0$ and $m\in\Z$, we get 
\begin{align}\notag
\sum_{n\in\Z\setminus\{0\}}n^{2k}|a_{m,n}|^2 
\ll_{k,N} &(\|F_{m,0,k}\|_{L^2}^2+\|F_{m,2,k}\|_{L^2}^2)\min(1,|m-cN\alpha|^{-4})
\\\label{expsumlem2pf2}
\ll_N &(\|F_{0,k}\|_{L^{1,2}}^2+\|F_{2,k}\|_{L^{1,2}}^2)\min(1,|m-cN\alpha|^{-4}).
\end{align}
Using this bound, $\sum_{n\neq0}|a_{m,n}|\leq(\sum_{n\neq0}|n|^{-2})^{\frac12}
(\sum_{n\neq0}n^2|a_{m,n}|^2)^{\frac12}$, and \eqref{expsumlemPF7},
we conclude that the Fourier series of $H$ is absolutely convergent,
and hence as in the proof of Lemma \ref{lem:exp sum}, we again have
\begin{align}\label{expsumlemPF8rep}
\sumone e(d\alpha) F\Bigl(\frac dc,\frac ac\Bigr)
=\sum_{m,n\in\Z}a_{m,n}S(m,n;c;R,N).
\end{align}
Using \eqref{expsumlemPF7} and \eqref{final decompRES3} for $n=0$,
and the generalized Weil bound \eqref{final decompRES1} for $n\neq0$, we see that \eqref{expsumlemPF8rep} is
\begin{align}\label{expsumlem2pf1}
\ll(\|F\|_{\L^1}+\|F_{2,0}\|_{\L^1})\sum_{\ell\in
\Z}\frac{(c,\lfloor cN\alpha+\ell\rfloor)}{1+\ell^2}
+\sigma(c)\sqrt c\sum_{m\in\Z}\sum_{n\neq0}|a_{m,n}|\sqrt{(n,c)}.
\end{align}
Note that for any integer $m$,
\begin{align*}
\sum_{n\neq0}|a_{m,n}|\sqrt{(n,c)}
&\leq
\sqrt{\sum_{n\neq0}\frac{(n,c)}{|n|^{1+\ve}}}
\sqrt{\sum_{n\neq0}|a_{m,n}|^2|n|^{1+\ve}}.
\end{align*}
Now, since $0<\ve<1$, we may apply H\"older's inequality with $p=\frac2{1-\ve}$ and $q=\frac2{1+\ve}$,
to get
\begin{align*}
\sum_{n\neq0}|a_{m,n}|^2|n|^{1+\ve}
=\sum_{n\neq0}|a_{m,n}|^{\frac2p}\cdot\bigl(|a_{m,n}|^{\frac2q}|n|^{1+\ve}\bigr)
\leq\biggl(\sum_{n\neq0}|a_{m,n}|^2\biggr)^{\frac1p}
\biggl(\sum_{n\neq0}|a_{m,n}|^2|n|^{(1+\ve)q}\biggr)^{\frac1q}
\\
\ll (\|F\|_{\L^{1,2}}+\|F_{2,0}\|_{\L^{1,2}})^{1-\ve}(\|F_{0,1}\|_{\L^{1,2}}+\|F_{2,1}\|_{\L^{1,2}})^{1+\ve}(1+|cN\alpha-m|)^{-4}.
\end{align*}
Here in the last step we use the Parseval bound, \eqref{expsumlem2pf2}, for $k=0$ and $k=1$.
Furthermore, 
\begin{align*}
\sum_{n\neq0}\frac{(n,c)}{|n|^{1+\ve}}
=2\sum_{d\mid c}\sum_{\substack{m\geq1\\(m,c)=d}}\frac d{m^{1+\ve}}
\leq2\sum_{d\mid c} d\sum_{k=1}^\infty\frac 1{(kd)^{1+\ve}}
\ll_\ve\sum_{d\mid c}d^{-\ve}\leq\sigma(c).
\end{align*}
Hence for any $m$,
\begin{multline*}
\sum_{n\neq0}|a_{m,n}|\sqrt{(n,c)}\\ \ll_\ve
\bigl(\|F\|_{\L^{1,2}}+\|F_{2,0}\|_{\L^{1,2}}\bigr)^{\frac{1-\ve}2}
\bigl(\|F_{0,1}\|_{\L^{1,2}}+\|F_{2,1}\|_{\L^{1,2}}\bigr)^{\frac{1+\ve}2}
(1+|cN\alpha-m|)^{-2}\sqrt{\sigma(c)}.
\end{multline*}
Using this bound in \eqref{expsumlem2pf1}, we obtain \eqref{eq:exp sum 3}.
\end{proof}

\section{The contribution from $A_k$-orbits}
\label{A sum}

\subsection{The case of Diophantine $\vecxi_2$}
We next study the sum in the second line of \eqref{MAINSTEP1}.
This sum will be bound by a generalization of the method in \cite{SASL}.
We first prove a bound which is adequate for \textit{any} $\vecxi=\cmatr{\vecxi_1}{\vecxi_2}\in\R^{2k}$
for which $\vecxi_2$ has good Diophantine 
properties.
This bound will be used in the special case $\vecxi_1=\bn$ in the proof of Theorem \ref{MAINTHM1}.
We note that we allow the special case $k=1$ in the present section, to allow comparison with
\cite[Prop.\ 8.3]{SASL};
cf.\ Remark~ \ref{Keq1remark} below.
\begin{prop}\label{AKBOUNDPROP}
Fix an integer $m\geq\max(8,k+3)$ and real numbers $a\in(\frac k2-\frac12,\frac m2-1)$ and $\ve>0$.
Then for any $f\in \C_a^m(X)$, $h\in\C^2(\R)$ with $S_{1,0,2}(h)<\infty$,
$\vecxi=\cmatr{\vecxi_1}{\vecxi_2}\in\R^{2k}$ and $0<y\leq1$, we have
\begin{align}\notag
\sum_{\veceta\in A_k}\sum_{R\in \overline\Gamma'_\infty\backslash\overline\Gamma'/\Gamma'}
\sum_{T\in \Gamma'_\infty \backslash [R]}
e((\trans T\veceta)\vecxi)\int_\R
\wh f_R\left(T\matr{\sqrt y}{x/\sqrt y}0{1/\sqrt y},\veceta\right)h(x)\,dx
\hspace{60pt}
\\\label{AKBOUNDPROPres}
\ll_{m,a,\ve}
\|f\|_{\C_a^m}S_{1,0,2}(h)\Bigl(\wdelta_{2a+1,\vecxi_2}(y^{-\frac12})+y^{\frac14-\ve}\Bigr).
\end{align}
\end{prop}
(Recall that the majorant function
$\wdelta_{\beta,\vecxi_2}(T)$ was introduced in \eqref{WMdef}.)

To start with the proof of Proposition \ref{AKBOUNDPROP},
let us fix some $\veceta=\cmatr\bn{\vecr}\in A_k$ and
$R=\matr{a_0}{b_0}{c_0}{d_0}\in\overline{\Gamma}'$. %
Using the notation introduced in Section \ref{cancellationsec},
the corresponding inner sum in \eqref{AKBOUNDPROPres} can be written as 
\begin{align}\label{APARTSTEP1}
\sum_{c\in c_0+N\Z}\:\sumone   \int_\R \widehat f_R\left(\matr abcd \matr{\sqrt y}{x/\sqrt y}0{1/\sqrt
y},\cmatr\bn\vecr\right)
e\left(\cmatr{c\vecr}{d\vecr}\vecxi\right)
h(x)\,dx.
\end{align}
The contribution from the terms with $c=0$ can be
bounded easily.
Indeed, there are at most two such terms in \eqref{APARTSTEP1}, and 
by Lemma \ref{DECAYTFNFROMCMNORMLEM} and the remarks below \eqref{FRDEF}, for any $b\in\Z$ we have
\begin{align}\label{C0BOUND}
\int_\R\biggl|\wh f_R\left(\pm\matr1b01\matr{\sqrt y}{x/\sqrt y}0{1/\sqrt
y},\cmatr\bn\vecr\right)h(x)\biggr|
\,dx
\ll\|h\|_{\L^1}\|f\|_{\C^m_0}y^{m/2}\|\vecr\|^{-m}.
\end{align}
Using this with $m=k+1$ and adding over all $\veceta\in A_k$,
we see that the contribution from all the terms with $T=\matr **0*$ in the second line of \eqref{MAINSTEP1}
is $O(\|h\|_{\L^1}\|f\|_{\C^{k+1}_0}y^{(k+1)/2})$, which is clearly subsumed by the bound in 
\eqref{AKBOUNDPROPres}.

Hence, from now on we focus on the terms with $c\neq0$.
The following lemma expresses the integral 
in \eqref{APARTSTEP1}
in the Iwasawa %
notation (cf.\ \eqref{TFNIWASAWA}).
\begin{lem}\label{IWASAWACOORDSLEM1}
For any $\smatr abcd\in G'$ with $c>0$,
any $y>0$ and any $f\in\C(G')$, %
\begin{align}\notag %
\int_\R
 f\left(\matr abcd\matr{\sqrt y}{x/\sqrt y}0{1/\sqrt
y}\right)h(x)\,dx
=\int_0^\pi 
f\biggl(\frac ac-\frac{\sin2\theta}{2c^2y},\frac{\sin^2\theta}{c^2y},
\theta\biggr)\,
h\biggl(-\frac dc+y\cot\theta\biggr)\,
\frac{y\,d\theta}{\sin^2\theta},
\end{align}
in the sense that if either of the integrals is absolutely convergent then so is the other,
and the equality holds.
\end{lem}
\begin{remark}\label{IWASAWACOORDSLEMREM2}
In the case $c<0$ one obtains exactly the same formula, except that
$\int_0^\pi$ is replaced by $\int_{-\pi}^0$
in the right hand side. %
\end{remark}
\begin{proof}
Cf.\ \cite[Lemma 6.1]{SASL}.
\end{proof}
We now prove that we have an absolute convergence in the left hand side of \eqref{AKBOUNDPROPres};
this fact is important in order to justify the manipulations which we will carry out later.
\begin{lem}\label{ASUMABSCONVlem}
Set $m=\max(3,k+1)$.
Then for any $f\in \C_0^m(X)$ and any $h\in\C^1(\R)$ with $S_{1,0,1}(h)<\infty$,
the expression
\begin{align}\label{ASUMABSCONVlemRES}
\sum_{\veceta\in A_k}\sum_{R\in \overline\Gamma'_\infty\backslash\overline\Gamma'/\Gamma'}
\sum_{T\in \Gamma'_\infty \backslash [R]}
\int_\R\biggl|\wh f_R\left(T\matr{\sqrt y}{x/\sqrt y}0{1/\sqrt y},\veceta\right)h(x)\biggr|\,dx
\end{align}
is finite for all $y>0$.
If, furthermore, $f\in\C_a^m(X)$ for some $a$ and $m$ subject to $a\geq 0$, $a>\frac k2-1$ and $m>2a+2$,
then the expression in \eqref{ASUMABSCONVlemRES} stays bounded as $y\to0$.
\end{lem}
(Note that the lemma in particular applies to any $f$ and $h$ as in Proposition \ref{AKBOUNDPROP}.)
\begin{proof}
As previously, we write $T=\smatr abcd$.
The contribution from terms with $c=0$ in \eqref{ASUMABSCONVlemRES} is treated %
by \eqref{C0BOUND}.
Thus, we only consider the terms with $c>0$;
the terms with $c<0$ can be dealt with similarly.
By Lemma \ref{IWASAWACOORDSLEM1}, and since $\overline\Gamma'_\infty\backslash\overline\Gamma'/\Gamma'$ is finite,
it suffices to prove that for each fixed 
$R=\smatr{a_0}{b_0}{c_0}{d_0}\in\overline\Gamma'_\infty\backslash\overline\Gamma'/\Gamma'$,
\begin{align}\label{ABSCONVPF1}
\sum_{\vecr\in\Z^k\setminus\{\bn\}}%
\:\sum_{\substack{c\equiv c_0\mod N\\ c>0}}\:\sumone
\int_0^\pi \biggl|\wh f_R\biggl(\frac ac-\frac{\sin2\theta}{2c^2y},\frac{\sin^2\theta}{c^2y},\theta,\cmatr{\bn}{\vecr}\biggr)\,
h\biggl(-\frac dc+y\cot\theta\biggr)\biggr|\,
\frac{y\,d\theta}{\sin^2\theta}
<\infty.
\end{align}
By Lemma \ref{DERDECAYTFNFROMCMNORMLEM2} (and the observations below \eqref{FRDEF}), 
for any $m\geq0$ and $a\in\R_{\geq0}$ we have
\begin{align*}
\biggl|\wh
f_R\left(u,\frac{\sin^2\theta}{c^2y},\theta;\cmatr\bn\vecr\right)\biggr|
\ll\|f\|_{\C^m_a}\|\vecr\|^{-m}\Bigl(\frac{|\sin\theta|}{c\sqrt y}\Bigr)^{m}
\min\Bigl(1,\Bigl(\frac{|\sin\theta|}{c\sqrt y}\Bigr)^{-2a}\Bigr),
\end{align*}
uniformly over $u\in\R$.
Using this bound for both $m=0$ and a general
$m\geq0$, we conclude
\begin{align}\notag
\biggl|\wh
f_R\left(u,\frac{\sin^2\theta}{c^2y},\theta;\cmatr\bn\vecr\right)\biggr|
\hspace{250pt}
\\\label{PROOF1STEP0help}
\ll\|f\|_{\C^m_a}
\min\biggl(\|\vecr\|^{-m}\Bigl(\frac{|\sin\theta|}{c\sqrt y}\Bigr)^m,
\|\vecr\|^{-m}\Bigl(\frac{|\sin\theta|}{c\sqrt y}\Bigr)^{m-2a},
\Bigl(\frac{|\sin\theta|}{c\sqrt y}\Bigr)^{-2a}\biggr).
\end{align}
We decompose the innermost sum in \eqref{ABSCONVPF1} in the same way as in \eqref{expsumlemPF20},
and then use the fact that 
\begin{align}
\sum_{n\in\Z}|h(\delta+n)|\leq S_{1,0,1}(h),\qquad\forall \delta\in\R,
\end{align}
which holds since
$|h(\alpha)|\leq\int_{\alpha-1/2}^{\alpha+1/2}(|h(x)|+|h'(x)|)\,dx$ for all $\alpha\in\R$.
From the proof of Lemma~\ref{multiplicative} we also have
\begin{align*}
\#[\Gamma_\infty'\backslash [R]/\Gamma_\infty'\: ;\:c]
=\#U[c,N;a_0,b_0,d_0]\leq c.
\end{align*}
Hence, we conclude that if $f\in\C_a^m(X)$ and $S_{1,0,1}(h)<\infty$ then the left hand side of \eqref{ABSCONVPF1} is
\begin{align*}
\ll %
\sum_{\vecr\in\Z^k\setminus\{\bn\}}\sum_{c=1}^\infty c
\int_0^\pi 
\min\biggl(\|\vecr\|^{-m}\Bigl(\frac{|\sin\theta|}{c\sqrt y}\Bigr)^{m},
\|\vecr\|^{-m}\Bigl(\frac{|\sin\theta|}{c\sqrt y}\Bigr)^{m-2{a}},
\Bigl(\frac{|\sin\theta|}{c\sqrt y}\Bigr)^{-2{a}}\biggr)
\,\frac{y\,d\theta}{\sin^2\theta}.
\end{align*}
Assuming $m>2a+2$, we get (cf.\ Lemma \ref{TECHNICALINTEGRALBOUNDLEM} below):
\begin{align}\notag
\ll\sum_{\vecr\in\Z^k\setminus\{\bn\}}\sum_{c=1}^\infty cy\left.\begin{cases}
\|\vecr\|^{-m}(c\sqrt y)^{-m}&\text{if }\:1\leq c\sqrt y
\\
\|\vecr\|^{-m}(c\sqrt y)^{2a-m}&\text{if }\:\|\vecr\|^{-1}\leq c\sqrt y\leq1
\\
\|\vecr\|^{-2a-1}(c\sqrt y)^{-1}&\text{if }\:c\sqrt y\leq\|\vecr\|^{-1}
\end{cases}\right\}
\hspace{100pt}
\\\notag
\ll\sum_{\vecr\in\Z^k\setminus\{\bn\}}\min\Bigl(\|\vecr\|^{-2-2a},\|\vecr\|^{-m}y^{1+a-\frac m2}\Bigr).
\end{align}
(Here $m>2a+1$ suffices for the first step, while $m>2a+2$ is needed to get the last bound.)
The last sum converges provided that either $m>k$ or $2a+2>k$;
and if $2a+2>k$ then it also stays bounded as $y\to0$.
\end{proof}

In the proof above, we used the following bound, which we will need again later.
\begin{lem}\label{TECHNICALINTEGRALBOUNDLEM}
Fix $a\geq0$ and $m>2a+1$.
Then for any $u>0$ and $r\geq1$ we have
\begin{align*}
\int_0^\pi\min\biggl(r^{-m}\bigl(u^{-1}\sin\theta\bigr)^m,r^{-m}\bigl(u^{-1}\sin\theta\bigr)^{m-2a},
\bigl(u^{-1}\sin\theta\bigr)^{-2a}\biggr)\,\frac{d\theta}{\sin^2\theta}
\\
\ll\begin{cases}
r^{-m}u^{-m} &\text{if }\: 1\leq u
\\
r^{-m}u^{2a-m}&\text{if }\: r^{-1}\leq u\leq1
\\
r^{-2a-1}u^{-1}&\text{if }\: u\leq r^{-1}.
\end{cases}
\end{align*}
\end{lem}
\begin{proof}
This is a straightforward case-by-case analysis.
\end{proof}

We continue with the proof of Proposition \ref{AKBOUNDPROP}.
Using Lemma \ref{IWASAWACOORDSLEM1} and Remark \ref{IWASAWACOORDSLEMREM2},
the sum in \eqref{APARTSTEP1}, excluding all terms with $c=0$, can be expressed as
\begin{align}\label{CONTRFROMFIXEDN}
&\sum_{\substack{c\equiv c_0\mod N\\ c>0}}
\int_{0}^\pi \sumone
h\biggl(-\frac dc+y\cot\theta\biggr)\,
\wh
f_R\biggl(\frac{a}c-\frac{\sin2\theta}{2c^2y},\frac{\sin^2\theta}{c^2y},
\theta;\cmatr\bn\vecr\biggr)
e\left(\cmatr{c\vecr}{d\vecr}\vecxi\right)\,\frac{y\,d\theta}{
\sin^2\theta}\\
&+\sum_{\substack{c\equiv c_0\mod N\\ c<0}}
\int_{-\pi}^0 \sumone
h\biggl(-\frac dc+y\cot\theta\biggr)\,
\wh
f_R\biggl(\frac{a}c-\frac{\sin2\theta}{2c^2y},\frac{\sin^2\theta}{c^2y},
\theta;\cmatr\bn\vecr\biggr)
e\left(\cmatr{c\vecr}{d\vecr}\vecxi\right)\,\frac{y\,d\theta}{
\sin^2\theta}.\nonumber
\end{align}
Here the change of order of summation and integration is justified by Lemma \ref{ASUMABSCONVlem}.
We will only deal with the first sum in \eqref{CONTRFROMFIXEDN}; the second sum can be dealt with similarly
(cf.\ Remark~\ref{cnegREM2}).
By Lemma \ref{lem:exp sum}, for any positive integer $c\equiv c_0\bmod{N}$ and any $\theta\in(0,\pi)$,  we have:
\begin{align}\notag
\sumone
h\biggl(-\frac dc+y\cot\theta\biggr)\,
\wh
f_R\biggl(\frac{a}c-\frac{\sin2\theta}{2c^2y},\frac{\sin^2\theta}{c^2y},
\theta;\cmatr\bn\vecr\biggr)
e(c\vecr\vecxi_1+d\vecr\vecxi_2)
\hspace{70pt}
\\\label{PROOF1STEP0}
\ll S_{1,0,2}(h)\biggl(\int_{\R/N\Z}\biggl|\wh
f_R\left(u,\frac{\sin^2\theta}{c^2y},\theta;\cmatr\bn\vecr\right)\biggr|\,
du\biggr)
\sum_{\ell\in\Z}\frac{(c,\lfloor cN\vecr\vecxi_2+\ell\rfloor)}{1+\ell^2}
\hspace{40pt}
\\\notag
+S_{1,0,2}(h)\biggl(\int_{\R/N\Z}
\biggl|\frac{\partial^2}{\partial u^2}
\wh
f_R\left(u,\frac{\sin^2\theta}{c^2y},\theta;\cmatr\bn\vecr\right)\biggr|\,
du\biggr)
\sigma(c)\sqrt c.
\end{align}
Here we will use the bound \eqref{PROOF1STEP0help}.
By a similar application of Lemma \ref{DERDECAYTFNFROMCMNORMLEM2} 
as in \eqref{PROOF1STEP0help},
we have for any $m'\in\Z_{\geq6}$ and $a'\in\R_{\geq0}$,
uniformly over $u\in\R$:
\begin{align*}
&\biggl|\frac{\partial^2}{\partial u^2}
\wh f_R\left(u,\frac{\sin^2\theta}{c^2y},\theta;\cmatr\bn\vecr\right)\biggr|
\\
&\ll \|f\|_{\C_{a'}^{m'}}\,\|\vecr\|^{-4}
\min\biggl(\|\vecr\|^{6-m'}\Bigl(\frac{|\sin\theta|}{c\sqrt y}\Bigr)^{m'-6},
\|\vecr\|^{6-m'}\Bigl(\frac{|\sin\theta|}{c\sqrt y}\Bigr)^{m'-6-2a'},
\Bigl(\frac{|\sin\theta|}{c\sqrt y}\Bigr)^{-2a'}\biggr).
\end{align*}
Using these bounds, we conclude that the first sum in \eqref{CONTRFROMFIXEDN} is
\begin{align}\notag
\ll S_{1,0,2}(h)\Biggl\{\|f\|_{\C_a^m}\sum_{c=1}^\infty
\int_0^\pi\min\biggl(\|\vecr\|^{-m}\Bigl(\frac{|\sin\theta|}{c\sqrt y}\Bigr)^m,
\|\vecr\|^{-m}\Bigl(\frac{|\sin\theta|}{c\sqrt y}\Bigr)^{m-2a},
\Bigl(\frac{|\sin\theta|}{c\sqrt y}\Bigr)^{-2a}\biggr)\,\frac{y\,d\theta}{\sin^2\theta}
\\\label{CONTRFROMFIXEDN2}
\times\sum_{\ell\in\Z}\frac{(c,\lfloor cN\vecr\vecxi_2+\ell\rfloor)}{1+\ell^2}
\\\notag
+\|f\|_{\C_{a'}^{m'}}\|\vecr\|^{-4}\sum_{c=1}^\infty
\int_0^\pi\min\biggl(\Bigl(\frac{|\sin\theta|}{\|\vecr\|c\sqrt y}\Bigr)^{{m'}-6},
\|\vecr\|^{6-{m'}}\Bigl(\frac{|\sin\theta|}{c\sqrt y}\Bigr)^{{m'}-6-2a'},
\Bigl(\frac{|\sin\theta|}{c\sqrt y}\Bigr)^{-2a'}\biggr)\,\frac{y\,d\theta}{\sin^2\theta}
\\\notag
\times\sigma(c)\sqrt c\Biggr\}.
\end{align}
By Lemma \ref{TECHNICALINTEGRALBOUNDLEM}, assuming $m>2a+1$ and $m'>2a'+7$,
\eqref{CONTRFROMFIXEDN2} is 
\begin{align}\notag
\ll\|f\|_{\C_a^m}S_{1,0,2}(h)\|\vecr\|^{-m}y^{1+a-\frac m2} 
\sum_{c=1}^\infty c^{-1}\bigl((\|\vecr\|\sqrt y)^{-1}+c\bigr)^{1+2a-m}
\sum_{\ell\in\Z}\frac{(c,\lfloor cN\vecr\vecxi_2+\ell\rfloor)}{1+\ell^2}
\\\label{CONTRFROMFIXEDN2a}
+\|f\|_{\C_{a'}^{m'}}S_{1,0,2}(h)
\|\vecr\|^{2-m'}y^{4+{a'}-\frac{m'}2}\sum_{c=1}^\infty
\bigl((\|\vecr\|\sqrt y)^{-1}+c\bigr)^{7+2{a'}-m'}\frac{\sigma(c)}{\sqrt c}.
\end{align}

\begin{lem}\label{PROOF1LEM1}
Fix $\beta>1$. Then for any $\alpha\in\R$ and $X>0$,
\begin{align}\label{PROOF1LEM1RES}
\sum_{c=1}^{\infty} c^{-1}(X+c)^{-\beta}
\sum_{k\in\Z}\frac{(c,k)}{1+|k-c\alpha|^2}
\ll_{\beta}  \begin{cases} 
X^{1-\beta}\,\sum_{j=1}^\infty\bigl(j^2+Xj\langle j\alpha\rangle\bigr)^{-1}  %
&\text{if }\: X\geq1
\\
1&\text{if }\:X<1.
\end{cases}
\end{align}
\end{lem}
\begin{proof}
If $\beta\in\Z$ then this is \cite[Lemma 8.2]{SASL} (with $\eta=1$ and $m=\beta+1$).
The proof extends without changes to the case of an arbitrary real $\beta>1$.
\end{proof}

\begin{lem}\label{SIGMASUMLEM}
For any $X>0$ and $\beta>\frac12$,
\begin{align*}
\sum_{c=1}^\infty(X+c)^{-\beta}\,\frac{\sigma(c)}{\sqrt c}\ll_\beta
\begin{cases}
X^{\frac12-\beta}\log(1+X)&\text{if }X\geq1
\\
1&\text{if }\: X<1.
\end{cases}
\end{align*}
\end{lem}
\begin{proof}
(Cf.\ \cite[Lemma 8.1]{SASL}.)
This follows by using $\sum_{1\leq c\leq x}\sigma(c)\ll x\log(1+x)$,
$\forall x\geq1$ (cf., e.g., \cite[(1.75)]{IK}),
and integration by parts.
\end{proof}

Using Lemma \ref{PROOF1LEM1} and Lemma \ref{SIGMASUMLEM},
and assuming from now on $m>2a+2$ and $m'>2a'+\frac{15}2$, we find that \eqref{CONTRFROMFIXEDN2a} 
(and thus \eqref{CONTRFROMFIXEDN2}) is
\begin{align}\notag
\ll\|f\|_{\C_a^m}S_{1,0,2}(h)\|\vecr\|^{-m}y^{1+a-\frac m2} 
\left.\begin{cases}
(\|\vecr\|\sqrt y)^{m-2a-2}\sum_{j=1}^\infty\min(j^{-2},\frac{\|\vecr\|\sqrt y}{j\langle j \vecr\vecxi_2\rangle})
&\text{if }\:\|\vecr\|\sqrt y\leq1
\\
1&\text{if }\: \|\vecr\|\sqrt y>1
\end{cases}\right\}
\hspace{15pt}
\\\notag
+\|f\|_{\C_{a'}^{m'}}S_{1,0,2}(h)
\|\vecr\|^{2-m'}y^{4+{a'}-\frac{m'}2}
\left.\begin{cases}
(\|\vecr\|\sqrt y)^{m'-2{a'}-\frac{15}2}\log(1+(\|\vecr\|\sqrt y)^{-1})&\text{if }\:\|\vecr\|\sqrt y\leq1
\\
1&\text{if }\:\|\vecr\|\sqrt y>1
\end{cases}\right\}
\\\label{PROOF1STEP1}
\ll\|f\|_{\C_a^m}S_{1,0,2}(h)\|\vecr\|^{-2a-1}
\left.\begin{cases}
\sum_{j=1}^\infty\min(j^{-2},\frac{\sqrt y}{j\langle j \vecr\vecxi_2\rangle})
&\text{if }\:\|\vecr\|\sqrt y\leq1
\\
(\|\vecr\|\sqrt y)^{2a+1-m}\sqrt y&\text{if }\: \|\vecr\|\sqrt y>1
\end{cases}\right\}
\hspace{115pt}
\\\notag
+\|f\|_{\C_{a'}^{m'}}S_{1,0,2}(h)
\|\vecr\|^{2-m'}y^{4+{a'}-\frac{m'}2}
\left.\begin{cases}
(\|\vecr\|\sqrt y)^{m'-2{a'}-\frac{15}2-\ve}&\text{if }\:\|\vecr\|\sqrt y\leq1
\\
1&\text{if }\:\|\vecr\|\sqrt y>1
\end{cases}\right\}.
\hspace{80pt}
\end{align}

In order to obtain a bound on the left hand side of \eqref{AKBOUNDPROPres}, %
we have to add over $R$ running through
the finite set $\overline\Gamma'_\infty\backslash\overline\Gamma'/\Gamma'$,
and add over all $\veceta\in A_k$, which means that $\vecr$ runs through a subset of $\Z^k\setminus\{\bn\}$.
For this to give a satisfactory result, we have to assume $2a+1>k$,
while in the second bound we choose $a'=\max(\frac k2-\frac{11}4,0)$;
with this choice, $m'=\max(8,k+3)$ satisfies the condition $m'>2a'+\frac{15}2$.
Adding now over $R$ and $\veceta$, we conclude that the left hand side of \eqref{AKBOUNDPROPres}
is $\ll \|f\|_{\C_a^m}S_{1,0,2}(h)\wdelta_{2a+1,\vecxi_2}(y^{-\frac12})+\|f\|_{\C_{a'}^{m'}}S_{1,0,2}(h)y^{\frac14-\ve}$.
Finally we note that $a>a'$, and so if we also assume $m\geq m'$ then $\|f\|_{\C_{a'}^{m'}}\ll\|f\|_{\C_a^m}$,
and we obtain the bound stated in Proposition \ref{AKBOUNDPROP}.
\hfill$\square$

\begin{remark}\label{Keq1remark}
In \eqref{PROOF1STEP1} the somewhat crude inequality
$\min\bigl(j^{-2},\frac{\|\vecr\|\sqrt y}{j\langle j \vecr\vecxi_2\rangle}\bigr)
\leq\|\vecr\|\min\bigl(j^{-2},\frac{\sqrt y}{j\langle j \vecr\vecxi_2\rangle}\bigr)$
was used.
In the special case $k=1$,
by avoiding using this bound one can keep $a=0$ in the treatment,
i.e.\ no cuspidal decay of $f$ has to be required;
cf.\ \cite[Prop.\ 8.3]{SASL}.
Note also that \cite[Prop.\ 8.3]{SASL}
has a better dependence on the test function $h$
(called ``$\nu$'' in \cite{SASL}) than Prop.~ \ref{AKBOUNDPROP},
namely, essentially, ``$S_{1,0,1+\ve}(h)$'' in place of $S_{1,0,2}(h)$.
We have avoided this in the present paper for simplicity of presentation.
\end{remark}

\subsection{The case $\vecxi_2=\bn$}
In this case, we prove the following bound.
\begin{prop}\label{AKBOUNDxi2zeroPROP}
Fix an integer $m\geq\max(8,k+3)$ and real numbers $a\in(\frac k2-\frac12,\frac m2-1)$ and $\ve>0$.
Then for any $f\in \C_a^m(X)$, $h\in\C^2(\R)$ with $S_{1,0,2}(h)<\infty$,
$\vecxi=\cmatr{\vecxi_1}{\bn}\in\R^{2k}$ and $0<y\leq1$, we have
\begin{align}\notag
\sum_{\veceta\in A_k}\sum_{R\in \overline\Gamma'_\infty\backslash\overline\Gamma'/\Gamma'}
\sum_{T\in \Gamma'_\infty \backslash[R]}
e((\trans T\veceta)\vecxi)\int_\R
\wh f_R\left(T\matr{\sqrt y}{x/\sqrt y}0{1/\sqrt y},\veceta\right)h(x)\,dx
\hspace{60pt}
\\\label{AKBOUNDxi2zeroPROPres}
\ll_{m,a,\ve}
\|f\|_{\C_a^m}S_{1,0,2}(h)\Bigl(\delta_{2a+1,\vecxi_1}(y^{-\frac12})+y^{\frac14-\ve}\Bigr).
\end{align}
\end{prop}
\begin{proof}
Arguing as in the proof of Proposition \ref{AKBOUNDPROP} we arrive again at the expression in \eqref{CONTRFROMFIXEDN},
where we now have $\cmatr{c\vecr}{d\vecr}\vecxi=c\vecr\vecxi_1$.
Applying Lemma \ref{lem:exp sum} with $\alpha=0$ and $K=\{0\}$ gives,
for any positive integer $c\equiv c_0\bmod N$,
decomposed as $c=c_1c_2$ where $c_1\mid N^\infty$ and $(c_2,N)=1$, and any $\theta\in(0,\pi)$:
\begin{align}\notag
\sumone &
h\biggl(-\frac dc+y\cot\theta\biggr)\,
\wh
f_R\biggl(\frac{a}c-\frac{\sin2\theta}{2c^2y},\frac{\sin^2\theta}{c^2y},
\theta;\cmatr\bn\vecr\biggr)
\\\label{PROOF1STEP01}
&=c_1\phi(c_2)\int_{\R}
h\bigl(-Nx_1+y\cot\theta\bigr) dx_1 \int_{\R/\Z}\wh
f_R\biggl(Nx_2,\frac{\sin^2\theta}{c^2y},
\theta;\cmatr\bn\vecr\biggr)dx_2 
\\\notag %
&\hspace{30pt}+O\biggl(S_{1,0,2}(h)\int_{\R/N\Z}\biggl|\wh
f_R\left(u,\frac{\sin^2\theta}{c^2y},\theta;\cmatr\bn\vecr\right)\biggr|\,du\biggr)
\sum_{\ell\in\Z\setminus\{0\}}\frac{(c,\ell)}{1+\ell^2}
\\\notag %
&\hspace{60pt}+O\biggl(S_{1,0,2}(h)\int_{\R/N\Z}\biggl|\frac{\partial^2}{\partial u^2}
\wh f_R\left(u,\frac{\sin^2\theta}{c^2y},\theta;\cmatr\bn\vecr\right)\biggr|\,
du\biggr) \sigma(c)\sqrt c.
\end{align}
Note here that the error term in last line is the same as the last line in \eqref{PROOF1STEP0};
hence it can be bounded as before;
cf.\ \eqref{CONTRFROMFIXEDN2}, \eqref{PROOF1STEP1}.
In the remaining error term in \eqref{PROOF1STEP01} we have
\begin{align*}
\sum_{\ell\in\Z\setminus\{0\}}\frac{(c,\ell)}{1+\ell^2}
\leq 2\sum_{m|c}\sum_{n=1}^\infty \frac{m}{(mn)^2}
\ll\sum_{m\mid c}\frac1m\leq\sigma(c).
\end{align*}
We can now argue as in the proof of Proposition \ref{AKBOUNDPROP}, but instead of Lemma \ref{PROOF1LEM1}
using the simple bound
\begin{align*}
&\|\vecr\|^{-m}y^{1+a-\frac m2} 
\sum_{c=1}^\infty c^{-1}\bigl((\|\vecr\|\sqrt y)^{-1}+c\bigr)^{1+2a-m}\sigma(c)
\ll \begin{cases}
\|\vecr\|^{-1-2a-\ve}y^{\frac{1-\ve}2}&\text{if }\:\|\vecr\|\sqrt y\leq1
\\
\|\vecr\|^{-m}y^{1+a-\frac m2}&\text{if }\:\|\vecr\|\sqrt y\geq1,
\end{cases}
\end{align*}
which is valid under the assumption that $m>2a+1$, and for any fixed $\ve>0$.
This leads to the conclusion that the contribution from the error terms in the last two lines of \eqref{PROOF1STEP01} to
the left hand side of \eqref{AKBOUNDxi2zeroPROPres} is
\begin{align*}
\ll \|f\|_{\C_a^m}S_{1,0,2}(h) y^{\frac12-\ve}
+\|f\|_{\C_{a'}^{m'}}S_{1,0,2}(h) y^{\frac14-\ve},
\end{align*}
with $a=\frac k2-\frac12$, $m=k+1$, $a'=\max(\frac k2-\frac{11}4,0)$ and $m'=\max(8,k+3)$.
This is clearly subsumed by the right hand side of \eqref{AKBOUNDxi2zeroPROPres}.

Now, it only remains to consider the first line in the right hand side of \eqref{PROOF1STEP01}.
The contribution from this line to the expression in the first line of \eqref{CONTRFROMFIXEDN}
can be written as follows,
after expressing the indicator function of $c\equiv c_0\mod N$ as $N^{-1}\sum_{b\mod N}e(b(c-c_0)/N)$:
\begin{align}\label{PROOF1STEP001}
\frac1{N^2}\int_{\R}h(x)\,dx\sum_{b\mod N}e\Bigl(-\frac{bc_0}N\Bigr)\int_0^\pi
\int_{\R/\Z}\sum_{c>0}e(c\alpha)c_1\phi(c_2)\wh f_R\biggl(Nx_2,\frac{\sin^2\theta}{c^2y},
\theta;\cmatr\bn\vecr\biggr)\,dx_2\,\frac{y\,d\theta}{\sin^2\theta},
\end{align}
where $\alpha:=\vecr\vecxi_1+b/N$.
We will use integration by parts to handle the sum over $c$.
Thus we let
\begin{align*}
B_\alpha(X)=\sum_{1\leq c\leq X} e(c\alpha)c_1\phi(c_2)=
\sum_{\substack{1\leq c_2\leq X\\ (c_2,N)=1 }}\sum_{\substack{1\leq c_1\leq X/c_2\\ c_1|N^\infty}}e(c_1c_2\alpha)c_1\phi(c_2).
\end{align*}
We have the following bound, analogous to \cite[Lemma 9.2]{SASL}.
\begin{lem}\label{BALPHALEM}
For any $\alpha\in \RR$, and $X\geq 1$, 
 \begin{align*}
 B_\alpha(X)\ll X^2\sum_{1\leq j\leq X}\min\left(\frac1{j^2},\frac1{Xj\langle j\alpha\rangle} \right).
\end{align*}
\end{lem}
\begin{proof}
 For any $c_2>0$, we have $\phi(c_2)=\sum_{d|c_2}\mu(c_2/d)d.$ Using this formula and substituting $c_2=jd$, we get
 \begin{align}\label{eq:Bvecrbound}
B_\alpha(X)&=\sum_{\substack{1\leq j\leq X\\ (j,N)=1 }}\mu(j)\sum_{\substack{1\leq d\leq X/j\\ (d,N)=1
}}\sum_{\substack{1\leq c_1\leq X/(jd)\\ c_1|N^\infty}} dc_1 e(jdc_1\alpha)
=\sum_{\substack{1\leq j\leq X\\ (j,N)=1 }}\mu(j)\sum_{1\leq k\leq X/j} k e(jk\alpha).
 \end{align}
However, for any $j,n\in \ZZ^+$, 
\begin{align*}
 \sum_{1\leq k\leq n} k e(jk\alpha)\ll \min\left(n^2,\frac{n}{\langle j\alpha\rangle}\right).
\end{align*}
(Cf.\ the proof of \cite[Lemma 9.2]{SASL}.) Applying this bound to \eqref{eq:Bvecrbound}, we get the lemma.
\end{proof}

For any $m\geq0$ and $a\in\R_{\geq0}$,
by Lemma \ref{DERDECAYTFNFROMCMNORMLEM2} we have (in a similar way as in \eqref{PROOF1STEP0help})
\begin{align}\notag
\frac{\partial}{\partial X}\wh f_R & \biggl(Nx_2,\frac{\sin^2\theta}{X^2y},\theta;\cmatr\bn\vecr\biggr)
\\\label{PROOFSTEP1STEP002}
&\ll
\|f\|_{\C^{m+1}_{a}}X^{-1}\min\left(\|\vecr\|^{-m}\left(\frac{
|\sin\theta|}{X\sqrt{y}}\right)^{m},\|\vecr\|^{-m}\left(\frac{|\sin\theta|}{X\sqrt{y}}\right)^{m-2a},\left(\frac{
|\sin\theta|}{X\sqrt{ y}}\right)^{-2a} \right).
\end{align}
Using integration by parts in \eqref{PROOF1STEP001}
(justified using \eqref{PROOFSTEP1STEP002} 
and $B_\alpha(X)\ll X^2$), we have:
\begin{align*}
\sum_{c>0}e(c\alpha)c_1\phi(c_2)\wh f_R\biggl(Nx_2,\frac{\sin^2\theta}{c^2y},\theta;\cmatr\bn\vecr\biggr)
=-\int_1^\infty \biggl(\frac{\partial}{\partial X}\wh f_R & \biggl(Nx_2,\frac{\sin^2\theta}{X^2y},\theta;\cmatr\bn\vecr\biggr)
\biggr)\, B_\alpha(X)\, dX.
\end{align*}
Furthermore,
Lemma \ref{TECHNICALINTEGRALBOUNDLEM} implies that for $m>2a+1$,
\begin{align*}
\int_0^\pi
\min\left(\|\vecr\|^{-m}\left(\frac{
|\sin\theta|}{X\sqrt{y}}\right)^{m},\|\vecr\|^{-m}\left(\frac{|\sin\theta|}{X\sqrt{y}}\right)^{m-2a},\left(\frac{
|\sin\theta|}{X\sqrt{ y}}\right)^{-2a} \right)\frac{y\,d\theta}{\sin^2\theta}
\hspace{60pt}
\\
\ll
X^{-1}\|\vecr\|^{-m}y^{1+a-\frac m2}\bigl((\|\vecr\|\sqrt y)^{-1}+X\bigr)^{1+2a-m}.
\end{align*}
Hence, also using Lemma \ref{BALPHALEM} and 
$\langle j\alpha\rangle=\langle j(\vecr\vecxi_1+b/N)\rangle\geq N^{-1}\langle jN\vecr\vecxi_1\rangle$, 
we find that the expression in \eqref{PROOF1STEP001} is
\begin{align}\label{PROOF1STEP002}
\ll S_{1,0,0}(h)\|f\|_{\C^{m+1}_a}\frac{y^{1+a-\frac m2}}{\|\vecr\|^m}
\int_1^\infty \bigl((\|\vecr\|\sqrt y)^{-1}+X\bigr)^{1+2a-m} %
\sum_{1\leq j\leq X}
\min\biggl(\frac1{j^2},\frac1{Xj\langle jN\vecr\vecxi_1\rangle} \biggr)\,dX.
\end{align}

\begin{lem}\label{UXINTLEM}
Assume $m>2a+2$. Then for any $\beta\in\R$ and $U>0$ we have
\begin{align}\notag
\int_1^\infty (U+X)^{1+2a-m}\sum_{1\leq j\leq X}\min\Bigl(\frac1{j^2},\frac1{Xj\langle j\beta\rangle} \Bigr)\,dX
\hspace{150pt}
\\\label{UXINTLEMres}
\ll_{m,a} (U+1)^{2+2a-m}\sum_{j=1}^\infty \min\Bigl(\frac1{j^2},\frac1{Uj\langle j\beta\rangle} \Bigr)
\Bigl(1+\log^+\Bigl(\frac{U\langle j\beta\rangle}j\Bigr)\Bigr).
\end{align}
\end{lem}
\begin{proof}
Changing order of summation and integration, the left hand side of \eqref{UXINTLEMres} becomes
\begin{align*}
\sum_{j=1}^\infty\int_j^\infty(U+X)^{1+2a-m}
\min\Bigl(\frac1{j^2},\frac1{Xj\langle j\beta\rangle} \Bigr)\,dX.
\end{align*}
Here for each $j\geq U$ we use 
$(U+X)^{1+2a-m}\leq X^{1+2a-m}$ and
$\min(\frac1{j^2},\frac1{Xj\langle j\beta\rangle})\leq j^{-2}$, to see that 
$\int_j^{\infty}\cdots\,dX\leq j^{2a-m}.$
On the other hand, for $j<U$ we have
\begin{align*}
\int_j^{\infty}\cdots\,dX &\leq U^{1+2a-m}\int_j^U \min\Bigl(\frac1{j^2},\frac1{Xj\langle j\beta\rangle} \Bigr)\,dX
+\min\Bigl(\frac1{j^2},\frac1{Uj\langle j\beta\rangle} \Bigr)\int_j^\infty X^{1+2a-m}\,dX
\\
&\ll U^{2+2a-m}\min\Bigl(\frac1{j^2},\frac1{Uj\langle j\beta\rangle}\Bigr)
\Bigl(1+\log^+\Bigl(\frac{U\langle j\beta\rangle}j\Bigr)\Bigr),
\end{align*}
where the last bound is proved by splitting into the two cases
$U\leq\frac j{\langle j\beta\rangle}$ and $U>\frac j{\langle j\beta\rangle}$
and evaluating the integrals.
The proof of the lemma is completed by 
adding up our bounds over all positive integers $j$, and noticing that $\sum_{j\geq U}j^{2a-m}\ll(U+1)^{1+2a-m}$,
which is bounded above by the contribution from $j=1$ in the right hand side of \eqref{UXINTLEMres}.
\end{proof}

Assuming now $m>2a+2$,
using the lemma we get, via \eqref{PROOF1STEP002}, that the expression in \eqref{PROOF1STEP001} is
\begin{align}\notag
\ll S_{1,0,0}(h)\|f\|_{\C^{m+1}_a}
\frac{(1+\|\vecr\|\sqrt y)^{2+2a-m}}{\|\vecr\|^{2+2a}}
\sum_{j=1}^\infty \min\Bigl(\frac1{j^2},\frac{\|\vecr\|\sqrt y}{j\langle jN\vecr\vecxi_1\rangle} \Bigr)
\Bigl(1+\log^+\Bigl(\frac{\langle jN\vecr\vecxi_1\rangle}{\|\vecr\|\sqrt yj}\Bigr)\Bigr)
\\\label{PROOF1STEP003}
\ll S_{1,0,0}(h)\|f\|_{\C^{m+1}_a}
\frac{(1+\|\vecr\|\sqrt y)^{2+2a-m}}{\|\vecr\|^{2+2a}}
\sum_{j=1}^\infty \min\Bigl(\frac1{j^2},\frac{\|\vecr\|\sqrt y}{j\langle j\vecr\vecxi_1\rangle} \Bigr)
\Bigl(1+\log^+\Bigl(\frac{\langle j\vecr\vecxi_1\rangle}{\|\vecr\|\sqrt yj}\Bigr)\Bigr).
\end{align}
(Indeed, the last bound holds even if the last sum over $j$ is restricted to $j=N,2N,3N,\ldots$.)
Finally we have to add this bound over all $R$ in the finite set
$\overline{\Gamma}'_\infty\backslash\overline{\Gamma}'/\tGamma $,
and over all $\veceta\in A_k$, which means that $\vecr$ runs through a subset of $\Z^k\setminus\{\bn\}$.
Comparing with the definition \eqref{MMXIDEF}, 
assuming now $a>\frac{k-1}2$ (viz., $2a+1>k$),
we immediately find that the sum of the bound in \eqref{PROOF1STEP003} over all $\vecr\in\Z^k$ with 
$0<\|\vecr\|<y^{-1/2}$ is
\begin{align}\label{PROOF1STEP004}
\ll S_{1,0,0}(h)\|f\|_{\C^{m+1}_a}\,\delta_{2a+1,\vecxi_1}(y^{-\frac12}).
\end{align}
On the other hand, 
for $\vecr$ with $\|\vecr\|\geq y^{-1/2}$, the sum over $j$ in \eqref{PROOF1STEP003} equals $\sum_{j=1}^\infty j^{-2}=\pi^2/6$,
and hence the sum of the bound in \eqref{PROOF1STEP003} over all such $\vecr$ is, assuming $m>k$
\begin{align*}
\ll S_{1,0,0}(h)\|f\|_{\C^{m+1}_a}\sum_{\substack{\vecr\in\Z^k\\(\|\vecr\|\geq y^{-1/2})}}\|\vecr\|^{-m}y^{1+a-\frac m2}
\ll S_{1,0,0}(h)\|f\|_{\C^{m+1}_a}\, y^{1+a-\frac k2}.
\end{align*}
However this is subsumed by the bound \eqref{PROOF1STEP004}, since
$a>\frac{k-1}2$ and $\delta_{\mu,\vecxi}(T)\geq(T+1)^{-1}$ for all $T>0$
(as is clear by taking $\vecr=\vece_1$, $j=1$ in \eqref{MMXIDEF}).
Hence for any fixed $a>\frac{k-1}2$ and $m>2a+2$, $m\in\Z$, we have proved that 
the contribution from the first line in the right hand side of \eqref{PROOF1STEP01}
to the left hand side of \eqref{AKBOUNDxi2zeroPROPres} is bounded by \eqref{PROOF1STEP004}.
This completes the proof of Proposition~\ref{AKBOUNDxi2zeroPROP}.
\end{proof}

\section{The contribution from $B_k$-orbits}
\label{BKSEC}

\subsection{The case $\vecxi_1=\bn$}
In this section we will bound the sum in the third line of \eqref{MAINSTEP1}.
We will assume $k\geq2$ throughout this section,
since $B_k$ is empty for $k=1$.
We will prove:
\begin{prop}\label{BKBOUNDPROP}
Let $k\geq2$.
Fix a real number $\ve>0$ and an integer $m\geq\max(8,2k+1)$.
For any $f\in \C_0^{3m+3}(X)$, $h\in\C^1(\R)$ with $S_{\infty,2+\ve,1}(h)<\infty$,
$\vecxi_2\in\R^k$ and $0<y\leq1$, we have
\begin{align}\notag
\sum_{\veceta\in B_k}\sum_{R\in \overline\Gamma'/\Gamma'} \sum_{T\in [R]}
e\left((\trans T\veceta)\cmatr{\bn}{\vecxi_2}\right)\int_\R
\wh f_R\left(T\matr{\sqrt y}{x/\sqrt y}0{1/\sqrt y},\veceta\right)h(x)\,dx
\hspace{60pt}
\\\label{BKBOUNDPROPres}
\ll_{m,\ve}\|f\|_{\C_0^{3m+3}}S_{\infty,2+\ve,1}(h)\Bigl(\delta_{m-k,\vecxi_2}(y^{-\frac12})+y^{\frac14-\ve}\Bigr).
\end{align}
\end{prop}

Note that Theorem \ref{MAINTHM1} follows from Proposition \ref{BKBOUNDPROP}
together with Proposition \ref{AKBOUNDPROP} and the relations \eqref{MAINSTEP1}, \eqref{MAINTERM}.
\vspace{5pt}

To start the proof of Proposition \ref{BKBOUNDPROP},
note that taking $\beta=\frac13$ in Lemma \ref{PARTBDERDECAYTFNFROMCMNORMLEM2}, replacing $m$ by $3m$
and using the remarks below \eqref{FRDEF}, we get
\begin{align}\label{PARTBDECAYFNLEMREMEXPL2}
\left|\partial_u^{\ell_1}\partial_v^{\ell_2}\partial_\theta^{\ell_3}\wh f_R\left(u,v,\theta;\veceta\right)\right|
\ll_{m,\ell}
\|f\|_{\C_0^{3m+\ell}}\,\|\veceta\|^{-m}\, v^{-\ell_1-\ell_2}\Bigl(\frac v{u^2+v^2+1}\Bigr)^{m/2},
\end{align}
for all $R\in\overline{\Gamma}'$, $u\in\R$, $v>0$, $\theta\in\R/2\pi\Z$,
$\veceta\in B_k$ and $\ell_1,\ell_2,\ell_3\geq0$, with $\ell=\ell_1+\ell_2+\ell_3$.

Any $T=\smatr abcd\in\overline{\Gamma}'$ with $c=0$ can be expressed as 
$T=\ve\smatr1n01$, where $\ve\in\{-1,1\}$ and $n\in\Z$,
and the contribution from these $T$ to the left hand side of \eqref{BKBOUNDPROPres} is
\begin{align}\label{PARTBc0pf1}
\ll\sum_{\veceta\in B_k}\sum_{\ve\in\{-1,1\}}\sum_{n\in\Z}
\int_\R \left|\wh f_R\left(\ve\matr{\sqrt y}{(x+n)/\sqrt y}0{1/\sqrt
y},\veceta\right)
h(x)\right|\,dx,
\end{align}
wherein $R$ %
denotes the unique element in our chosen system of
representatives
$\overline{\Gamma}'/\tGamma$
satisfying $R\equiv\ve\smatr1n01 \mod N$.
Using \eqref{PARTBDECAYFNLEMREMEXPL2} and $m\geq2k+1$, we get that the sum in consideration is
\begin{align}\notag
\ll \|f\|_{\C_0^{3m}}\sum_{\veceta\in B_k}\|\veceta\|^{-m}\int_\R\sum_{n\in\Z}
\frac{y^{m/2}}{(1+|x+n|)^{m}}\bigl| h(x)\bigr|\,dx
\ll \|f\|_{\C_0^{3m}}y^{\frac m2}\sum_{\veceta\in B_k}\|\veceta\|^{-m}
\int_\R \bigl| h(x)\bigr|\,dx
\\\label{PARTBc0pf2}
\ll\|f\|_{\C_0^{3m}}\|h\|_{\L^1}y^{\frac m2}.
\end{align}
This is clearly subsumed by the bound in \eqref{BKBOUNDPROPres}.

Hence from now on we focus on the terms for $T=\smatr abcd$ with $c\neq0$ in
the left hand side of \eqref{BKBOUNDPROPres}.
We will restrict to the case $c>0$; the case $c<0$ can be handled completely analogously.
We fix some $\veceta=\scmatr\vecq\vecr\in B_k$ and 
$R=\smatr{a_0}{b_0}{c_0}{d_0}\in\overline{\Gamma}'$.   %
Using Lemma \ref{IWASAWACOORDSLEM1}, the inner sum can be expressed as:
\begin{align}\label{BPART2}
\sum_{\substack{\smatr abcd\in[R]\\[2pt] c>0}}
e\bigl((b\vecq+d\vecr)\vecxi_2\bigr)
\int_0^\pi \wh f_R\left(\frac
ac-\frac{\sin2\theta}{2c^2y},\frac{\sin^2\theta}{c^2y},\theta;\veceta\right)
h\Bigl(-\frac dc+y\cot\theta\Bigr)\,\frac{y\,d\theta}{\sin^2\theta}.
\end{align}
Let us first record a trivial upper bound on \eqref{BPART2},
variants of which will be used repeatedly below.
\begin{lem}\label{BPARTTRIVBOUNDLEM}
For any $\veceta\in B_k$ and $R\in\overline{\Gamma}'/\tGamma$,
\begin{align}\label{BPARTTRIVBOUNDLEMRES}
\sum_{\substack{\smatr abcd\in[R]\\[2pt] c>0}}
\int_0^\pi \biggl|\wh f_R\left(\frac
ac-\frac{\sin2\theta}{2c^2y},\frac{\sin^2\theta}{c^2y},\theta;\veceta\right)
h\Bigl(-\frac dc+y\cot\theta\Bigr)\biggr|\,\frac{y\,d\theta}{\sin^2\theta}
\hspace{50pt}
\\ \notag
\ll_m\|f\|_{\C_0^{3m}} S_{\infty,2,0}(h)\|\veceta\|^{-m}.
\end{align}
\end{lem}
\begin{proof}
We overestimate the sum by letting $\langle a,c,d\rangle$ run through \textit{all} integer triples with 
$c>0$ and $ad\equiv1\mod c$. Using \eqref{PARTBDECAYFNLEMREMEXPL2} we then get that the left hand side of 
\eqref{BPARTTRIVBOUNDLEMRES} is
\begin{align}\label{BPARTTRIVBOUNDLEMPF1}
\ll\|f\|_{\C_0^{3m}}\|\veceta\|^{-m}\sum_{c=1}^\infty\sum_{\substack{d\in\Z\\(d,c)=1}}\int_0^\pi
\sum_{n\in\Z}\Bigl(\frac v{u_n^2+v^2+1}\Bigr)^{m/2}\,
\biggl|h\Bigl(-\frac dc+y\cot\theta\Bigr)\biggr|\,\frac{y\,d\theta}{\sin^2\theta},
\end{align}
where $v=v(y,c,\theta)=\frac{\sin^2\theta}{c^2y}$ and
$u_n=u_n(y,c,d,\theta)=n+\frac{\alpha}c-\frac{\sin2\theta}{2c^2y}$,
with $\alpha=\alpha(c,d)$ being the unique integer between $1$ and $c$ satisfying $\alpha d\equiv1\mod c$.
But here
\begin{align}\label{BPARTTRIVBOUNDLEMPF4}
\sum_{n\in\Z}\Bigl(\frac v{u_n^2+v^2+1}\Bigr)^{m/2}
\ll\sum_{\substack{n\in\Z\\|u_n|\leq1+v}}\Bigl(\frac v{v^2+1}\Bigr)^{m/2}+
\sum_{\substack{n\in\Z\\|u_n|>1+v}}\Bigl(\frac v{u_n^2}\Bigr)^{m/2}
\ll\min\bigl(v^{\frac m2},v^{1-\frac m2}\bigr),
\end{align}
where we used the fact that $m>2k\geq2$.
Furthermore, if $S_{\infty,2,0}(h)<\infty$ then we have
\begin{align}\label{BPARTTRIVBOUNDLEMPFtomodify}
&\sum_{d\in\Z}\Bigl|h\Bigl(-\frac dc+y\cot\theta\Bigr)\Bigr|
\leq S_{\infty,2,0}(h)\sum_{d\in\Z}\Bigl(1+\Bigl|-\frac dc+y\cot\theta\Bigr|\Bigr)^{-2}
\ll S_{\infty,2,0}(h)c.
\end{align}
Hence we obtain that \eqref{BPARTTRIVBOUNDLEMPF1} is
\begin{align}\label{BPARTTRIVBOUNDLEMPF2}
\ll\|f\|_{\C_0^{3m}} S_{\infty,2,0}(h)\|\veceta\|^{-m}y\sum_{c=1}^\infty c
\int_0^\pi \min\bigl(v^{\frac m2},v^{1-\frac m2}\bigr)\,\frac{\,d\theta}{\sin^2\theta}.
\end{align}
However,
\begin{align}\label{BPARTTRIVBOUNDLEMPF3}
\int_0^\pi \min\bigl(v^{\frac m2},v^{1-\frac m2}\bigr)\,\frac{\,d\theta}{\sin^2\theta}
\ll \min\bigl((c\sqrt y)^{-1},(c\sqrt y)^{-m}\bigr),
\end{align}
as one verifies by treating the two cases $c^2y\geq1$ and
$c^2y<1$ separately, and in the latter case,
splitting the interval for $\theta$ into the parts $\{\theta:|\sin\theta|<c\sqrt y\}$ and
$\{\theta:|\sin\theta|\geq c\sqrt y\}$.
Now the lemma follows by using \eqref{BPARTTRIVBOUNDLEMPF3} in \eqref{BPARTTRIVBOUNDLEMPF2}.
\end{proof}

Adding the bound in Lemma \ref{BPARTTRIVBOUNDLEM} over all $R\in\overline{\Gamma}'/\tGamma$
and $\veceta\in B_k$ (again using $m>2k$),
we immediately see that the sum in the left hand side of \eqref{BKBOUNDPROPres}
stays \textit{bounded} as $y\to0$.
In order to show that the sum actually \textit{decays} as $y\to0$,
we have to establish cancellation in \eqref{BPART2}.

It will be convenient later (cf.\ Lemma \ref{FDTHETADERLEM2MD} below)
to note that we may restrict the integral in \eqref{BPART2} to those $\theta\in(0,\pi)$
which satisfy $y|\cot\theta|\leq1$.
Indeed, if $y|\cot\theta|>1$ then $|\sin\theta|<y$, and we note that
for any $c\geq1$ we have, with $v %
=\frac{\sin^2\theta}{c^2y}$ as in the proof of Lemma \ref{BPARTTRIVBOUNDLEM},
\begin{align*}
\int_{\substack{0<\theta<\pi\\ (|\sin\theta|<y)}}\min\bigl(v^{\frac m2},v^{1-\frac m2}\bigr)\,\frac{\,d\theta}{\sin^2\theta}
=\int_{\substack{0<\theta<\pi\\ (|\sin\theta|<y)}} v^{\frac m2}\,\frac{\,d\theta}{\sin^2\theta}
\ll_m c^{-m}y^{\frac m2-1}.
\end{align*}
Using this bound in place of \eqref{BPARTTRIVBOUNDLEMPF3} in the proof of Lemma \ref{BPARTTRIVBOUNDLEM},
we conclude that the contribution from $\theta$ with $y|\cot\theta|>1$ in \eqref{BPART2} is
$\ll\|f\|_{\C_0^{3m}} S_{\infty,2,0}(h)\|\veceta\|^{-m}y^{m/2}$.
Adding this over $R$ and $\veceta$ as in the left hand side of \eqref{BKBOUNDPROPres},
we again obtain a bound which is (by far) subsumed by the bound in \eqref{BKBOUNDPROPres}.

Let us also note that if $T=\smatr abcd$ in \eqref{BPART2} has $d=0$ then necessarily $c=1$,
and inspecting the proof of Lemma \ref{BPARTTRIVBOUNDLEM} we see that 
the contribution from all such $T$ in \eqref{BPART2} is
$\ll\|f\|_{\C_0^{3m}} S_{\infty,2,0}(h)\|\veceta\|^{-m}\sqrt y$.
This gives a contribution $\ll\|f\|_{\C_0^{3m}} S_{\infty,2,0}(h)\sqrt y$ in the left hand side of \eqref{BKBOUNDPROPres},
which is ok.
Hence from now on we may consider the sum in \eqref{BPART2} restricted by $d\neq0$.

Next we will make use of the approximation $\frac ac=\frac{1+bc}{dc}\approx\frac bd$.
The error in doing so %
is controlled by the following lemma.
\begin{lem}\label{BPARTAPPRBOUNDLEM}
Assuming that $m\geq4$, we have
\begin{align}\notag
\sum_{\substack{\smatr abcd\in[R]\\[2pt] c>0,\: d\neq0}}
\int_0^\pi \left|
\wh f_R\left(\frac ac-\frac{\sin2\theta}{2c^2y},\frac{\sin^2\theta}{c^2y},\theta;\veceta\right)
-\wh f_R\left(\frac bd-\frac{\sin2\theta}{2c^2y},\frac{\sin^2\theta}{c^2y},\theta;\veceta\right)\right|
\hspace{50pt}
\\[-10pt] \label{BPARTAPPRBOUNDLEMRES}
\times
\biggl|h\Bigl(-\frac dc+y\cot\theta\Bigr)\biggr|\,\frac{y\,d\theta}{\sin^2\theta}
\ll_m\|f\|_{\C_0^{3m+1}} S_{\infty,2,0}(h)\frac{\sqrt y\log(2+y^{-1})}{\|\veceta\|^{m}}.
\end{align}
\end{lem}
\begin{proof}
For any $\smatr abcd\in\overline\Gamma'$ with $c,d\neq0$ we have,
letting $J$ be the interval with endpoints $\frac ac-\frac{\sin2\theta}{2c^2y}$ and $\frac bd-\frac{\sin2\theta}{2c^2y}$,
and using $\frac ac-\frac bd=\frac1{dc}$ and \eqref{PARTBDECAYFNLEMREMEXPL2}, 
\begin{align*}
&\left|\wh f_R\left(\frac ac-\frac{\sin2\theta}{2c^2y},\frac{\sin^2\theta}{c^2y},\theta;\veceta\right)
-\wh f_R\left(\frac bd-\frac{\sin2\theta}{2c^2y},\frac{\sin^2\theta}{c^2y},\theta;\veceta\right)\right|
\\
&\ll\frac1{|dc|}\sup_{x\in J}\left|\partial_{x}\wh f_R\left(x,\frac{\sin^2\theta}{c^2y},\theta;\veceta\right)\right|
\ll\|f\|_{\C_0^{3m+1}}\|\veceta\|^{-m} |c|^{-1} v^{-1}\Bigl(\frac v{u^2+v^2+1}\Bigr)^{m/2},
\end{align*}
with $v=\frac{\sin^2\theta}{c^2y}$ and $u=\frac ac-\frac{\sin2\theta}{2c^2y}$.
(We used the crude bound $|d|^{-1}\leq1$,
and the fact that $(u+\xi)^2+1\asymp u^2+1$ for all $u\in\R$, $|\xi|\leq1$.)
Hence, arguing as in the proof of Lemma \ref{BPARTTRIVBOUNDLEM},
and using the same notation ``$u_n$'' as there,
we find that the left hand side of \eqref{BPARTAPPRBOUNDLEMRES} is
\begin{align*}
\ll\|f\|_{\C_0^{3m+1}}\|\veceta\|^{-m}\sum_{c=1}^\infty c^{-1}\sum_{\substack{d\in\Z\\(d,c)=1}}\int_0^\pi
\sum_{n\in\Z}v^{-1}\Bigl(\frac v{u_n^2+v^2+1}\Bigr)^{m/2}\,
\biggl|h\Bigl(-\frac dc+y\cot\theta\Bigr)\biggr|\,\frac{y\,d\theta}{\sin^2\theta},
\end{align*}
The rest of the proof is very similar to Lemma \ref{BPARTTRIVBOUNDLEM}, except that we now use
\begin{align*}
\int_0^\pi \min\bigl(v^{\frac m2-1},v^{-\frac m2}\bigr)\,\frac{\,d\theta}{\sin^2\theta}
\ll \min\bigl((c\sqrt y)^{-1},(c\sqrt y)^{2-m}\bigr)
\end{align*}
in place of \eqref{BPARTTRIVBOUNDLEMPF3}.
\end{proof}

Adding the bound in Lemma \ref{BPARTAPPRBOUNDLEM} over all $R\in\overline{\Gamma}'/\tGamma$
and $\veceta\in B_k$ gives a bound\linebreak
$\|f\|_{\C_0^{3m+1}} S_{\infty,2,0}(h)\sqrt y\log(2+y^{-1})$,
and this is subsumed by the bound in \eqref{BKBOUNDPROPres}.
Hence from now on we may replace $\frac ac$ by $\frac bd$ in \eqref{BPART2}.
Restricting the summation to $d>0$
(the case $d<0$ being completely analogous),
and writing $I_y:=\{\theta\in(0,\pi)\col y|\cot\theta|\leq1\}$,
the resulting sum is:
\begin{align}\notag
&\sum_{\substack{\smatr abcd\in[R]\\[2pt] c>0,\: d>0}}
e\bigl((b\vecq+d\vecr)\vecxi_2\bigr)
\int_{I_y} \wh f_R\left(\frac
bd-\frac{\sin2\theta}{2c^2y},\frac{\sin^2\theta}{c^2y},\theta;\veceta\right)
h\Bigl(-\frac dc+y\cot\theta\Bigr)\,\frac{y\,d\theta}{\sin^2\theta}.
\end{align}
Replacing $\langle a,b,c,d\rangle$ by $\langle -b,d,-a,c\rangle$ in this sum gives,
with $\widetilde R:=\smatr{-c_0}{-a_0}{d_0}{b_0}$:
\begin{align}\notag
&=\sum_{\substack{\smatr abcd\in[\widetilde R]\\[2pt] a<0,\: c>0}}
e\bigl((d\vecq+c\vecr)\vecxi_2\bigr)
\int_{I_y} \wh f_R\left(\frac dc-\frac{\sin2\theta}{2a^2y},\frac{\sin^2\theta}{a^2y},\theta;\veceta\right)
h\Bigl(\frac ca+y\cot\theta\Bigr)\,\frac{y\,d\theta}{\sin^2\theta}
\\\label{BPART5}
&=\sum_{\substack{c\equiv d_0\mod N\\ c>0}}e(c\vecr\vecxi_2)\int_{I_y}\:\:\sumtone e(d\vecq\vecxi_2) 
F_{c,\theta}\Bigl(\frac dc,\frac ac\Bigr)\,\frac{y\,d\theta}{\sin^2\theta},
\end{align}
where $\sum^{(\widetilde 1)}$ is the same as $\sum^{(1)}$ (cf.\ p.\ \pageref{SUMONEDEF}) but using $\widetilde R$
in place of $R$, and for any $c\in\Z^+$ and $\theta\in(0,\pi)$, $F_{c,\theta}(x_1,x_2)$
is the function on $\R\times(\R/N\Z)$ given by
\begin{align}\label{FDTHETADEF}
F_{c,\theta}(x_1,x_2):=\sum_{\substack{s\in x_2+N\Z\\ s<0}}
\wh f_R\biggl(x_1-\frac{\sin2\theta}{2yc^2s^2},\frac{\sin^2\theta}{yc^2s^2},\theta;\veceta\biggr)
h\Bigl(\frac1s+y\cot\theta\Bigr).
\end{align}
(Note that $F_{c,\theta}$ also depends on $N,y,R,\veceta$.)
Using $|\wh f_R(u,v,\theta;\veceta)|\ll\min(v,v^{-1})^{m/2}$, cf.\ \eqref{PARTBDECAYFNLEMREMEXPL2},
we see that the sum defining $F_{c,\theta}(x_1,x_2)$ is %
absolutely convergent,
and that $F_{c,\theta}(x_1,x_2)$ is continuous on $\R\times(\R/N\Z)$.
If $F_{c,\theta}$ is sufficiently differentiable with the first few derivatives being in $\L^{1,2}$, 
then we may apply Lemma \ref{expsumlem2},
to see that, for any $0<\ve<\frac12$,
\begin{align}\notag
\sumtone e(d\vecq\vecxi_2) F_{c,\theta}\Bigl(\frac dc,\frac ac\Bigr)
\ll_{\ve}(\|F_{c,\theta}\|_{\L^1}+\|\partial_{x_1}^2F_{c,\theta}\|_{\L^1})
\sum_{\ell\in\Z}\frac{(c,\lfloor cN\vecq\vecxi_2+\ell\rfloor)}{1+\ell^2}
\hspace{30pt}
\\\label{BPART4}
+\bigl(\|F_{c,\theta}\|_{\L^{1,2}}+\|\partial_{x_1}^2F_{c,\theta}\|_{\L^{1,2}}\bigr)^{\frac12-\ve}
\bigl(\|\partial_{x_2}F_{c,\theta}\|_{\L^{1,2}}+\|\partial_{x_1}^2\partial_{x_2}F_{c,\theta}\|_{\L^{1,2}}\bigr)^{\frac12+\ve}
\sigma(c)^{3/2}\sqrt{c}.
\end{align}

\vspace{4pt}

Bounds on the $\L^{1,2}$-norms of derivatives of $F_{c,\theta}$ are provided by the following lemma.

\begin{lem}\label{FDTHETADERLEM2MD}
For any integer $\ell$ with $1\leq \ell<\frac12(m-1)$, we have
$F_{c,\theta}\in\C^\ell(\R\times(\R/N\Z))$ 
provided that $f\in\C_0^{3m+\ell}(X)$ and $h\in\C^\ell(\R)$ with $S_{\infty,0,\ell}(h)<\infty$.
Furthermore for any integers $\ell_1,\ell_2\geq0$, and $a\in\R_{\geq0}$, $0<\ve<1$,
if $\ell=\ell_1+\ell_2$, $m>2\ell+1$,
$f\in\C_0^{3m+\ell}$, $h\in\C^{\ell_2}(\R)$, $S_{\infty,a,\ell_2}(h)<\infty$ and $y|\cot\theta|\leq1$, then
\begin{align}\notag
\biggl(\int_{\R/N\Z}\biggl(\int_{\R}
\biggl|\frac{\partial^{\ell_1+\ell_2}}{\partial x_1^{\ell_1}\partial x_2^{\ell_2}}F_{c,\theta}(x_1,x_2)\biggr|\,
\,dx_1\biggr)^2\,dx_2\biggr)^{1/2}
\ll_{m,\ell,\ve}
\|f\|_{\C_0^{3m+\ell}}S_{\infty,a,\ell_2}(h)\|\veceta\|^{-m}
\hspace{40pt}
\\\label{FDTHETADERLEM2MDRES1}
\times \begin{cases}
|\sin\theta|^{-\ell_2}\bigl(\frac{|\sin\theta|}{c\sqrt y}\bigr)^{1-\ell_2+\frac{\ve}2}
&\text{if }\: c\sqrt y\leq|\sin\theta|
\\[10pt]
|\sin\theta|^{-\ell_2}\bigl(\frac{|\sin\theta|}{c\sqrt y}\bigr)^{\frac12+a-\ell_2}
\left\{1+\bigl(\frac{|\sin\theta|}{c\sqrt y}\bigr)^{m-\frac12-a-2\ell_1+\ell_2}\right\}
&\text{if }\: |\sin\theta|\leq c\sqrt y\leq1
\\[10pt]
\bigl(\frac{|\sin\theta|}{c\sqrt y}\bigr)^{\frac12+a-2\ell_2}
\left\{1+|\sin\theta|^{-\ell_2}\bigl(\frac{|\sin\theta|}{c\sqrt y}\bigr)^{m-\frac12-a-2\ell_1+2\ell_2}\right\}
&\text{if }\: c\sqrt y\geq1.
\end{cases}
\end{align}
\end{lem}
\begin{proof}
By repeated differentiation we obtain, for any $\ell_1\geq0$ and $\ell_2\geq1$,
\begin{align}\notag
&\frac{\partial^{\ell}}{\partial x_1^{\ell_1}\partial s^{\ell_2}}
\biggl(\wh f_R\biggl(x_1-\frac{\sin2\theta}{2yc^2s^2},\frac{\sin^2\theta}{yc^2s^2},\theta;\veceta\biggr)
h\Bigl(\frac1s+y\cot\theta\Bigr)\biggr)
\\\label{FDTHETADERLEMPF1}
&=\sum_{1\leq {\alpha}+{\beta}+{\gamma}\leq\ell_2}K_{{\alpha},{\beta},{\gamma}}^{(\ell_2)}\,\frac{(\sin2\theta)^{\alpha}(\sin\theta)^{2{\beta}}}{(yc^2s^2)^{{\alpha}+{\beta}}s^{{\gamma}+\ell_2}}
\bigl[\partial_1^{\ell_1+{\alpha}}\partial_2^{\beta}\wh f_R\bigr]
\biggl(x_1-\frac{\sin2\theta}{2yc^2s^2},\frac{\sin^2\theta}{yc^2s^2},\theta;\veceta\biggr)
h^{({\gamma})}\Bigl(\frac1s+y\cot\theta\Bigr),
\end{align}
where $\langle {\alpha},{\beta},{\gamma}\rangle$ runs through all triples of nonnegative integers satisfying 
$1\leq \alpha+{\beta}+{\gamma}\leq\ell_2$,
each coefficient $K_{a,{\beta},{\gamma}}^{(\ell_2)}$ is an integer,
and $\partial_1$ and $\partial_2$ denote differentiation with respect to the first and second argument of $\wh f_R$.
Using \eqref{PARTBDECAYFNLEMREMEXPL2},
we find that the absolute value of \eqref{FDTHETADERLEMPF1} is
\begin{align}\label{FDTHETADERLEMPF2}
\ll_{m,\ell}\|f\|_{\C_0^{3m+\ell}}\|\veceta\|^{-m}
\biggl(\Bigl(x_1-\frac{\sin2\theta}{2yc^2s^2}\Bigr)^2+\Bigl(\frac{\sin^2\theta}{yc^2s^2}\Bigr)^2+1\biggr)^{-\frac m2}
\Bigl(\frac{\sin^2\theta}{yc^2s^2}\Bigr)^{\frac m2-\ell_1}
|s|^{-\ell_2}
\\\notag
\times\sum_{1\leq \alpha+\beta+\gamma\leq\ell_2}|\sin\theta|^{-\alpha}|s|^{-\gamma}
\Bigl|h^{(\gamma)}\Bigl(\frac1s+y\cot\theta\Bigr)\Bigr|.
\end{align}
Here the sum in the second line is $\asymp_{\ell_2}\sum_{\gamma=0}^{\ell_2}|\sin\theta|^{\gamma-\ell_2}
|s|^{-\gamma}|h^{(\gamma)}(s^{-1}+y\cot\theta)|$.
On the other hand for %
$\ell_2=0$,
the left hand side of \eqref{FDTHETADERLEMPF1} trivially equals
$[\partial_1^{\ell_1}\wh f_R](\ldots)\,h(s^{-1}+y\cot\theta)$,
and thus the bound in \eqref{FDTHETADERLEMPF2} is again valid,
with the last sum replaced by ``$|h(s^{-1}+y\cot\theta)|$''.

Now assume $S_{\infty,0,\ell_2}(h)<\infty$. 
Then the bound in \eqref{FDTHETADERLEMPF2} is $\ll|s|^{-m+2\ell_1-\ell_2}$ for $|s|\geq1$ and 
$\ll|s|^{m+2\ell_1-2\ell_2}$ for $0<|s|\leq1$, uniformly with respect to $x_1\in\R$ 
when keeping all other parameters fixed.
Hence if $m>\max(2\ell_1-\ell_2+1,-2\ell_1+2\ell_2)$, then the sum obtained by 
term-wise application of $\partial^\ell/(\partial x_1^{\ell_1}\partial x_2^{\ell_2})$ %
in \eqref{FDTHETADEF} is absolutely convergent, uniformly with respect to $(x_1,x_2)\in\R\times(\R/N\Z)$,
and defines a continuous function of $(x_1,x_2)$.
(The continuity along the line $x_2=0$ holds since the bound in \eqref{FDTHETADERLEMPF2} tends to $0$ as $|s|\to0$.)
In particular, if $m>2\ell+1$ and $S_{\infty,0,\ell}(h)<\infty$, then it follows that $F_{c,\theta}\in\C^\ell$
and that $(\partial^{\ell_1+\ell_2}/(\partial x_1^{\ell_1}\partial x_2^{\ell_2}))F_{c,\theta}$
may be computed by term-wise differentiation in the sum in \eqref{FDTHETADEF}, for any $\ell_1,\ell_2\geq0$
with $\ell_1+\ell_2\leq\ell$.

We now turn to the proof of the bound \eqref{FDTHETADERLEM2MDRES1}.
Using $S_{\infty,a,\ell_2}(h)<\infty$, $y|\cot\theta|\leq1$ and \eqref{FDTHETADERLEMPF2},
we see that \eqref{FDTHETADERLEMPF1} is
$\ll_{m,\ell}B(x_1,s)$, where
\begin{align}\notag
B(x_1,s):=\|f\|_{\C_0^{3m+\ell}}S_{\infty,a,\ell_2}(h)\|\veceta\|^{-m}
\biggl(\Bigl(x_1-\frac{\sin2\theta}{2yc^2s^2}\Bigr)^2+\Bigl(\frac{\sin^2\theta}{yc^2s^2}\Bigr)^2+1\biggr)^{-\frac m2}
\Bigl(\frac{\sin^2\theta}{yc^2s^2}\Bigr)^{\frac m2-\ell_1}
|s|^{-\ell_2}
\\\label{FDTHETADERLEM2PF1}
\times\sum_{\gamma=0}^{\ell_2}|\sin\theta|^{\gamma-\ell_2}|s|^{-\gamma+a}(1+|s|)^{-a}.
\end{align}
This bound is also valid when $\ell_2=0$.
Next, using the fact that $\int_{\R}(u^2+A)^{-m/2}\,du\ll A^{(1-m)/2}$
for all $A\geq1$, we have $\int_\R B(x_1,s)\,dx_1\ll_{m,\ell} B_1(s)$, where
\begin{align}\notag
B_1(s):=\|f\|_{\C_0^{3m+\ell}}S_{\infty,a,\ell_2}(h)\|\veceta\|^{-m}
\left(1+\Bigl(\frac{\sin^2\theta}{yc^2s^2}\Bigr)^2\right)^{\frac{1-m}2}
\Bigl(\frac{\sin^2\theta}{yc^2s^2}\Bigr)^{\frac m2-\ell_1}
|s|^{-\ell_2}
\hspace{70pt}
\\\label{FDTHETADERLEM2PFmodify1}
\times\sum_{\gamma=0}^{\ell_2}|\sin\theta|^{\gamma-\ell_2}|s|^{-\gamma+a}(1+|s|)^{-a}.
\end{align}
It follows that the left hand side of \eqref{FDTHETADERLEM2MDRES1}, after squaring, is
\begin{align}\notag
\ll_{m,\ell}\int_{\R/N\Z}\biggl(\int_{\R}\sum_{\substack{s\in x_2+N\Z\\s<0}}B(x_1,s)\,dx_1\biggr)^2\,dx_2
\ll_{m,\ell}\int_{\R/N\Z}\biggl(\sum_{\substack{s\in x_2+N\Z\\s<0}}B_1(s)\biggr)^2\,dx_2
\hspace{60pt}
\\\notag
\ll_{\ve}\int_{\R/N\Z}\sum_{\substack{s\in x_2+N\Z\\ s<0}}B_1(s)^2(1+|s|)^{1+\ve}\,dx_2
=\int_{-\infty}^0 B_1(s)^2(1+|s|)^{1+\ve}\,ds
\hspace{85pt}
\\\notag
=\|f\|_{\C_0^{3m+\ell}}^2S_{\infty,a,\ell_2}(h)^2\|\veceta\|^{-2m}\sum_{\gamma=0}^{\ell_2}|\sin\theta|^{2(\gamma-\ell_2)}
\hspace{200pt}
\\\label{FDTHETADERLEM2MDPF2}
\times
\int_0^{\infty}\Bigl(1+\Bigl(\frac{\sin^2\theta}{yc^2s^2}\Bigr)^2\Bigr)^{1-m}
\Bigl(\frac{\sin^2\theta}{yc^2s^2}\Bigr)^{m-2\ell_1}s^{-2\ell_2-2\gamma+2a}(s+1)^{1+\ve-2a}\,ds.
\end{align}
Using $m>2\ell+\frac32>2\ell+1+\frac{\ve}2$, we find that the integral in the last line of \eqref{FDTHETADERLEM2MDPF2} is
\begin{align*}
\ll_{m,\ell,\ve}
\begin{cases}
\bigl(\frac{|\sin\theta|}{\sqrt yc}\bigr)^{2-2\ell_2-2\gamma+\ve}
&\text{if }\: \frac{|\sin\theta|}{\sqrt yc}\geq1
\\[10pt]
\bigl(\frac{|\sin\theta|}{\sqrt yc}\bigr)^{-2\ell_2-2\gamma+2a+1}
+\bigl(\frac{|\sin\theta|}{\sqrt yc}\bigr)^{2m-4\ell_1}
&\text{if }\: \frac{|\sin\theta|}{\sqrt yc}\leq1.
\end{cases}
\end{align*}
Carrying out the addition over $\gamma$, %
we obtain the bound in \eqref{FDTHETADERLEM2MDRES1}.
\end{proof}

Note that Lemma \ref{FDTHETADERLEM2MD} also applies to give a bound on 
$\|\partial_{x_1}^{\ell_1}\partial_{x_2}^{\ell_2}F_{c,\theta}\|_{\L^1}$,
since $\|F\|_{\L^1}\leq\sqrt N\|F\|_{\L^{1,2}}$ for any function $F$ on $\R\times(\R/N\Z)$, by Cauchy-Schwarz.
However, in the case $c\sqrt y\leq|\sin\theta|$, we need to get rid of the $\ve$-power in
\eqref{FDTHETADERLEM2MDRES1}.
Thus we prove:
\begin{lem}\label{FDTHETADERLEM2MD2}
For any integers $\ell_1\geq0$ and $m>2\ell_1+1$,
for any $f\in \C_0^{3m+\ell_1}(X)$ and $h\in\C^{\ell_1}(\R)$ with $S_{\infty,0,0}(h)<\infty$,
if $0<c\sqrt y\leq|\sin\theta|$, then
\begin{align}\label{FDTHETADERLEM2MD2RES}
\int_{\R/N\Z}\int_\R\biggl|\frac{\partial^{\ell_1}}{\partial x_1^{\ell_1}}F_{c,\theta}(x_1,x_2)\biggr|\,dx_1\,dx_2
\ll_{m,\ell_1}
\|f\|_{\C_0^{3m+\ell_1}}S_{\infty,0,0}(h)\|\veceta\|^{-m}
\frac{|\sin\theta|}{c\sqrt y}.
\end{align}
\end{lem}
\begin{proof}
Following the proof of Lemma \ref{FDTHETADERLEM2MD},
we see that the left hand side of \eqref{FDTHETADERLEM2MD2RES} is
\linebreak
$\ll\int_{-\infty}^0B_1(s)\,ds$,
where $B_1(s)$ is given by \eqref{FDTHETADERLEM2PFmodify1} (with $a=\ell_2=0$).
This integral is bounded by a direct computation,
and we obtain the bound in \eqref{FDTHETADERLEM2MD2RES}.
\end{proof}
We are now ready to complete the proof of Proposition \ref{BKBOUNDPROP}.
Take $m\geq\max(8,2k+1)$, $a>2$ and $\ve,\ve'\in(0,\frac12)$;
also take $f\in\C_0^{3m+3}(X)$ and $h\in\C^1(\R)$ with $S_{\infty,a,1}(h)<\infty$.
Let $\vecxi_2\in\R^k$ and $0<y\leq1$ be given.
By Lemmata \ref{FDTHETADERLEM2MD}  and \ref{FDTHETADERLEM2MD2},
we have for every $c\in\Z^+$ and every $\theta$ with $y|\cot\theta|\leq1$, 
\begin{align}\notag
\|F_{c,\theta}\|_{\L^1}+\|\partial_{x_1}^2F_{c,\theta}\|_{\L^1}
\ll_{m,\ve'}\|f\|_{\C_0^{3m+2}}S_{\infty,a,0}(h)\|\veceta\|^{-m}
\hspace{150pt}
\\\notag
\times \begin{cases}
\frac{|\sin\theta|}{c\sqrt y}
&\text{if }\: c\sqrt y\leq|\sin\theta|
\\[10pt]
\bigl(\frac{|\sin\theta|}{c\sqrt y}\bigr)^{\frac12+a}
\left\{1+\bigl(\frac{|\sin\theta|}{c\sqrt y}\bigr)^{m-a-\frac92}\right\}
&\text{if }\: |\sin\theta|\leq c\sqrt y.
\end{cases}
\end{align}
Therefore,
\begin{align}\notag
\int_{I_y} %
\bigl(\|F_{c,\theta}\|_{\L^1}+\|\partial_{x_1}^2F_{c,\theta}\|_{\L^1}\bigr)
\,\frac{y\,d\theta}{\sin^2\theta}
\hspace{210pt}
\\\label{FDTHETALEMappl1}
\ll_{m,a}
\|f\|_{\C_0^{3m+2}}S_{\infty,a,0}(h)\|\veceta\|^{-m} y
\begin{cases}
(c\sqrt y)^{-1}(1+\log((c\sqrt y)^{-1})) & \text{if }\: c\sqrt y\leq 1
\\
(c\sqrt y)^{-\beta} & \text{if }\: c\sqrt y\geq 1,
\end{cases}
\end{align}
where $\beta:=\min(\frac12+a,m-4)>\frac52$.
Lemma \ref{FDTHETADERLEM2MD} also gives
\begin{align}\notag
&\bigl(\|F_{c,\theta}\|_{\L^{1,2}}+\|\partial_{x_1}^2F_{c,\theta}\|_{\L^{1,2}}\bigr)^{\frac12-\ve}
\bigl(\|\partial_{x_2}F_{c,\theta}\|_{\L^{1,2}}+\|\partial_{x_1}^2\partial_{x_2}F_{c,\theta}\|_{\L^{1,2}}\bigr)^{\frac12+\ve}
\\[10pt] \notag
&\hspace{30pt}\ll_{m,\ve'}\|f\|_{\C_0^{3m+3}}S_{\infty,a,1}(h)\|\veceta\|^{-m}
\\\label{FDTHETALEMappl2pre}
&\hspace{60pt}
\times \begin{cases}
|\sin\theta|^{-\frac12-\ve}\bigl(\frac{|\sin\theta|}{c\sqrt y}\bigr)^{\frac12+\ve'-\ve}
&\text{if }\: c\sqrt y\leq|\sin\theta|
\\[10pt]
|\sin\theta|^{-\frac12-\ve}\bigl(\frac{|\sin\theta|}{c\sqrt y}\bigr)^{a-\ve}
\left\{1+\bigl(\frac{|\sin\theta|}{c\sqrt y}\bigr)^{m-a-4+\ve}\right\}
&\text{if }\: |\sin\theta|\leq c\sqrt y\leq1
\\[10pt]
\bigl(\frac{|\sin\theta|}{c\sqrt y}\bigr)^{a-\frac12-2\ve}
\left\{1+|\sin\theta|^{-\frac12-\ve}\bigl(\frac{|\sin\theta|}{c\sqrt y}\bigr)^{m-a-\frac72+2\ve}\right\}
&\text{if }\: c\sqrt y\geq1,
\end{cases}
\end{align}
which leads to (using $a>2>\frac32+\ve$)
\begin{align}\notag
\int_{I_y} %
\bigl(\|F_{c,\theta}\|_{\L^{1,2}}+\|\partial_{x_1}^2F_{c,\theta}\|_{\L^{1,2}}\bigr)^{\frac12-\ve}
\bigl(\|\partial_{x_2}F_{c,\theta}\|_{\L^{1,2}}+\|\partial_{x_1}^2\partial_{x_2}F_{c,\theta}\|_{\L^{1,2}}\bigr)^{\frac12+\ve}
\,\frac{y\,d\theta}{\sin^2\theta}
\\\label{FDTHETALEMappl2}
\ll_{m,a,\ve}\|f\|_{\C_0^{3m+3}}S_{\infty,a,1}(h)\|\veceta\|^{-m}
y\begin{cases}
(c\sqrt y)^{-\frac32-\ve}&\text{if }\: c\sqrt y\leq1
\\
(c\sqrt y)^{-\delta} &\text{if }\: c\sqrt y\geq1,
\end{cases}
\end{align}
where $\delta:=\min(a-\frac12-2\ve,m-4)$.
Let us now also assume $\ve<\frac{a-2}3$.
Then $\delta>\frac32+\ve$, %
and using \eqref{BPART4}, \eqref{FDTHETALEMappl1} and \eqref{FDTHETALEMappl2},
it follows that the expression in \eqref{BPART5} %
is
\begin{align}\notag
\ll_{m,a,\ve}\|f\|_{\C_0^{3m+3}}S_{\infty,a,1}(h)\|\veceta\|^{-m} 
\sum_{c=1}^\infty \biggl\{y^{1-\frac{\beta}2} c^{-1}(y^{-\frac12}+c)^{1-\beta}
\Bigl(1+\log^+\Bigl(\frac1{c\sqrt y}\Bigr)\Bigr)
\hspace{60pt}
\\\label{BPART6}
\times \sum_{\ell\in\Z}\frac{(c,\lfloor cN\vecq\vecxi_2+\ell\rfloor)}{1+\ell^2}
\:+\: y^{\frac14-\frac{\ve}2}c^{-1-\frac{\ve}2}
\biggr\}.
\end{align}

We now need the following modification of 
Lemma \ref{PROOF1LEM1}:
\begin{lem}\label{ASLLEM8P2mod}
Fix $\beta>2$. Then for any $\alpha\in\R$ and $X\geq1$ we have
\begin{align}\label{ASLLEM8P2modres}
\sum_{c=1}^{\infty} c^{-1}(X+c)^{1-\beta}
&\Bigl(1+\log^+\Bigl(\frac Xc\Bigr)\Bigr)
\sum_{k\in\Z}\frac{(c,k)}{1+|k-c\alpha|^2}
\hspace{100pt}
\\\notag
&\ll_{\beta}  
X^{2-\beta}\sum_{j=1}^\infty\min\Bigl(\frac1{j^2},\frac1{Xj\langle j\alpha\rangle}\Bigr)
\biggl(1+\log^+\Bigl(\frac{X\langle j\alpha\rangle}{j}\Bigr)\biggr).
\end{align}
\end{lem}
\begin{proof}
The proof of \cite[Lemma 8.2]{SASL} carries over with easy modifications.
The main new technicality is to verify the bound
\begin{align*}
\sum_{d=1}^\infty
\frac{(X+\ell d)^{1-\beta}}{1+(d\langle\ell\alpha\rangle)^2}\biggl(1+\log^+\Bigl(\frac X{\ell d}\Bigr)\biggr)
\hspace{150pt}
\\
\ll_{\beta}
\begin{cases}
X^{2-\beta}\ell^{-1}&\text{if }\:1\leq X/\ell\leq\langle\ell\alpha\rangle^{-1}
\\
X^{1-\beta}\langle\ell\alpha\rangle^{-1}\bigl(1+\log\bigl(\frac{X\langle\ell\alpha\rangle}{\ell}\bigr)\bigr)
&\text{if }\:\langle\ell\alpha\rangle^{-1}<X/\ell
\\
\ell^{1-\beta}&\text{if }\:X/\ell<1,
\end{cases}
\end{align*}
valid for all $d,\ell\in\Z^+$.
\end{proof}

Using Lemma \ref{ASLLEM8P2mod},
it follows that \eqref{BPART6}, and hence also \eqref{BPART5}, is
\begin{align*}
&\ll_{\beta,\ve}\|f\|_{\C_0^{3m+3}}S_{\infty,a,1}(h)\|\veceta\|^{-m} 
\biggl\{\sum_{j=1}^\infty\min\Bigl(\frac1{j^2},\frac{\sqrt y}{j\langle j\vecq\vecxi_2\rangle}\Bigr)
\Bigl(1+\log^+\Bigl(\frac{\langle j\vecq\vecxi_2\rangle}{j\sqrt y}\Bigr)\Bigr)
+y^{\frac14-\frac{\ve}2}\biggr\}.
\end{align*}
(We replaced ``$N\vecq\vecxi_2$'' by $\vecq\vecxi_2$ through the same type of estimate as in \eqref{PROOF1STEP003}.)
Adding the last bound over $R\in \overline\Gamma'/\Gamma'$ and $\veceta\in B_k$,
using $\sum_{\veceta\in\Z^{2k}\setminus\{\bn\}}\|\veceta\|^{-m}<\infty$ %
and $\sum_{\vecr\in\Z^k}\|\scmatr{\vecq}{\vecr}\|^{-m}\ll\|\vecq\|^{k-m}$
for every $\vecq\in\Z^k\setminus\{\bn\}$,
and noticing that $a$ and $\ve$ can be taken arbitrarily near $2$ and $0$, respectively,
we obtain the bound in Proposition \ref{BKBOUNDPROP}.
This completes the proof of Proposition \ref{BKBOUNDPROP},
and also of Theorem \ref{MAINTHM1}.
\hfill$\square$

\begin{remark}\label{L12NORMrem}
We now explain why we had to use %
Lemma \ref{expsumlem2} in place of Lemma \ref{lem:exp sum} in the above proof of Proposition \ref{BKBOUNDPROP}.
One can prove a bound for the $\L^1$-norm of $\partial_{x_1}^{\ell_1}\partial_{x_2}^{\ell_2}F_{c,\theta}$
which is very similar to the bound in Lemma \ref{FDTHETADERLEM2MD},
and in the case $c\sqrt y\leq1$ this leads to a bound
\begin{align*}
\int_{\substack{0<\theta<\pi\\ (y|\cot\theta|\leq1)}}
\|\partial_{x_1}^{\ell_1}\partial_{x_2}^{\ell_2}F_{c,\theta}\|_{\L^{1}}\,\frac{y\,d\theta}{\sin^2\theta}
\ll y(c\sqrt y)^{-1-\ell_2}.
\end{align*}
Multiplying this with $\sigma(c)^{3/2}\sqrt c$ and adding over $c$
(cf.\ \eqref{BPART5}, \eqref{BPART4}) gives (if $\ell_2>\frac12$) a bound $y^{(1-\ell_2)/2}$,
which is insufficient.
Indeed, Lemma \ref{lem:exp sum} requires us to take $\ell_2$ as large as $2$.
Using instead the $\L^{1,2}$-norm and Lemma \ref{expsumlem2} 
means that we can effectively take $\ell_2$ to be as small as $\frac12+\ve$,
leading to the final bound $y^{\frac14-\frac{\ve}2}$.
(One could sharpen Lemma \ref{lem:exp sum} 
to a bound of the same style as in Lemma \ref{expsumlem2} but only involving the $\L^1$-norm;
this would allow us to use ``$\ell_2=1+\ve$'';
however this would still not be sufficient.)
\end{remark}

\subsection{The case $\vecxi_2=\bn$}
The treatment in this case is quite a bit easier than for $\vecxi_1=\bn$. %
We %
prove the following bound:
\begin{prop}\label{BKX1BOUNDPROP}
Let $k\geq2$.
Fix a real number $\ve>0$ and an integer $m\geq\max(7,2k+1)$.
For any $f\in \C_0^{3m+2}(X)$, $h\in\C^2(\R)$ with $S_{1,0,2}(h)<\infty$,
$\vecxi_1\in\R^k$ and $0<y\leq1$, we have
\begin{align}\notag
\sum_{\veceta\in B_k}\sum_{R\in \overline\Gamma'/\Gamma'} \sum_{T\in[R]} %
e\left((\trans T\veceta)\cmatr{\vecxi_1}{\bn}\right)\int_\R
\wh f_R\left(T\matr{\sqrt y}{x/\sqrt y}0{1/\sqrt y},\veceta\right)h(x)\,dx
\hspace{60pt}
\\\label{BKX1BOUNDPROPres}
\ll_{m,\ve}\|f\|_{\C_0^{3m+2}}S_{1,0,2}(h)\Bigl(\wdelta_{m-k,\vecxi_1}(y^{-\frac12})+y^{\frac14-\ve}\Bigr).
\end{align}
\end{prop}
Note that Theorem \ref{MAINTHM2} follows from Proposition \ref{BKX1BOUNDPROP}
together with Proposition \ref{AKBOUNDxi2zeroPROP} and the relations \eqref{MAINSTEP1}, \eqref{MAINTERM}.

\vspace{5pt}

\begin{proof}
The beginning of the proof of Proposition \ref{BKBOUNDPROP} carries over without changes;
the first difference is that in place of \eqref{BPART2} we get:
\begin{align}\label{eq:Bpartbound2}
\sum_{\substack{\smatr abcd\in[R]\\[2pt] c>0}}
e\bigl((a\vecq+c\vecr)\vecxi_1\bigr)
\int_0^\pi \wh f_R\left(\frac
ac-\frac{\sin2\theta}{2c^2y},\frac{\sin^2\theta}{c^2y},\theta;\veceta\right)
h\Bigl(-\frac dc+y\cot\theta\Bigr)\,\frac{y\,d\theta}{\sin^2\theta}.
\end{align}
Interchanging the roles of  $a$ and $d$ in the summation, we see that 
\eqref{eq:Bpartbound2} can be alternatively expressed as:
\begin{align}\label{eq:Bpartbound2alt}
\sum_{\substack{c\equiv c_0\mod N\\c>0}}e(c\vecr\vecxi_1)\int_0^\pi\sumtone
e(d\vecq\vecxi_1)\wh f_R\left(\frac
dc-\frac{\sin2\theta}{2c^2y},\frac{\sin^2\theta}{c^2y},\theta;\veceta\right)
\nu\Bigl(-\frac ac+y\cot\theta\Bigr)\,\frac{y\,d\theta}{\sin^2\theta},
\end{align}
where $\nu(x)=\sum_{n\in\Z}h(x+nN)$ (a function on $\R/N\Z$) and where 
$\sum^{(\widetilde 1)}$ is the same as $\sum^{(1)}$ (cf.\ p.\ \pageref{SUMONEDEF}) but using 
$\widetilde R:=\matr{d_0}{b_0}{c_0}{a_0}$ in place of $R=\matr{a_0}{b_0}{c_0}{d_0}$.
Now by Lemma \ref{lem:exp sum} we have, for any $c$ and $\theta$ appearing above:
\begin{align*}
&\sumtone
e\Bigl(d\vecq\vecxi_1\Bigr) \wh f_R\left(\frac
dc-\frac{\sin2\theta}{2c^2y},\frac{\sin^2\theta}{c^2y},\theta;\veceta\right)
\nu\Bigl(-\frac{a}{c}+y\cot\theta\Bigr)\\
&\ll \biggl\{\int_{\R}
\biggl|\wh
f_R\left(u,\frac{\sin^2\theta}{c^2y},\theta;\veceta\right)\biggr|\,du+
\int_{\R}\biggl|\frac{\partial^2}{\partial u^2}\wh
f_R\left(u,\frac{\sin^2\theta}{c^2y},\theta;\veceta\right)\biggr|\,
du\biggr\}\\
&\hspace{60pt}\times\biggl(\|\nu\|_{\L^1(\R/N\ZZ)}
\sum_{\ell\in\Z}\frac{(c,\lfloor cN\vecq\vecxi_1+\ell\rfloor)}{1+\ell^2}
+\|\nu''\|_{\L^1(\R/N\ZZ)}
\sigma(c)\sqrt c\biggr).
\end{align*}
Using \eqref{PARTBDECAYFNLEMREMEXPL2} and writing $v=\frac{\sin^2\theta}{c^2y}$,
for any $\ell\geq0$, %
we get:
\begin{align*}
\int_{\R}\biggl|\frac{\partial^\ell}{\partial u^\ell}
\wh f_R\left(u,\frac{\sin^2\theta}{c^2y},\theta;\veceta\right)\biggr|\,du
&\ll \|f\|_{\C_0^{3m+\ell}}\|\veceta\|^{-m}
v^{-\ell+\frac m2}\int_{\R}(u^2+v^2+1)^{-\frac m2}\,du
\\
&\ll \|f\|_{\C_0^{3m+\ell}}\|\veceta\|^{-m}
\min\biggl(\Bigl(\frac{|\sin\theta|}{c\sqrt y}\Bigr)^{-2\ell+m},\Bigl(\frac{|\sin\theta|}{c\sqrt y}\Bigr)^{2-2\ell-m}\biggr),
\end{align*}
and thus
\begin{align*}
\sum_{\ell\in\{0,2\}}\int_{\R}\biggl|\frac{\partial^\ell}{\partial u^\ell}
\wh f_R\left(u,\frac{\sin^2\theta}{c^2y},\theta;\veceta\right)\biggr|\,du
&\ll \|f\|_{\C_0^{3m+2}}\|\veceta\|^{-m}
\begin{cases}
\bigl(\frac{|\sin\theta|}{c\sqrt y}\bigr)^{m-4}&\text{if }\:\frac{|\sin\theta|}{c\sqrt y}\leq1
\\[5pt]
\bigl(\frac{|\sin\theta|}{c\sqrt y}\bigr)^{2-m}&\text{if }\:\frac{|\sin\theta|}{c\sqrt y}\geq1.
\end{cases}
\end{align*}
Using also
\begin{align*}
\int_0^\pi\left.\begin{cases}
\bigl(\frac{|\sin\theta|}{c\sqrt y}\bigr)^{m-4}&\text{if }\:\frac{|\sin\theta|}{c\sqrt y}\leq1
\\[5pt]
\bigl(\frac{|\sin\theta|}{c\sqrt y}\bigr)^{2-m}&\text{if }\:\frac{|\sin\theta|}{c\sqrt y}\geq1
\end{cases}
\right\}\,\frac{y\,d\theta}{\sin^2\theta}
\ll y\min\bigl((c\sqrt y)^{-1},(c\sqrt y)^{4-m}\bigr)
\\
\ll y^{3-\frac m2}c^{-1}(y^{-\frac12}+c)^{5-m}.
\end{align*}
we conclude that \eqref{eq:Bpartbound2} is %
\begin{align*}
\ll \|f\|_{\C_0^{3m+2}}\|\veceta\|^{-m}y^{3-\frac m2}\biggl\{
S_{1,0,0}(h)\sum_{c=1}^\infty c^{-1}(y^{-\frac12}+c)^{5-m}
\sum_{\ell\in\Z}\frac{(c,\lfloor cN\vecq\vecxi_1+\ell\rfloor)}{1+\ell^2}
\hspace{50pt}
\\
+S_{1,0,2}(h)\sum_{c=1}^\infty (y^{-\frac12}+c)^{5-m}\frac{\sigma(c)}{\sqrt c}\biggr\},
\end{align*}
and by Lemma \ref{PROOF1LEM1} and Lemma \ref{SIGMASUMLEM} (using $m\geq7$), this is 
\begin{align*}
\ll \|f\|_{\C_0^{3m+2}}S_{1,0,2}(h)\|\veceta\|^{-m}\biggl\{
\sum_{j=1}^\infty\frac1{j^2+y^{-1/2}j\langle j\vecq\vecxi_1\rangle}+y^{\frac14-\ve}\biggr\},
\end{align*}
Adding this bound over $R\in \overline\Gamma'/\Gamma'$ and $\veceta\in B_k$,
using $\sum_{\vecr\in\Z^k}\|\scmatr{\vecq}{\vecr}\|^{-m}\ll\|\vecq\|^{k-m}$
for every $\vecq\in\Z^k\setminus\{\bn\}$,
and $\sum_{\veceta\in\Z^{2k}\setminus\{\bn\}}\|\veceta\|^{-m}<\infty$ %
(these hold since $m>2k$), we obtain the bound in Proposition \ref{BKX1BOUNDPROP}.
This also completes the proof of Theorem \ref{MAINTHM2}.
\end{proof}
\newpage

\section{Application to a quantitative Oppenheim result}
\label{APPLSEC}

Our goal in this section is to prove Theorem \ref{MAINAPPLTHM},
by making Marklof's approach from \cite{MarklofpaircorrI} effective.
This will involve an application of Theorem \ref{MAINTHM1} at a key step.

\subsection{Set-up}

Let $\HH=\{\tau=u+iv\in\CC\col v>0\}$, the Poincar\'e upper half plane.
Let $k$ be a positive integer and let $\Sw(\RR^k)$
be the Schwartz space of functions on $\R^k$ which, together with their derivatives, decrease rapidly at infinity.
A central role in the approach of  %
\cite{MarklofpaircorrI}
is played by the \textit{Jacobi theta sum},
$\Theta_f(\tau,\phi;\vecxi)$.
It is defined by the following formula,
for any $f\in\Sw(\RR^k)$, $\tau=u+iv\in\HH$, $\phi\in\R$ and $\vecxi=\scmatr{\vecxi_1}{\vecxi_2}\in\R^{2k}$:
\begin{align}\label{ThetafDEF}
\Theta_f(\tau,\phi;\vecxi) = v^{k/4} 
\sum_{\vecm\in\ZZ^k} 
f_\phi( (\vecm-\vecxi_2) v^{1/2})\, e(\tfrac12\|\vecm-\vecxi_2\|^2 u 
+ \vecm\cdot \vecxi_1) ,
\end{align}
where, for $\phi$ in any interval $\nu\pi<\phi<(\nu+1)\pi$ ($\nu\in\ZZ$),
$f_\phi$ is given by the formula
$$
f_\phi(\vecw)= \int_{\RR^k} G_\phi(\vecw,\vecw')  f(\vecw') \,d\vecw' ,
$$
with the integral kernel
\begin{align}\label{GPHIdef}
G_\phi(\vecw,\vecw')=e\biggl(-\frac{k(2\nu+1)}8\biggr) %
|\sin\phi|^{-k/2}
e\left[\frac{\tfrac12(\|\vecw\|^2+\|\vecw'\|^2)
\cos\phi- \vecw\cdot\vecw'}{\sin\phi}\right],
\end{align}
while for $\phi=\nu\pi$ ($\nu\in\Z$) we have
$f_\phi(\vecw)=e(-\frac{k\nu}4)f((-1)^{\nu}\vecw)$.
The operators $U^\phi: f \mapsto f_\phi$ form a one-parameter group of unitary operators on $\L^2(\R^k)$;
in particular, $U^\phi \circ U^{\phi'}=U^{\phi+\phi'}$ for any $\phi,\phi'\in\R$.
Cf.\ \cite[Sec.\ 3-4]{MarklofpaircorrI}.   %

For any $f,g\in \Sw(\RR^k)$, the product
$\Theta_f\left(\tau,\phi;\vecxi\right)\overline{\Theta_g\left(\tau,\phi;\vecxi\right)}$ 
depends only on $\phi\mod2\pi$ and may thus be viewed as a function on $G=\SLR\ltimes\RR^{2k}$
through the Iwasawa parametrization (cf.\ \eqref{TFNIWASAWA})
\begin{align*}
(\tau,\phi,\vecxi)\mapsto
\left(\matr 1u01\matr{\sqrt v}00{1/\sqrt v}\matr{\cos\phi}{-\sin\phi}{\sin\phi}{\cos\phi},\vecxi\right),
\qquad \text{where }\:\tau=u+iv.
\end{align*}
By \cite[Proposition 4.9]{MarklofpaircorrI}, 
this function %
$\Theta_f\overline\Theta_g\in\C^\infty(G)$
is in fact left $\Gamma^k$ invariant, where
\begin{align}\label{GAMMAkdef}
\Gamma^k=\left\{\left(\matr abcd,\cmatr{ab\vecs}{cd\vecs}+\vecm\right)\col
\matr abcd\in \SLZ,\: \vecm\in \ZZ^{2k}\right\}
\end{align}
with $\vecs:=\trans(1/2,\ldots,1/2)\in \RR^k$.
The group $\Gamma^k$ is a finite index subgroup of $\SLZ\ltimes (\frac12\ZZ)^{2k}$,
and contains $\Gamma_\theta\ltimes \ZZ^{2k}$ as an index $3$ subgroup,
where $\Gamma_\theta$ is the theta group, i.e.
\begin{align*}
\Gamma_\theta=\left\{\matr abcd\in\SL(2,\Z)\col ab\equiv cd\equiv0\mod 2\right\}.
\end{align*}
Cf.\ \cite[Lemmata 4.11, 4.12]{MarklofpaircorrI}.

For the proof of Theorem \ref{MAINAPPLTHM},
we will eventually specialize to $k=2$:
The starting point for the method developed in \cite{MarklofpaircorrI} 
is the following identity\footnote{Cf.\ \cite[Sec.\ 2.3]{MarklofpaircorrI},
where the identity \eqref{MARKLOFkeyid}
appears in the special case when $f(\vecx)\equiv\psi_1(\|\vecx\|^2)$, $g(\vecx)\equiv\psi_2(\|\vecx\|^2)$ 
and using a slightly different notation than in \eqref{MARKLOFkeyid}.
Note that 
we write $\hh(s)=\int_{\R}h(u)e(-su)\,du$
in \eqref{MARKLOFkeyid},
in line with previous definitions in our paper,
whereas a different normalization of $\hh$ is used in
\cite[p.\ 423(top)]{MarklofpaircorrI}.},
valid for any $f,g\in\scrS(\R^2)$, $h\in \L^1(\R)$, $T>0$ and $\vecxi_2\in\R^2$:
\begin{align}\notag
\int_{\R}\Theta_f\biggl(u+T^{-2}i, 0 ;\cmatr{\bn}{\vecxi_2}\biggr)\,
\overline{\Theta_g\biggl(u+T^{-2}i, 0 ;\cmatr{\bn}{\vecxi_2}\biggr)}\, h(u)\,du
\hspace{160pt}
\\\label{MARKLOFkeyid}
=\frac1{T^2}\sum_{\vecm_1\in\Z^2}\sum_{\vecm_2\in\Z^2}
f\bigl(T^{-1}(\vecm_1-\vecxi_2)\bigr)\,\overline{g\bigl(T^{-1}(\vecm_2-\vecxi_2)\bigr)}
\,\,\hh\biggl(-\tfrac12\, Q\hspace{-3pt}\cmatr{\vecm_1}{\vecm_2}\biggr),
\end{align}
where $Q$ is the inhomogeneous quadratic form on $\R^4$ given by \eqref{Qdef}
with $\vecxi_2=\cmatr{\alpha}{\beta}\in\R^2$, i.e., %
\begin{align}\label{MARKLOFkeyidpf1}
Q\hspace{-2pt}\cmatr{\vecx_1}{\vecx_2}=\|\vecx_1-\vecxi_2\|^2-\|\vecx_2-\vecxi_2\|^2,\qquad\forall \vecx_1,\vecx_2\in\R^2.
\end{align}
The formula \eqref{MARKLOFkeyid} follows by replacing $\Theta_f$ and $\Theta_g$ by their defining sums
(cf.\ \eqref{ThetafDEF})
and changing the order of summation and integration. %

The key step in \cite{MarklofpaircorrI}
is then to determine the limit of the left hand side of \eqref{MARKLOFkeyid} as $T\to\infty$,
by %
using the invariance properties of the function
$\Theta_f\overline{\Theta_g}$
and an equidistribution result as in Theorem \ref{INEFFECTIVETHM} above 
(with $\vecxi_1=\bn$);
this is where we will apply our effective result, Theorem \ref{MAINTHM1}, instead.
A central difficulty in \cite{MarklofpaircorrI}
comes from the fact that the theta functions $\Theta_f,\Theta_g$ are unbounded;
thus one needs to truncate the function 
$\Theta_f\overline{\Theta_g}$ in the cusp before the equidistribution result can be applied,
and then bound %
the error caused by the truncation.
In fact it turns out that one picks up an explicit extra contribution from the
part of the integral in \eqref{MARKLOFkeyid} over a tiny interval $|u|\ll T^{-(1+\ve)}$,
whereas the error caused by the truncation for the remaining part of the integral can be
proved to be appropriately small, 
provided that $\vecxi_2$ is Diophantine.
The treatment of these matters in 
\cite{MarklofpaircorrI} is already in principle effective,
and so %
our work concerning the truncation error 
will essentially only consist in
keeping more explicit track on how the bounds in 
\cite{MarklofpaircorrI} depends on various parameters;
cf.\ in particular Proposition~\ref{opp:upper bound} below.
Also, for the application of Theorem \ref{MAINTHM1},
we require precise bounds on derivatives of the function $\Theta_f\overline{\Theta_g}$;
this is worked out in Lemma \ref{lem:betatheta} below.

\subsection{Bounds for the derivatives of $\Theta_f\overline{\Theta_g}$}

Although we will eventually specialize to $k=2$,
we will consider a general $k\in\Z^+$ as long as this causes no extra work.
We will use the same notation $S_{p,a,n}$ as introduced in the introduction
also for the corresponding weighted Sobolev norm
of a function $f\in\C^n(\R^k)$ with $k\geq2$; namely
\begin{align}\label{genSobdef}
S_{p,a,n}(f)=\sum_{|\gamma|\leq n}\|(1+\|\vecx\|)^a\,\partial^\gamma f(\vecx) \|_{\L^p}.
\end{align}
Here we use standard multi-index notation, i.e.\ $\gamma$ runs through $k$-tuples of
nonnegative integers, %
$|\gamma|=\gamma_1+\ldots+\gamma_k$ and $\partial^\gamma=\partial^{\gamma_1}_{x_1}\cdots\partial^{\gamma_k}_{x_k}$.
It will be convenient to work with the Sobolev norms $S_{2,a,a}$ on functions in $\C^a(\R^k)$,
and we introduce the notation $\|\cdot\|_{\L_a^2}$ for these.
Thus for any integer $a\geq0$ and $f\in\C^a(\R^k)$,
\begin{align}\label{eq:Spandef}
\|f\|_{\L^2_a}:=S_{2,a,a}(f)
=\sum_{|\beta|\leq a}\|(1+\|\vecx\|)^a\partial^\beta f(\vecx)\|_{\L^2}.
\end{align}
We note that (cf., e.g., \cite[Ch.\ 8.1, Exc.\ 1]{gF99})
\begin{align}\label{LN2BASICFACT1}
\|f\|_{\L_a^2}\asymp \sum_{|\beta|\leq a}\sum_{|\beta'|\leq a}\|\vecx^{\beta'}\partial^\beta f(\vecx)\|_{\L^2}
\asymp \sum_{|\beta|\leq a}\sum_{|\beta'|\leq a}\|\partial^\beta(\vecx^{\beta'} f(\vecx))\|_{\L^2}.
\end{align}
Combining this relation with the Plancherel Theorem we also have
\begin{align} \label{eq:Plancherel}
\|\widehat f\|_{\L_a^2}\asymp\|f\|_{\L_a^2},
\end{align}
where $\widehat f(\vecy)=\int_{\RR^k} f(\vecx)e(-\vecx\vecy)\,d\vecx$ is the Fourier transform of $f$. 
In \eqref{LN2BASICFACT1} and \eqref{eq:Plancherel}, the implied constants only depend on $k$ and $a$.

Given $f\in \scrS(\RR^k)$, we view $f_\phi(\vecw)$ as a function on the space 
$\RR^{k+1} $, given by the coordinates $(\vecw,\phi)$.
Thus $\partial^\beta f_\phi(\vecw)$ for $\beta\in(\Z_{\geq0})^{k+1}$ denotes
$\partial_{w_1}^{\beta_1}\cdots\partial_{w_k}^{\beta_k}\partial_\phi^{\beta_{k+1}} f_\phi(\vecw)$.
The following lemma corresponds to
\cite[lemma 4.3]{MarklofpaircorrI}, but extended to arbitrary derivatives of $f_\phi$ and with the
implied constant made more precise.

\begin{lem}\label{schwartz}
Let $A\in\Z_{\geq0}$, $\beta\in(\Z_{\geq0})^{k+1}$ and $a\in\Z$, $a>A+\frac k2+4|\beta|$.
Then for any $f\in\scrS(\RR^k)$, $\vecw\in\R^k$ and $\phi\in\R$,
\begin{align*}
|\partial^\beta f_\phi(\vecw)| \ll_{A,\beta} \|f\|_{\L^2_a}(1+\|\vecw\|)^{-A}.
\end{align*}
\end{lem}
\begin{proof}
For $\phi$ in any interval $\nu\pi+\frac1{100}<\phi<(\nu+1)\pi-\frac1{100}$, $\nu\in\ZZ$, we use
\begin{align*}
\partial^\beta f_\phi(\vecw)= \int_{\RR^k} \bigl(\partial^\beta G_\phi(\vecw,\vecw')\bigr)  f(\vecw') \,d\vecw' ,
\end{align*}
with $G_\phi(\vecw,\vecw')$ as in \eqref{GPHIdef}, and with $\partial^\beta$ acting on the $k+1$ variables
$w_1,\ldots,w_k,\phi$.
One proves by induction that
\begin{align*}
\partial^\beta G_\phi(\vecw,\vecw')=G_\phi(\vecw,\vecw')\frac{P_\beta(\vecw,\vecw',\sin\phi,\cos\phi)}{(\sin\phi)^{2|\beta|}},
\end{align*}
where $P_\beta$ is a polynomial in $2k+2$ variables,
with complex coefficients which only depend on $k$ and $\beta$,
and only containing terms $w_1^{\alpha_1}\cdots w_k^{\alpha_k}{w_1'}^{\alpha_{k+1}}\cdots {w_k'}^{\alpha_{2k}}
(\sin\phi)^{\alpha_{2k+1}}(\cos\phi)^{\alpha_{2k+2}}$
with $\sum_1^{2k}\alpha_j\leq2|\beta|$.
Integrating by parts $n\geq0$ times with respect to $w_j'$ for some $j$, it follows that
\begin{align*}
&\partial^\beta f_\phi(\vecw)= \frac{e\bigl(-\frac18k(2\nu+1)\bigr)}{|\sin\phi|^{k/2}}
\Bigl(\frac{\sin\phi}{2\pi i w_j}\Bigr)^n
\int_{\R^k}K(\vecw,\vecw',\phi)\,e\biggl[-\frac{\vecw\cdot\vecw'}{\sin\phi}\biggr]\,d\vecw'
\end{align*}
where
\begin{align*}
K(\vecw,\vecw',\phi) &=\Bigl(\frac{\partial}{\partial w_j'}\Bigr)^n
\biggl(e\biggl[\frac{\tfrac12(\|\vecw\|^2+\|\vecw'\|^2)\cos\phi}{\sin\phi}\biggr]
\frac{P_\beta(\vecw,\vecw',\sin\phi,\cos\phi)}{(\sin\phi)^{2|\beta|}}f(\vecw')\biggr)
\end{align*}
and so
\begin{align*}
|K(\vecw,\vecw',\phi)|\ll_{\beta,n} 
(1+\|\vecw\|)^{2|\beta|}\sum_{\ell=0}^{n}(1+\|\vecw'\|)^{2|\beta|+\ell}
\,\,\biggl|\Bigl(\frac{\partial}{\partial w_j'}\Bigr)^{n-\ell} f(\vecw')\biggr|
\end{align*}
for all $\vecw,\vecw'\in\R^k$ and $\phi\in(\nu\pi+\frac1{100},(\nu+1)\pi-\frac1{100})$.
If $\|\vecw\|\geq1$, then we apply the above with $n=A+2|\beta|$ and $j$ being the index for which
$|w_j|=\max(|w_1|,\ldots,|w_k|)$;
if $\|\vecw\|<1$ then we instead use $n=0$.
The desired bound follows using the Cauchy-Schwarz inequality combined with the fact that 
$(1+\|\vecw'\|)^{-\frac{k+\ve}2}\in\L^2(\R^k)$  %
for any $\ve>0$.

To treat the remaining values of $\phi$, we use the fact that
$f_{\phi+\frac\pi2}=e^{-\frac14 \pi ki}U^\phi\widehat f$;
hence by what we have already proved, $|\partial^\beta f_\phi(\vecw)| \ll\|\widehat f\|_{\L^2_a}(1+\|\vecw\|)^{-A}$
for $\phi$ in any interval $(\nu-\frac12)\pi+\frac1{100}<\phi<(\nu+\frac12)\pi-\frac1{100}$, $\nu\in\Z$,
and the desired bound follows using \eqref{eq:Plancherel}.
\end{proof}

Using Lemma \ref{schwartz}, we now obtain bounds on arbitrary derivatives of the function
$\Theta_f\overline\Theta_g\in\C^\infty(G)$.
Recall that we write $\sum_{\ord(D)\leq m}$ to denote a sum over all monomials $D$ of degree $\leq m$
in the fixed basis $X_1,\ldots,X_{3+2k}$ of $\ig$ (cf.\ \eqref{FIXEDBASIS}).

\begin{lem}
\label{lem:betatheta}
Let $f,g\in\scrS(\RR^k)$.
Let $m$ and $a$ be integers satisfying $m\geq0$ and $a>\frac32k+6m+1$.
Then for any $(\tau,\phi,\vecxi)\in G$ with $v=\tim \tau\geq\frac12$,
\begin{align}\label{eq:thetaderivbound}
\sum_{\ord(D)\leq m}\bigl|\bigl(D(\Theta_f\overline\Theta_g)\bigr)(\tau,\phi;\vecxi)\bigr|
\ll_m \|f\|_{\L^2_a}\|g\|_{\L^2_a}\,v^{m+\frac12k}.
\end{align}
Next let $A$ and $a$ be integers satisfying $A\geq1$ and $a>\frac32k+2A$.
Then for any $(\tau,\phi,\vecxi)\in G$ with $v=\tim \tau\geq\frac12$,
\begin{align}\label{eq:vecmsum}
&\biggl|(\Theta_f\overline\Theta_g)(\tau,\phi;\vecxi)
- v^{k/2} \sum_{\vecm\in\ZZ^k}f_\phi((\vecm-\vecxi_2) v^{1/2}) \overline{g_\phi((\vecm-\vecxi_2) v^{1/2})}\biggr|
\ll_A \|f\|_{\L^2_a}\|g\|_{\L^2_a}v^{-A},
\end{align}
and if furthermore $\vecxi_2\in\vecn+[-\tfrac12,\tfrac12]^k$, $\vecn\in\ZZ^k$, then
\begin{align}\label{eq:m=0term}
&\biggl|(\Theta_f\overline\Theta_g)(\tau,\phi;\vecxi)
- v^{k/2} f_\phi((\vecn-\vecxi_2) v^{1/2}) \overline{g_\phi((\vecn-\vecxi_2) v^{1/2})}\biggr|
\ll_A \|f\|_{\L^2_a}\|g\|_{\L^2_a}v^{-A}.
\end{align}
\end{lem}
\begin{proof}
Recall that we write $\tau=u+iv$. We have
\begin{align*}
&\Theta_f\overline\Theta_g\left(\tau,\phi;\vecxi\right)
= v^{k/2} 
\sum_{\vecm_1,\vecm_2\in\ZZ^k} 
f_\phi( (\vecm_1-\vecxi_2) v^{1/2})\overline{g_\phi((\vecm_2-\vecxi_2) v^{1/2})}\, 
\\[-5pt]
&\hspace{185pt}\times e\bigl(\tfrac12(\|\vecm_1-\vecxi_2\|^2-\|\vecm_2-\vecxi_2\|^2) u 
+ (\vecm_1-\vecm_2)\cdot \vecxi_1\bigr)
\\[10pt]
&= v^{k/2} 
\sum_{\vecm_1,\vecm_2\in\ZZ^k}  
f_\phi( (\vecm_1-\vecxi_2) v^{1/2})\overline{g_\phi((\vecm_2-\vecxi_2) v^{1/2})}\, 
e\bigl(\tfrac12(\vecm_1-\vecm_2)((\vecm_1+\vecm_2-2\vecxi_2) u +2 \vecxi_1)\bigr)\\
&= \sum_{\vecm,\vecm'\in\ZZ^k}  v^{k/2} 
f_\phi( (\vecm-\vecxi_2) v^{1/2})\overline{g_\phi((\vecm-\vecm'-\vecxi_2) v^{1/2})}\, 
e\left(\vecm'\cdot((\vecm-\tfrac12\vecm'-\vecxi_2) u 
+\vecxi_1)\right)\\
&:=\sum_{\vecm,\vecm'\in\ZZ^k} F_{\vecm,\vecm'}(\tau,\phi;\vecxi),
\end{align*}
say. Note that $F_{\vecm,\vecnull}(\tau,\phi;\vecxi)
=v^{k/2} f_\phi((\vecm-\vecxi_2) v^{1/2}) \overline{g_\phi((\vecm-\vecxi_2) v^{1/2})} $.

\vspace{2pt}

In the $(u+iv,\phi;\vecxi)$-coordinates with $\vecxi=\trans(\xi_1,\ldots,\xi_{2k})$,
the Lie derivatives $X_1,\ldots,X_{3+2k}$ are given by
\begin{align*}
&X_1=v(\cos 2\phi)\partial_u-v(\sin 2\phi)\partial_v-(\sin\phi)^2\partial_\phi;
\\
&X_2=v(\cos 2\phi)\partial_u-v(\sin 2\phi)\partial_v+(\cos\phi)^2\partial_\phi;
\\
&X_3=2v(\sin 2\phi)\partial_u+2v(\cos 2\phi)\partial_v+(\sin 2\phi)\partial_\phi;
\\
&\begin{cases}X_{3+\ell}=\bigl(\frac{v\cos\phi+u\sin\phi}{\sqrt v}\bigr)\partial_{\xi_\ell}
+\bigl(\frac{\sin\phi}{\sqrt v}\bigr)\partial_{\xi_{k+\ell}}
\\
X_{3+k+\ell}=\bigl(\frac{-v\sin\phi+u\cos\phi}{\sqrt v}\bigr)\partial_{\xi_\ell}
+\bigl(\frac{\cos\phi}{\sqrt v}\bigr)\partial_{\xi_{k+\ell}}
\end{cases}
\qquad(\ell\in\{1,2,\ldots,k\}).
\end{align*}
(The formulas for $X_1,X_2,X_3$ are standard and may e.g.\ easily be derived using
the formulas in the proof of \cite[Lemma 6.1]{SASL} and $X_2-X_1=\smatr0{-1}10=\partial_\phi$.
Regarding $X_{3+\ell}$ and $X_{3+k+\ell}$, cf.\ \eqref{DECAYTFNFROMCMNORMLEMpf1}.)
Using the automorphy of $\Theta_f\overline\Theta_g$, it is enough to prove \eqref{eq:thetaderivbound} when 
$|u|\leq\frac12$ and $\|\vecxi\|\ll 1$.
Using the above formulas we then get, for any monomial $D$ in $X_1,\ldots,X_{3+2k}$
with $d_1$ factors in $\{X_1,X_2,X_3\}$ and $d_2$ factors in $\{X_4,\ldots,X_{3+2k}\}$:
 \begin{align}
 \label{eq:Thetabound}
  |D F_{\vecm,\vecm'}(\tau,\phi;\vecxi )|&\ll v^{d_1+\frac12d_2}\sum_{|\alpha|\leq d_1}\sum_{|\beta|\leq d_2} 
\left|\partial^{\alpha_1}_u\partial^{\alpha_2}_v\partial^{\alpha_3}_\phi
\partial^{\beta_1}_{\xi_1}\cdots\partial^{\beta_{2k}}_{\xi_{2k}}
 F_{\vecm,\vecm'}(u+iv,\phi;\vecxi )\right| ,
 \end{align}
where $\alpha$ runs through multi-indices in $(\Z_{\geq0})^3$
and $\beta$ runs through multi-indices in $(\Z_{\geq0})^{2k}$.
Next, from the definition of $F_{\vecm,\vecm'}(\tau,\phi;\vecxi)$,
by a standard computation, we obtain:
\begin{align*}
\left|\partial^{\alpha_1}_u\partial^{\alpha_2}_v\partial^{\alpha_3}_\phi
\partial^{\beta_1}_{\xi_1}\cdots\partial^{\beta_{2k}}_{\xi_{2k}}
 F_{\vecm,\vecm'}(u+iv,\phi;\vecxi )\right|
\ll\bigl(1+\|\vecm\|+\|\vecm'\|\bigr)^{2|\alpha|+|\beta|}v^{(k+|\beta|)/2}
\hspace{50pt}
\\
\times\sum_{|\beta'|+|\beta''|\leq|\alpha|+|\beta|}\Bigl|
(\partial^{\beta'}f_\phi)((\vecm-\vecxi_2) 
v^{1/2})(\partial^{\beta''}g_\phi)((\vecm-\vecm'-\vecxi_2) v^{1/2})\Bigr|,
\end{align*}
where $\beta'$ and $\beta''$ run through multi-indices in $(\Z_{\geq0})^{k+1}$,
with $\partial^{\beta'}$ and $\partial^{\beta''}$ having the same meaning as in Lemma \ref{schwartz}.
Applying now Lemma \ref{schwartz}, with any fixed integers $A$ and $a$ subject to $A\geq0$ and 
$a>A+\frac k2+4(d_1+d_2)$, we get
\begin{align}\notag
|D F_{\vecm,\vecm'}(\tau,\phi;\vecxi )|
\ll \|f\|_{\L_a^2}\|g\|_{\L_a^2} v^{d_1+d_2+\frac12k}\bigl(1+\|\vecm\|+\|\vecm'\|\bigr)^{2d_1+d_2}
\hspace{100pt}
\\\notag
\times (1+\|\vecm-\vecxi_2\|v^{1/2})^{-A}(1+\|\vecm-\vecm'-\vecxi_2\|v^{1/2})^{ -A }.
\end{align}
Let $d=d_1+d_2=\deg(D)$.
Using $1+\|\vecm\|+\|\vecm'\|\ll (1+\|\vecm\|)(1+\|\vecm-\vecm'\|)$ we obtain:
\begin{align*}
\bigl|\bigl(D(\Theta_f\overline\Theta_g)\bigr)(\tau,\phi;\vecxi)\bigr|
\ll\|f\|_{\L_a^2}\|g\|_{\L_a^2} v^{d+\frac12k}
\sum_{\vecm,\vecm'}(1+\|\vecm\|)^{2d-A}(1+\|\vecm-\vecm'\|)^{2d-A}.
\end{align*}
Taking here $A=2d+k+1$ we obtain \eqref{eq:thetaderivbound}.

For the remaining bounds, we apply Lemma \ref{schwartz} with fixed integers $A'\geq0$ and $a>A'+\frac k2$.
For any $(\vecm,\vecm')$ satisfying $\|\vecm-\vecxi_2\|+\|\vecm-\vecm'-\vecxi_2\|\geq\frac12$ this leads to
\begin{align*}
\bigl|F_{\vecm,\vecm'}(\tau,\phi;\vecxi )\bigr|
&=v^{k/2}\Bigl|f_\phi( (\vecm-\vecxi_2) v^{1/2})g_\phi((\vecm-\vecm'-\vecxi_2) v^{1/2})\Bigr|
\\
&\ll \|f\|_{\L_a^2}\|g\|_{\L_a^2}v^{(k-A')/2}
(1+\|\vecm-\vecxi_2\|)^{-A'}(1+\|\vecm-\vecm'-\vecxi_2\|)^{-A'}.
\end{align*}
In particular, this holds for all $(\vecm,\vecm')$ with $\vecm'\neq\bn$.
Hence, taking $A'=k+2A$, we obtain \eqref{eq:vecmsum}.
Similarly, if $\vecxi_2\in\vecn+[-\tfrac12,\tfrac12]^k$ then we note that 
$\|\vecm-\vecxi_2\|+\|\vecm-\vecm'-\vecxi_2\|\geq\frac12$ holds for all $(\vecm,\vecm')$
except $(\vecm,\vecm')=(\vecn,\vecnull)$,
and we thus obtain \eqref{eq:m=0term}.
\end{proof}

\subsection{Bounds on the truncation error}

Let us fix, once and for all, a $\C^\infty$ function $g_1:\R_{>0}\to[0,1]$
satisfying $g_{1|(0,1]}\equiv0$ and $g_{1|[2,\infty)}\equiv1$.
For any $Y\geq1$ we then define $g_Y:\R_{>0}\to[0,1]$, $g_Y(y):=g_1(y/Y)$,
so that $g_{Y|(0,Y]}\equiv0$ and $g_{Y|[2Y,\infty)}\equiv1$.
Next, we define the function $\scrX_Y:G\to\R_{\geq0}$ through
\begin{align}\label{XYdef}
\scrX_Y(\tau,\phi;\vecxi)=\scrX_Y(\tau)= 
\sum_{\gamma\in (\pm\oGamma'_\infty)\backslash\SL(2,\Z)}
g_Y\bigl(\tim\gamma\tau\bigr).
\end{align}
Note here that $(\pm\oGamma_\infty')=\bigl\{\pm\smatr1n01\col n\in\Z\bigr\}$;
cf.\ \eqref{GAMMAinfdef}.
The function $\scrX_Y$ is smooth and $\SL(2,\Z)$-invariant. %
For any $\tau\in\HH$, there is (since $Y\geq1$) at most one term in the sum in 
\eqref{XYdef} which gives a non-zero contribution.
In particular, $\scrX_Y(g)\in[0,1]$ for all $g\in G$.
Also, in terms of the cuspidal height function $\scrY$, 
cf.\ \eqref{CHdef},
we have $\scrX_Y(g)=0$ whenever $\scrY(g)\leq Y$ and $\scrX_Y(g)=1$ whenever $\scrY(g)\geq2Y$.
\begin{lem}\label{XYDERlem}
For any $Y\geq1$ and any monomial $D$ in $X_1,\ldots,X_{3+2k}$ of degree $\leq m$,
$D\scrX_Y$ is a bounded function on $G$ with 
$\|D\scrX_Y\|_{\L^\infty}\ll_m1$.
\end{lem}
\begin{proof}
Since $\scrX_Y$ (and thus $(D\scrX_Y)$) is $\SL(2,\Z)$-invariant, it suffices to consider points $(\tau,\phi,\vecxi)$ with
$\tau=u+iv$ belonging to the standard fundamental domain for $\SL(2,\Z)$,
i.e.\ $|u|\leq\frac12$ and $|v|\geq1$.
Then we may in fact assume $v>1$, since otherwise $(\tau,\phi,\vecxi)$ is not in the support of $\scrX_Y$.
However, for $v>1$ we have:
\begin{align*}
\scrX_Y(\tau,\phi;\vecxi)=g_Y(v)=g_1\biggl(\tim\matr{Y^{-1/2}}00{Y^{1/2}}(\tau)\biggr).
\end{align*}
Since $D$ is left invariant, this implies that
$\|D\scrX_Y\|_{\L^\infty}=\|D\tscrX_1\|_{\L^\infty}$,
where $\tscrX_1$ is the function $\tscrX_1:G\to[0,1]$,
$(\tau,\phi,\vecxi)\mapsto g_1(\tim\tau)$.
This $\L^\infty$-norm is clearly finite, 
and independent of $Y$.
\end{proof}

For 
$\vecxi=\scmatr{\vecxi_1}{\vecxi_2}\in\R^{2k}$ and $\gamma=\smatr{*}{*}{c}{d}\in \SLZ $,
we introduce the short-hand notation $\vecxi_\gamma:=c\vecxi_1+d\vecxi_2$.
We also write $v_\gamma:=\tim \gamma\tau$ when $\tau=u+iv\in\HH$.
Given $Y\geq1$ and $f\in\C(\R^k)$ with $S_{\infty,A,0}(f)<\infty$ for some $A>k$, 
we define the function $F_{f,Y}:G\to\CC$ by
(cf.\ \cite[6.2]{MarklofpaircorrI}):
\begin{align}\label{FYdef}
F_{f,Y}(\tau,\phi;\vecxi)= 
F_{f,Y}(\tau;\vecxi):= 
\sum_{\gamma\in\oGamma'_\infty\backslash\SL(2,\ZZ)}
\sum_{\vecm\in\ZZ^k} 
f\left( (\vecxi_\gamma+\vecm) v_\gamma^{1/2} \right)\,
v_\gamma^{k/2} \, g_Y(v_\gamma).
\end{align}
This series is absolutely convergent,
and $F_{f,Y}$ is left $\oGamma$ invariant.
In fact we will only use
$F_{f,Y}$ for functions $f\geq0$;
then of course $F_{f,Y}\geq0$.

As in \cite[6.4]{MarklofpaircorrI}, \cite[6.3]{MarklofpaircorrII}, 
we have the explicit formula
\begin{align}\notag
F_{f,Y}(\tau;\vecxi)=
\sum_{\vecm\in\ZZ^k} \big\{ f\big( (\vecxi_2+\vecm) v^{1/2} \big) 
+ f\big( (-\vecxi_2+\vecm) v^{1/2} \big) \big\}
v^{k/2} g_Y(v) 
\hspace{70pt}
\\ \label{FYexpansion}
+
\sum_{\vecm\in\ZZ^k} 
\bigg\{ f\bigg((\vecxi_1 + \vecm) \frac{v^{1/2}}{|\tau|} \bigg) +
f\bigg((-\vecxi_1 + \vecm) \frac{v^{1/2}}{|\tau|} \bigg) \bigg\}
\,\frac{v^{k/2}}{|\tau|^{k}} g_Y\Bigl(\frac{v}{|\tau|^2}\Bigr) 
\hspace{30pt}
\\ \notag
+
\sum_{\substack{(c,d)\in\ZZ^2\\ \gcd(c,d)=1 \\ c,d\neq 0}}
\sum_{\vecm\in\ZZ^k} 
f\bigg((c \vecxi_1 + d \vecxi_2 + \vecm) \frac{v^{1/2}}{|c\tau+d|} \bigg)
\frac{v^{k/2}}{|c\tau+d|^{k}} g_Y\Bigl(\frac{v}{|c\tau+d|^2}\Bigr).
\end{align}
The following lemma shows that for an appropriate choice of $f^*$, the function $F_{f^*,Y}$
controls the error when truncating $\Theta_f\overline{\Theta_g}$ at height $\asymp Y$.
\begin{lem}\label{thetaboundLEM}
Let $f,g\in\scrS(\R^k)$ and let $A$ and $a$ be integers satisfying $A\geq1$ and $a>\frac32k+2A$.
Set $f^*(\vecw)=\sup_{\phi\in\R}\bigl|f_\phi(\tfrac12\vecw)g_\phi(\tfrac12\vecw)\bigr|$.
Then for any $Y\geq1$,
\begin{align}\label{thetaboundLEMRES}
\scrX_Y(\tau)\,\,\bigl|(\Theta_f\overline\Theta_g)(\tau,\phi;\vecxi)\bigr|
\leq F_{f^*,Y}(\tau;2\vecxi)+O_A\Bigl(\|f\|_{\L^2_a}\|g\|_{\L^2_a}Y^{-A}\Bigr),
\qquad
\forall (\tau,\phi,\vecxi)\in G.
\end{align}
\end{lem}
(Here $F_{f^*,Y}(\tau;2\vecxi)$ %
is well-defined, since $S_{\infty,A',0}(f^*)<\infty$ for all $A'>0$ by Lemma \ref{schwartz}.)
\begin{proof}
(Cf.\ \cite[8.4.3]{MarklofpaircorrI}.)
For any $(\tau,\phi,\vecxi)$, with 
$\tau$ lying in the standard fundamental domain $\scrF$
for $\SL(2,\Z)$, $\scrF=\{\tau=u+iv\in\HH\col |u|\leq\frac12,\:|\tau|\geq1\}$,
it follows from the definition of $f^*$ together with \eqref{eq:vecmsum} in Lemma \ref{lem:betatheta} that
\begin{align*}
\bigl|(\Theta_f\overline\Theta_g)(\tau,\phi;\vecxi)\bigr|
\leq v^{k/2}\sum_{\vecm\in\ZZ^k}f^*\bigl(2(\vecm-\vecxi_2) v^{1/2}\bigr)
+O_A\Bigl(\|f\|_{\L^2_a}\|g\|_{\L^2_a}v^{-A}\Bigr).
\end{align*}
Multiplying this inequality with $g_Y(v)$ and comparing with \eqref{FYexpansion},
we obtain that \eqref{thetaboundLEMRES} holds for all $(\tau,\phi,\vecxi)\in G$ with $\tau\in\scrF$, %
since $\scrX_Y(\tau)=g_Y(v)$ for all such $\tau$.
But both sides in \eqref{thetaboundLEMRES} are functions of $(\tau,\phi,\vecxi)\in G$ which are 
$\Gamma^k$ left invariant
(for the function $\hF_{f^*,Y}(\tau;\vecxi):=F_{f^*,Y}(\tau;2\vecxi)$ this is noted in
\cite[7.5-6]{MarklofpaircorrI}, \cite[6.9-10]{MarklofpaircorrII});
hence the inequality holds for \textit{all} $(\tau,\phi,\vecxi)\in G$.
\end{proof}

The following lemma is a more explicit version of
\cite[lemma 6.5]{MarklofpaircorrII}.
Recall that our $\kappa$ corresponds to ``$\kappa-1$'' in \cite{MarklofpaircorrII}.

\begin{lem}\label{lem:bound upper}
Let $A>k$.
Then for any $[\kappa;c]$-Diophantine $\vecalf\in\R^k$, and any $D,T\geq1$,
\begin{align}\label{lem:bound upperRES}
\sum_{d=1}^D \sum_{\vecm\in\ZZ^k} \bigl(1+T\|d\vecalf+\vecm\|\bigr)^{-A}
\ll_{k,A}
\begin{cases}
D^{A\kappa+1}(cT)^{-A}
&\text{if }\:D^{\kappa+A^{-1}}\leq cT
\\
1&\text{if }\:D^{\kappa}\leq cT\leq D^{\kappa+A^{-1}}
\\
D(cT)^{-1/\kappa}&\text{if }\:cT\leq D^{\kappa}.
\end{cases}
\end{align}
\end{lem}
\begin{proof}
Since $\vecalf$ is  $[\kappa;c]$-Diophantine,
$\|d\vecalf+\vecm\|\geq cd^{-\kappa}$ for all integers $d\geq1$ and $\vecm\in\Z^k$.
Also for each fixed $d$ there is at most one $\vecm\in\Z^k$
in the box $-d\vecalf+(-\frac12,\frac12)^k$,
and in particular there is at most one $\vecm\in\Z^k$ with $\|d\vecalf+\vecm\|<\frac12$.
Hence for $d\in\{1,\ldots,D\}$,
\begin{align*}
\sum_{\vecm\in\ZZ^k} \bigl(1+T\|d\vecalf+\vecm\|\bigr)^{-A}
\ll_{k,A} (1+Tcd^{-\kappa})^{-A}+T^{-A}\ll_A \Bigl(\frac{D^{\kappa}}{cT}\Bigr)^A, %
\end{align*}
where in the last inequality we use the fact that $D^{\kappa}/c>1$
(note that $\vecalf$ being  $[\kappa;c]$-Diophantine implies $c\leq\frac12$).
Adding the above bound over $d=1,\ldots,D$, we obtain that the left hand side of
\eqref{lem:bound upperRES} is $\ll D^{A\kappa+1}(cT)^{-A}$.

To prove another bound on the same sum,
for any fixed $b\in\Z$, 
we start by considering
the set 
\begin{align}\label{lem:bound upperPF2}
M_b=\bigl\{T((b+d)\vecalf+\vecm)\col d\in\Z,\:0\leq d \leq (cT)^{1/\kappa},\:\vecm\in\ZZ^k\bigr\}.
\end{align}
The distance between any two distinct points in this set is bounded from below by 
\begin{align*}
\min\Bigl\{
T\| q\vecalf+\vecn \| 
\col
q\in\Z,\:\vecn\in\ZZ^k,\:|q| \leq (cT)^{1/\kappa},\:
[q\neq0\text{ or }\vecn\neq\bn]\Bigr\}
\hspace{40pt}
\\
\geq \min\Bigl(T,
\min_{0<q \leq (cT)^{1/\kappa}} Tcq^{-\kappa}\Bigr)\geq 1,
\end{align*}
where the first inequality holds since 
$\vecalf$ is $[\kappa;c]$-Diophantine.
(Note also %
that there is no double representation in \eqref{lem:bound upperPF2},
i.e.\ $T((b+d)\vecalf+\vecm)$ is an injective function of $\langle d,\vecm\rangle\in\Z\times\Z^k$.)
Hence for any $R\geq1$, $M_b$ contains $\ll_k R^k$ points with $\|\vecx\|\leq R$, 
and so by a standard dyadic decomposition we have
\begin{align}\label{SPARSESUMlemPF2}
\sum_{\vecx\in M_b}(1+\|\vecx\|)^{-A}\ll_{k,A}1.
\end{align}
Now by appropriate choices of $b$,
the sum in \eqref{lem:bound upperRES} can be majorized by $1+D(cT)^{-1/\kappa}$
sums as in \eqref{SPARSESUMlemPF2}.

We have thus proved that the left hand side of \eqref{lem:bound upperRES}
is always $\ll D^{A\kappa+1}(cT)^{-A}$
and also $\ll 1+D(cT)^{-1/\kappa}$.
Splitting into cases depending on which bound is strongest,
we obtain the statement in \eqref{lem:bound upperRES}.
\end{proof}

The following proposition is an effective version of \cite[Prop.\ 6.5]{MarklofpaircorrI},
and is the central result needed to bound the error caused 
by truncating the function $\Theta_f\overline{\Theta_g}$
in the integral \eqref{MARKLOFkeyid}.
We here specialize to the case $k=2$;
the case $k\geq3$ involves in principle the same computations however
there are several differences in the detailed analysis %
(cf.\ \cite[Prop. 6.4]{MarklofpaircorrII}).

\begin{prop}\label{opp:upper bound}
Let $k=2$ and let $\vecxi_2\in\R^2$ be $[\kappa;c_0]$-Diophantine.
Let $Y\geq1$, $0<v\leq Y$, $A>2$, $B\geq1$, and $H\geq1$.
Then for any $f\in\C(\R^2)$ with $S_{\infty,A,0}(f)<\infty$,
and any bounded function $h:\R\to\R$ with support contained in $[-H,H]$,
\begin{align}\notag
\int_{|u|> Bv} F_{f,Y}\biggl(u+i v;\cmatr{\bn}{\vecxi_2}\biggr) \, h(u)\; du 
\hspace{240pt}
\\\label{opp:upper boundRES}
\ll_{A} S_{\infty,A,0}(f)\|h\|_{\L^\infty}
\biggl\{B^{-1}+H\bigl(H^{-1}c_0^{-1}v^{\frac12}\bigr)^{\frac1{\kappa+1+A^{-1}}}
+H \kappa c_0^{-\frac1{\kappa}}Y^{-\frac1{2\kappa}}
\biggr\}.
\end{align}
\end{prop}
We remark that the integral in \eqref{opp:upper boundRES} \textit{vanishes}
if $Y^{-1}\leq v\leq Y$;
hence the bound in \eqref{opp:upper boundRES} is mainly relevant when $v<Y^{-1}$.

\begin{proof}
Without loss of generality,
let us assume  that $f$ is positive and even, 
i.e., $f\geq 0$ and $f(-\vecw)=f(\vecw)$.
Recall the expansion \eqref{FYexpansion},
and note that the terms with $g_Y(v)$ vanish since $v\leq Y$;
hence we are left with
\begin{align}\notag
F_{f,Y}\biggl(\tau;\cmatr{\bn}{\vecxi_2}\biggr)= 
2 \sum_{\vecm\in\ZZ^2} 
f\Big(\vecm \frac{v^{1/2}}{|\tau|} \Big) 
\frac{v}{|\tau|^2} g_Y\Bigl(\frac{v}{|\tau|^2}\Bigr) 
\hspace{160pt}
\\\label{opp:upper boundPF3}
+2\sum_{\substack{(c,d)\in\ZZ^2\\ \gcd(c,d)=1 \\ c>0,d\neq 0}}\sum_{\vecm\in\ZZ^2} 
f\Big((d \vecxi_2 + \vecm) \frac{v^{1/2}}{|c\tau+d|} \Big)
\frac{v}{|c\tau+d|^2} g_Y\Bigl(\frac{v}{|c\tau+d|^2}\Bigr).
\end{align}

The contribution from the first sum in \eqref{opp:upper boundPF3}
to the integral in \eqref{opp:upper boundRES} is
\begin{align}\notag
&\int_{|u|>Bv}  \sum_{\vecm\in\ZZ^2} 
f\Big(\vecm \frac{v^{1/2}}{|\tau|} \Big) 
\frac{v}{|\tau|^2} g_Y\Bigl(\frac{v}{|\tau|^2}\Bigr)
\,h(u)\, du 
\\\label{opp:upper boundPF4}
&= \int_{|t|>B}\sum_{\vecm\in\ZZ^2} 
f\Big(\frac{\vecm}{v^{1/2} (t^2+1)^{1/2}} \Big) 
\frac{1}{t^2+1} 
g_Y\Bigl(\frac{1}{v(t^2+1)}\Bigr)\, h(vt)\, dt.
\end{align}
Using $|f(\vecx)|\leq S_{\infty,A,0}(f)(1+\|\vecx\|)^{-A}$ and $A>2$ we have
\begin{align*}
\sum_{\vecm\in\Z^2}f\Big(\frac{\vecm}{v^{1/2} (t^2+1)^{1/2}} \Big) \ll_{A} S_{\infty,A,0}(f)
\end{align*}
uniformly over all $v,t$ subject to $v^{-1/2}(t^2+1)^{-1/2}\geq 1$;
and for all other pairs $v,t$ the factor 
$g_Y\bigl(\frac{1}{v(t^2+1)}\bigr)$ vanishes (since $Y\geq1$).
Hence \eqref{opp:upper boundPF4} is
\begin{align*}
\ll_A S_{\infty,A,0}(f)\|h\|_{\L^\infty}\int_{|t|>B}\frac{dt}{t^2+1}
\ll S_{\infty,A,0}(f)\|h\|_{\L^\infty}B^{-1}.
\end{align*}

The contribution from the remaining double sum in \eqref{opp:upper boundPF3}
to the integral in \eqref{opp:upper boundRES}
is bounded above by (we drop the condition $|u|>Bv$ in the integral):
\begin{align}\notag
&\int_{\R}\sum_{\substack{(c,d)\in\ZZ^2\\ \gcd(c,d)=1 \\ c>0,d\neq 0}}\sum_{\vecm\in\ZZ^2} 
f\Big((d \vecxi_2 + \vecm) \frac{v^{1/2}}{|c\tau+d|} \Big)
\frac{v}{|c\tau+d|^2} g_Y\Bigl(\frac{v}{|c\tau+d|^2}\Bigr)\,h(u)\,du
\\\label{opp:upper boundPF5}
&=\sum_{\substack{(c,d)\in\ZZ^2\\ \gcd(c,d)=1 \\ c>0,d\neq 0}}\frac{1}{c^2}\int_{\R}\sum_{\vecm\in\ZZ^2} 
f\Big(\frac{d \vecxi_2 + \vecm}{\sqrt{c^2 v(t^2+1)}} \Big)
g_Y\Big(\frac{1}{c^2v(t^2+1)}\Big) 
\, h\Big(vt-\frac{d}{c}\Big)\, \frac{dt}{t^2+1},
\end{align}
where we changed the order of integration and summation and then substituted $u=vt-\frac dc$.
The $g_Y$-factor in the above expression vanishes unless $t$ belongs to the set
\begin{align*}
I_{c}=\biggl\{t\in\R\col 
\frac{1}{c^2 v(t^2+1)}>Y\biggr\}
=
\bigl\{t\in\R\col \sqrt{t^2+1}<(vY)^{-\frac12}c^{-1}\bigr\}.
\end{align*}
Furthermore, the factor $h(vt-\frac dc)$ vanishes unless
$|vt-\frac dc|\leq H$,
and for $t\in I_c$ this forces
\begin{align*}
\Bigl|\frac dc\Bigr|\leq H+\frac{v}{c\sqrt Y}\leq H+1\leq 2H.
\end{align*}
Also using $f(\vecx)\leq S_{\infty,A,0}(f)(1+\|\vecx\|)^{-A}$,
we conclude that the expression in \eqref{opp:upper boundPF5} is
\begin{align*}
\leq S_{\infty,A,0}(f)\|h\|_{\L^\infty}\sum_{c=1}^\infty 
\frac{1}{c^2}
\int_{I_c}
\sum_{0<|d|\leq 2Hc}
\;\sum_{\vecm\in\ZZ^2} 
\biggl(1+\frac{\|d \vecxi_2 + \vecm\|}{\sqrt{c^2 v(t^2+1)}} \biggr)^{-A}
\frac{dt}{t^2+1}.
\end{align*}
Applying Lemma \ref{lem:bound upper} with
\begin{align*}
D=2Hc
\qquad\text{and}\qquad
T=\frac{1}{\sqrt{c^2 v(t^2+1)}},
\end{align*}
and both $\vecalf=\vecxi_2$ and $\vecalf=-\vecxi_2$,
we get %
\begin{align}\notag
\ll_A S_{\infty,A,0}(f)\|h\|_{\L^\infty}
\sum_{c=1}^\infty c^{-2} \biggl(
\int_{I_{1,c}}D^{A\kappa+1}(c_0T)^{-A}\frac{dt}{t^2+1}
+\int_{I_{2,c}}\frac{dt}{t^2+1}
\hspace{70pt}
\\\label{opp:upper boundPF7}
+\int_{I_{3,c}}D(c_0T)^{-\delta}\frac{dt}{t^2+1}\biggr)
\end{align}
where $\delta:=1/\kappa$ and 
\begin{align*}
&I_{1,c}=\bigl\{t\in I_c\col \sqrt{t^2+1}\leq c_0(v^{\frac12}c)^{-1}D^{-(\kappa+A^{-1})}\bigr\};
\\
&I_{2,c}=\bigl\{t\in I_c\col c_0(v^{\frac12}c)^{-1}D^{-(\kappa+A^{-1})}<\sqrt{t^2+1}\leq 
c_0(v^{\frac12}c)^{-1}D^{-\kappa}\bigr\};
\\
&I_{3,c}=\bigl\{t\in I_c\col c_0(v^{\frac12}c)^{-1}D^{-\kappa}\leq\sqrt{t^2+1}\bigr\}.
\end{align*}

We discuss the three integrals in \eqref{opp:upper boundPF7} one by one.
Firstly, note that $I_{1,c}\neq\emptyset$ implies
$c\leq C_1:=\bigl((2H)^{-(\kappa+A^{-1})}c_0v^{-\frac12}\bigr)^{1/(\kappa+1+A^{-1})}$,
and for each such $c$, $t\in I_{1,c}$ implies
$\sqrt{t^2+1}\leq (C_1/c)^{\kappa+1+A^{-1}}$.
Hence
\begin{align*}
&\sum_{c=1}^\infty c^{-2}\int_{I_{1,c}}D^{A\kappa+1}(c_0T)^{-A}\frac{dt}{t^2+1}
\\
&\leq\sum_{1\leq c\leq C_1} c^{-2}\cdot D^{A\kappa+1}c_0^{-A}v^{\frac A2}c^A\int_{I_{1,c}} (t^2+1)^{\frac A2-1}\,dt
\\
&\ll_A\sum_{1\leq c\leq C_1} c^{-2}\cdot D^{A\kappa+1}c_0^{-A}v^{\frac A2}c^A\bigl(C_1/c\bigr)^{(\kappa+1+A^{-1})(A-1)}
\\
&=(2H)^{\kappa+A^{-1}}c_0^{-1}v^{\frac12}\sum_{1\leq c\leq C_1} c^{\kappa+A^{-1}-1}
\ll(2H)^{\kappa+A^{-1}}c_0^{-1}v^{\frac12}C_1^{\kappa+A^{-1}}
=C_1^{-1}.
\end{align*}

Turning to the integral over $I_{2,c}$,
the fact that
$t\in I_{2,c}$ forces $\sqrt{t^2+1}>(C_1/c)^{\kappa+1+A^{-1}}$,
with $C_1$ as above.
Therefore, we see that
\begin{align*}
&\sum_{c=1}^\infty c^{-2}\int_{I_{2,c}}\frac{dt}{t^2+1}
\leq\sum_{1\leq c\leq C_1}c^{-2}(C_1/c)^{-(\kappa+1+A^{-1})}
+\sum_{c>C_1}c^{-2}\ll C_1^{-1}.
\end{align*}

Finally for the integral over $I_{3,c}$ we have,
using only $I_{3,c}\subset I_c$,
\begin{align*}
\sum_{c=1}^\infty c^{-2}\int_{I_{3,c}}D(c_0T)^{-\delta}\frac{dt}{t^2+1}
\leq2Hc_0^{-\delta}v^{\frac{\delta}2}
\int_{\R}\biggl(\sum_{1\leq c<(vY(t^2+1))^{-1/2}}c^{\delta-1}\biggr)(t^2+1)^{\frac{\delta}2-1}\,dt
\\
\ll Hc_0^{-\delta}v^{\frac{\delta}2}\int_{\R}\frac{(vY(t^2+1))^{-\frac{\delta}2}}{\delta}(t^2+1)^{\frac{\delta}2-1}\,dt
\ll H\kappa c_0^{-\delta}Y^{-\frac{\delta}2}.
\end{align*}
Hence we obtain the bound in \eqref{opp:upper boundRES}.
\end{proof}

Next, we note that Proposition \ref{opp:upper bound} can be extended in a straightforward
manner to the case of functions $h$ which do not have compact support
but decay appropriately at infinity:
\begin{cor}\label{opp:upper bound2}
Let $\vecxi_2\in\R^2$ be $[\kappa;c_0]$-Diophantine.
Let $Y\geq1$, $0<v\leq Y$, $A>2$ and $B\geq1$.
Then for any $f\in\C(\R^2)$ with $S_{\infty,A,0}(f)<\infty$
and any function $h:\R\to\R$ with $S_{\infty,2,0}(h)<\infty$,
\begin{align}\notag
\int_{|u|> Bv} F_{f,Y}\biggl(u+i v;\cmatr{\bn}{\vecxi_2}\biggr) \, h(u)\; du 
\hspace{240pt}
\\\label{opp:upper bound2RES}
\ll_{A} S_{\infty,A,0}(f)S_{\infty,2,0}(h)
\biggl\{B^{-1}+\bigl(c_0^{-1}v^{\frac12}\bigr)^{\frac1{\kappa+1+A^{-1}}}
+ \kappa c_0^{-1/\kappa}Y^{-\frac1{2\kappa}}
\biggr\}.
\end{align}
\end{cor}
\begin{proof}
Decompose the function $h$ as $h=h_0+h_1+\cdots$ where $h_0=h\cdot\chi_{[-1,1]}$
and $h_j=h\cdot\bigl(\chi_{[-2^j,-2^{j-1})}+\chi_{(2^{j-1},2^j]}\bigr)$ for $j\geq1$;
then apply Proposition \ref{opp:upper bound2} to bound the contribution from each function $h_j$,
using $\supp h_j\subset[-2^j,2^j]$ and $\|h_j\|_{\L^\infty}\ll S_{\infty,2,0}(h) 2^{-2j}$.
\end{proof}

\subsection{Proof of Theorem \ref{MAINAPPLTHM}}
We are now ready to give the proof of Theorem \ref{MAINAPPLTHM}.
The first step is to give
an effective rate %
for the convergence of the integral in  \eqref{MARKLOFkeyid} to its limit;
this is obtained in Proposition \ref{MAINAPPLPROP1} below.
The proof of this proposition is divided into two lemmas,
Lemma \ref{|v|>}
which concerns the part of the integral where $u$ is not very near zero,
and Lemma~\ref{|v|<} which concerns the remaining part.
These two lemmas are (in principle) 
effective versions of \cite[Cor.\ 7.4]{MarklofpaircorrI}
and \cite[Lemma 8.3]{MarklofpaircorrI}, respectively.

Throughout this section, we let $\Gamma=\Gamma(2)\ltimes \ZZ^{4}$, and $G=\SL(2,\RR)\ltimes (\RR^{2})^{\oplus 2}$.
Recall that $\Theta_f\overline{\Theta_g}$
is a left $\Gamma^2$ invariant function on $G$,
with $\Gamma^2$ as in \eqref{GAMMAkdef};
thus in particular, it is left invariant under $\Gamma=\Gamma(2)\ltimes\Z^{4}$.
As always, we let $\mu$ be the probability measure on $\GaG$ induced by an appropriately normalized Haar measure on $G$
(which we also denote by $\mu$).

\begin{lem}\label{|v|>}
Let $f,g\in\scrS(\RR^2)$
and $h\in\C^2(\R)$, and assume $S_{\infty,3,2}(h)<\infty$.
Let $\vecxi_2\in\R^2$ be $[\kappa;c]$-Diophantine.
Then for any $v\in(0,1]$ and any real number $B$ subject to 
\begin{align}\label{|v|>REQ}
1\leq B\leq \tfrac12\, v^{-\frac12}
\Bigl(\tfrac14\,\delta_{6,\vecxi_2}\bigl(v^{-\frac12}\bigr)^{\frac12}\Bigr)^{\frac{\kappa}{1+61\kappa}},
\end{align}
we have
\begin{align}\notag
\biggl|\int_{|u|> Bv} \Theta_f\biggl(u+i v, 0 ;\cmatr{\bn}{\vecxi_2}\biggr)\,
\overline{\Theta_g\biggl(u+i v, 0 ;\cmatr{\bn}{\vecxi_2}\biggr)}\, h(u)\,du-
\int_{\Gamma\backslash G}\Theta_f\overline{\Theta_g} \, d\mu\int_{\R}h\,du \biggr|
\hspace{30pt}
\\\label{|v|>RES}
\ll \|f\|_{\L_{167}^2}\|g\|_{\L_{167}^2}
S_{\infty,3,2}(h)\,\Bigl(\kappa c^{-\frac1{\kappa}}
\,\delta_{6,\vecxi_2}\!\bigl(v^{-\frac12}\bigr)^{\frac1{127\,\kappa}}+B^{-1}\Bigr).
\end{align}
\end{lem}

\begin{remark}\label{|v|>REM}
As will be seen in the proof, %
the (quite small) power $\frac1{127\kappa}$
which we obtain in \eqref{|v|>RES}
depends strongly 
on which $\C_a^m$-norm %
of the test function (i.e., $\rule{0pt}{10pt}\tF$ below)
is required to bound in the effective equidistribution result of Theorem \ref{MAINTHM1}.
Since we did not make any effort to optimize the $a$ and $m$ in Theorem \ref{MAINTHM1},
we do not attempt %
to optimize the decay rate with respect to $v$ in Lemma \ref{|v|>}
nor in Lemma \ref{|v|<} or Proposition \ref{MAINAPPLPROP1}.
Instead we focus on giving results which are simple to state, yet explicit.
\end{remark}
\begin{proof}
Let $F=\Theta_f\overline{\Theta}_g$.
Also, for $Y\geq1$ a real number which we will choose below
(cf.\ \eqref{|v|>PF3}),
let $\tF=(1-\scrX_Y)\cdot F$,
i.e.
\begin{align*}
\tF(\tau,\phi,\vecxi)=\bigl(1-\scrX_Y(\tau)\bigr)\cdot(\Theta_f\overline{\Theta}_g)(\tau,\phi;\vecxi).
\end{align*}
Then both $F$ and $\tF$ are $\Gamma^2$ left invariant functions on $G$;
in particular, they are left invariant under $\Gamma=\Gamma(2)\ltimes\Z^{4}$.
Our choice of $Y$ will be such that
\begin{align}\label{|v|>PF5}
Y\leq\frac1{2(B^2+1)v}.
\end{align}
Then for all $u$ with $|u|\leq Bv$ we have
\begin{align*}
\tim\left(\matr0{-1}10 (u+iv)\right) %
=\frac v{u^2+v^2}\geq\frac 1{(1+B^2)v}\geq2Y\geq2,
\quad\text{and thus }\:\scrX_Y(u+iv)=1
\end{align*}
(cf.\ the discussion below \eqref{XYdef}).
Hence $\int_{|u|>Bv}\tF(u+iv,0;\vecxi)\,h(u)\,du=\int_{\R}\tF(u+iv,0;\vecxi)\,h(u)\,du$,
and so by Theorem \ref{MAINTHM1},
applied with $\beta=6$, $m=27$ and $a=5/2$, 
we have
\begin{align}\notag
\int_{|u|>Bv} \tF &\biggl(u+iv,0;\cmatr{\bn}{\vecxi_2}\biggr)\,h(u)\,du
\\\label{|v|>PF1}
&=\int_{\GaG}\tF\,d\mu \int_{\R} h(u)\,du+O_{\ve}\Bigl(\|\tF\|_{\C^{27}_{5/2}} S_{\infty,3,2}(h) 
\,\Bigl(\delta_{6,\vecxi_2}(v^{-\frac12})+v^{\frac14-\ve}\Bigr)\Bigr).
\end{align}

In order to bound $\|\tF\|_{\C^{27}_{5/2}}$,
we apply
Lemma \ref{XYDERlem} and \eqref{eq:thetaderivbound} in Lemma \ref{lem:betatheta},
together with the fact that $F$ and $\tF$ are $\Gamma^2$-invariant.
It follows that for any integer $a>4+6\cdot 27=166$,
\begin{align*}
\sum_{\ord(D)\leq 27}\bigl|(D\tF)(M,\vecxi)\bigr|
\ll\|f\|_{\L_a^2}\|g\|_{\L_a^2}\,\scrY(M)^{28},
\qquad\forall (M,\vecxi)\in G.
\end{align*}
Hence, since the support of $\tF$ is contained in $\{\scrY(M)\leq 2Y\}$,   
\begin{align*}
\|\tF\|_{\C^{27}_{5/2}}\ll \|f\|_{\L_a^2}\|g\|_{\L_a^2}\,Y^{61/2}.
\end{align*}

Next, we bound the error caused by replacing $\tF$ by $F$ in the two integrals
$\int_{\GaG}\tF\,d\mu$ and $\int_{|u|>Bv} \tF \bigl(u+iv,0;\scmatr{\bn}{\vecxi_2}\bigr)\,h(u)\,du$ 
in \eqref{|v|>PF1}.
First note that
by Lemma \ref{thetaboundLEM}
we have
\begin{align}\label{|v|>PF2}
\bigl|F(\tau,\phi;\vecxi)-\tF(\tau,\phi;\vecxi)\bigr|
\leq F_{f^*,Y}(\tau;2\vecxi)+O\Bigl(\|f\|_{\L^2_6}\|g\|_{\L^2_6}Y^{-1}\Bigr),
\qquad
\forall (\tau,\phi;\vecxi)\in G,
\end{align}
with $f^*(\vecw):=\sup_{\phi\in\R}\bigl|f_\phi(\tfrac12\vecw)g_\phi(\tfrac12\vecw)\bigr|$.
Hence 
\begin{align*}
&\biggl|\int_{|u|>Bv} \tF\biggl(u+iv,0;\cmatr{\bn}{\vecxi_2}\biggr)\,h(u)\,du
-\int_{|u|>Bv} F\biggl(u+iv,0;\cmatr{\bn}{\vecxi_2}\biggr)\,h(u)\,du\biggr|
\\
&\leq\int_{|u|>Bv} F_{f^*,Y}\biggl(u+iv;\cmatr{\bn}{2\vecxi_2}\biggr)\,|h(u)|\,du
+O\Bigl(\|f\|_{\L^2_6}\|g\|_{\L^2_6}\|h\|_{\L^1}\,Y^{-1}\Bigr).
\end{align*}
Note that since $\vecxi_2$ is $[\kappa;c]$-Diophantine,
$2\vecxi_2$ is $[\kappa;2^{-\kappa}c]$-Diophantine.
Hence, applying Corollary \ref{opp:upper bound2} with $A=3$, 
and noticing that $S_{\infty,3,0}(f^*)\ll\|f\|_{\L^2_4}\|g\|_{\L^2_4}$
by Lemma \ref{schwartz}, we get
\begin{align*}
\int_{|u|>Bv}F_{f^*,Y} & \biggl(u+iv;\cmatr{\bn}{2\vecxi_2}\biggr)\,|h(u)|\,du
\\
&\ll_{\ve} \|f\|_{\L^2_4}\|g\|_{\L^2_4}S_{\infty,2,0}(h)
\biggl\{B^{-1}+\bigl(c^{-1}v^{\frac12}\bigr)^{1/(\kappa+\frac43)}
+\kappa\, c^{-\frac1{\kappa}}Y^{-\frac1{2\kappa}}
\biggr\}.
\end{align*}
Also by \eqref{|v|>PF2} we have
\begin{align*}
\int_{\GaG}\bigl|F-\tF\bigr|\,d\mu\leq
\int_{\GaG}F_{f^*,Y}(\tau;2\vecxi)\,d\mu(\tau,\phi;\vecxi)+O\Bigl(\|f\|_{\L^2_6}\|g\|_{\L^2_6}Y^{-1}\Bigr).
\end{align*}
Here one computes, by a standard unfolding argument (cf.\ \cite[6.2]{MarklofpaircorrII}),
\begin{align}\notag %
\int_{\GaG}F_{f^*,Y}(\tau;2\vecxi)\,d\mu %
=\int_{\GaG}F_{f^*,Y}(\tau;\vecxi)\,d\mu %
=\frac{3}{\pi Y}\int_0^\infty g_1(y)\, \frac{dy}{y^2}\,\int_{\R^2}f^*\,d\vecw
\ll \|f\|_{\L^2_4}\|g\|_{\L^2_4}Y^{-1}.
\end{align}
We combine the above bounds with \eqref{|v|>PF1},
where we also use the fact that $\delta_{\beta,\vecxi_2}(T)\geq T^{-1}$ ($\forall T\geq1$),
which follows by just considering the terms corresponding to $\vecr=\pm\vece_1$ and $j=1$ 
in \eqref{MMXIDEF}.
We then get, with $a=167$:
\begin{align}\notag
&\biggl|\int_{|u|>Bv} F\biggl(u+iv,0;\cmatr{\bn}{\vecxi_2}\biggr)\,h(u)\,du-\int_{\GaG}F\,d\mu \int_{\R} h\,du\biggr|
\\\label{|v|>PF4}
&\hspace{40pt}
\ll_{\ve}\|f\|_{\L_a^2}\|g\|_{\L_a^2}
S_{\infty,3,2}(h)\, \delta_{6,\vecxi_2}(v^{-\frac12})^{\frac12-2\ve}\,Y^{\frac{61}2}
+\|f\|_{\L^2_6}\|g\|_{\L^2_6}\|h\|_{\L^1}Y^{-1}
\\\notag
&\hspace{80pt}
+\|f\|_{\L^2_4}\|g\|_{\L^2_4}
S_{\infty,2,0}(h)
\biggl\{B^{-1}+\bigl(c^{-1}v^{\frac12}\bigr)^{1/(\kappa+\frac43)}
+\kappa\, c^{-\frac1{\kappa}}Y^{-\frac1{2\kappa}}
\biggr\}.
\end{align}

In order to minimize the order of magnitude of
the maximum of
$\delta_{6,\vecxi_2}(v^{-\frac12})^{\frac12-2\ve}Y^{\frac{61}2}$ and $Y^{-\frac1{2\kappa}}$,
we now make the choice
\begin{align}\label{|v|>PF3}
Y:=\Bigl(\,\tfrac1{3}\,\,\delta_{6,\vecxi_2}(v^{-\frac12})^{\frac12-2\ve}\Bigr)^{-\frac{2\kappa}{1+61\kappa}}.
\end{align}
Because of the factor $\frac1{3}$ in this expression,
we are guaranteed to have $Y\geq1$, as required above.
Indeed, one verifies $\sum_{\vecr\in\Z^2\setminus\{\bn\}}\|\vecr\|^{-6}\sum_{j=1}^\infty j^{-2}<9$;
hence $\delta_{6,\vecxi_2}(T)<9$ for all $T\geq1$ (cf.\ \eqref{deltabetaxiBASICBOUND}).
Furthermore, our assumption \eqref{|v|>REQ}
ensures that %
\eqref{|v|>PF5} is fulfilled, so long as $9^{2\ve}\leq\frac43$.
Note also that
$Y^{-\frac1{2\kappa}}\geq Y^{-1}$ (since $\kappa\geq\frac12$)
and 
$\bigl(c^{-1}v^{\frac12}\bigr)^{1/(\kappa+\frac43)}\ll\kappa\, c^{-\frac1{\kappa}}Y^{-\frac1{2\kappa}}$
(since $\kappa\geq\frac12$, $0<c<2^{-1/2}<1$ and 
$\delta_{6,\vecxi_2}(v^{-\frac12})^{\frac12-2\ve}\geq v^{\frac14-\ve}\geq v^{\frac14}$).
Finally, we take $\ve=\frac1{508}$ and note that we
then have $(\frac12-2\ve)\frac1{1+61\kappa}\geq\frac1{127\kappa}$,
since $\kappa\geq\frac12$;
also $9^{2\ve}\leq\frac43$ as required above.
Hence the bound \eqref{|v|>RES}
now follows from \eqref{|v|>PF4}.
\end{proof}

\begin{lem}\label{|v|<} 
Suppose $f,g\in\scrS(\R^2)$.
Let $h\in\C^1(\R)$ with $h$ and $h'$ bounded.
Then for any $\vecxi=\cmatr{\bn}{\vecxi_2}$, $v\in(0,1]$ and $B\in[1,v^{-1/2}]$, we have
\begin{multline*}
\int_{|u|< Bv}\Theta_f(u+i v, 0 ;\vecxi)
\overline{\Theta_g(u+i v, 0 ;\vecxi)}\, h(u)\,du  
=\lambda_{f,\og}\,h(0)+O\Bigl(\|f\|_{\L^2_{6}}\|g\|_{\L^2_{6}} S_{\infty,0,1}(h)B^{-1}\Bigr),
\end{multline*}
where
\begin{align}\label{lambdafgdef}
\lambda_{f,g}=\int_0^\infty\biggl(\int_0^{2\pi}f(r\cos\zeta,r\sin\zeta)\,d\zeta\biggr)
\biggl(\int_0^{2\pi}g(r\cos\zeta,r\sin\zeta)\,d\zeta\biggr)\,r\,dr.
\end{align}
\end{lem}
\begin{proof}
Recall that the function $\Theta_f\overline{\Theta}_g$
is left $\Gamma^2$ invariant;
in particular, it is invariant under left multiplication by $\bigl(\smatr 0{-1}10,\bn\bigr)\in\Gamma^2$,
and so
\begin{align*}
\Theta_f(\tau, 0 ;\vecxi)\overline{\Theta_g(\tau, 0 ;\vecxi)}
=\Theta_f\biggl(-\frac{1}{\tau},\arg\tau;\cmatr{-\vecxi_2}{\vecnull}\biggr)
\overline{\Theta_g\biggl(-\frac{1}{\tau},\arg\tau;\cmatr{-\vecxi_2}{\vecnull}\biggr)},
\end{align*}
for all $\tau=u+iv\in\HH$.
By \eqref{eq:m=0term} in Lemma \ref{lem:betatheta} 
(applied with $A=1$),
if $\tim(-1/\tau)\geq\frac12$ then the last expression equals
\begin{align*}
f_{\arg\tau}(\vecnull) \overline{g_{\arg\tau}(\vecnull)}\frac{v}{|\tau|^{2}}
+O\Bigl(\|f\|_{\L^2_{6}}\|g\|_{\L^2_{6}}\frac{|\tau|^2}{v}\Bigr).
\end{align*}
Note that $|u|\leq v^{1/2}\leq1$ implies $\tim(-1/\tau)\geq\frac12$,
i.e.\ the above holds for all $\tau\in\HH$ with $|u|\leq v^{1/2}\leq 1$.
Hence we get
\begin{multline*}
\int_{|u|< Bv}
\Theta_f(u+i v, 0 ;\vecxi)
\overline{\Theta_g(u+i v, 0 ;\vecxi)} h(u)\,du  \\
=  \int_{|u|< Bv}\frac{v}{|\tau|^2}
f_{\arg\tau}(\vecnull)\overline{g_{\arg\tau}(\vecnull)}\,  h( u)\,du
+O\Bigl(\|f\|_{\L^2_{6}}\|g\|_{\L^2_{6}}\|h\|_{\L^\infty}B^{3}v^{2} \Bigr).
\end{multline*}

Using polar coordinates we get (cf.\ \cite[p.\ 457]{MarklofpaircorrI})
\begin{align*}
f_{\arg\tau}(\vecnull)\overline{g_{\arg\tau}(\vecnull)} 
= \frac{|\tau|^2}{v^2}\, \pi^2 \,
\hpsi_1\Bigl(\frac u{2v}\Bigr)
\overline{\hpsi_2\Bigl(\frac u{2v}\Bigr)}
\end{align*}
where $\psi_1(r):=\frac1{2\pi}\int_0^{2\pi}f\bigl(\sqrt r(\cos\zeta,\sin\zeta)\bigr)\,d\zeta$,
$\psi_2(r):=\frac1{2\pi}\int_0^{2\pi}g\bigl(\sqrt r(\cos\zeta,\sin\zeta)\bigr)\,d\zeta$,
and $\hpsi_j(u):=\int_0^\infty e(ur)\psi_j(r)\,dr$.
Therefore, using also $h(u)=h(0)+O\bigl(\|h'\|_{\L^\infty}|u|\bigr)$,
\begin{align}\notag
\int_{|u|< Bv}\frac{v}{|\tau|^2}
f_{\arg\tau}(\vecnull)\overline{g_{\arg\tau}(\vecnull)}  h( u)\,du
=\frac{\pi^2h(0)}v\int_{|u|< Bv}
\hpsi_1\Bigl(\frac u{2v}\Bigr)
\overline{\hpsi_2\Bigl(\frac u{2v}\Bigr)}\,du
\hspace{40pt}
\\\label{|v|<pf1} 
+O\biggl(\frac{\|h'\|_{\L^\infty}}v \int_{|u|< Bv}|u|
\Bigl|\hpsi_1\Bigl(\frac u{2v}\Bigr)
\overline{\hpsi_2\Bigl(\frac u{2v}\Bigr)}\Bigr|\,du
\biggr).
\end{align}
To bound the last error term,
first replace the integration variable $u$ by $2vu$;
then use the fact that by integration by parts we have
$|\hpsi_j(u)|\ll \int_0^\infty\bigl(|\psi_j|+|\psi_j'|\bigr)\,dr\cdot\min(1,|u|^{-1})$.
Here
$\int_0^\infty|\psi_1|\,dr %
\leq\frac1\pi \|f\|_{\L^1}$,
while
\begin{align*}
\int_0^\infty|\psi_1'(r)|\,dr
&\leq\frac1{2\pi}\int_0^\infty\int_0^{2\pi}\Bigl(\bigl|(\partial_{x_1}f)\bigl(\sqrt r(\cos\zeta,\sin\zeta)\bigr)\bigr|
+\bigl|(\partial_{x_2}f)\bigl(\sqrt r(\cos\zeta,\sin\zeta)\bigr)\bigr|\Bigr)\,\frac{d\zeta\,dr}{2\sqrt r}
\\
&=\frac1{2\pi}\int_{\R^2}\Bigl(\bigl|(\partial_{x_1}f)(\vecx)\bigr|+\bigl|(\partial_{x_2}f)(\vecx)\bigr|\Bigr)\,
\frac{d\vecx}{\|\vecx\|}
\ll S_{\infty,0,1}(f)+S_{1,0,1}(f)
\ll \|f\|_{\L_6^2}.
\end{align*}
(The next to last bound follows by splitting the domain of integration into
the two parts $\{\|\vecx\|\leq1\}$ and $\{\|\vecx\|>1\}$,
and the last bound is immediate by Sobolev embedding.)
Similarly for $\psi_2$.
Hence the error term in \eqref{|v|<pf1}  is
\begin{align*}
\ll \|f\|_{\L_6^2}\|g\|_{\L_6^2}\|h'\|_{\L^\infty}\, v\log(B+2).
\end{align*}
Finally, we are left with:
\begin{align*}
&\frac{\pi^2h(0)}v\int_{|u|< Bv}
\hpsi_1\Bigl(\frac u{2v}\Bigr)
\overline{\hpsi_2\Bigl(\frac u{2v}\Bigr)}\,du
=\pi^2h(0)\int_{|u|< B}
\hpsi_1\Bigl(\frac u{2}\Bigr)
\overline{\hpsi_2\Bigl(\frac u{2}\Bigr)}\,du
\\
&=\pi^2h(0)\biggl(\int_{-\infty}^{\infty}
\hpsi_1\Bigl(\frac u{2}\Bigr)
\overline{\hpsi_2\Bigl(\frac u{2}\Bigr)}\,du
+O\biggl(\|f\|_{\L_6^2}\|g\|_{\L_6^2}
\int_{|u|>B}|u|^{-2}\,du\biggr)\biggr)
\\
&=2\pi^2h(0)\int_0^{\infty}\psi_1(r)\overline{\psi_2(r)}\,dr
+O\bigl(\|f\|_{\L_6^2}\|g\|_{\L_6^2}|h(0)| B^{-1}\bigr)
=\lambda_{f,\og}h(0)
+O\bigl(\|f\|_{\L_6^2}\|g\|_{\L_6^2}|h(0)| B^{-1}\bigr),
\end{align*}
where in the next to last equality we used Parseval's identity. %

Collecting the above results,
and noticing that
$v\log(B+2)\ll B^{-1}$ and $B^3v^2\leq B^{-1}$
because of $1\leq B\leq v^{-1/2}$,
we obtain the statement of the lemma.
\end{proof}

\begin{prop}\label{MAINAPPLPROP1}
Let $f,g\in\scrS(\R^2)$ and $h\in\C^2(\R)$, and assume $S_{\infty,3,2}(h)<\infty$.
Let $\vecxi_2\in\R^2$ be $[\kappa;c]$-Diophantine.
Then for any $v\in(0,1]$,
\begin{align}\notag
\int_{\R} \Theta_f\biggl(u+i v, 0 ;\cmatr{\bn}{\vecxi_2}\biggr)\,
\overline{\Theta_g\biggl(u+i v, 0 ;\cmatr{\bn}{\vecxi_2}\biggr)}\, h(u)\,du
\hspace{140pt}
\\\label{MAINAPPLPROP1RES}
=\int_{\R^2} f(\vecx)\overline{g(\vecx)}\,d\vecx\int_{\R}h\,du
+\lambda_{f,\og}\, h(0)
\hspace{160pt}
\\\notag
+O\biggl( \|f\|_{\L_{167}^2}\|g\|_{\L_{167}^2}
S_{\infty,3,2}(h)\,\kappa c^{-\frac1{\kappa}}
\,\delta_{6,\vecxi_2}(v^{-\frac12})^{\,\frac{1}{127\,\kappa}}\biggr).
\end{align}
\end{prop}
\begin{proof}
By \cite[lemma 8.2]{MarklofpaircorrI} (cf.\ also \cite[lemma 7.2]{MarklofpaircorrII}),
\begin{align*}
\int_{\Gamma\backslash G}
\Theta_f\overline{\Theta_g} \, d\mu=
\int_{\R^2} f(\vecx)\overline{g(\vecx)}\,d\vecx.
\end{align*}
Therefore, the proposition follows from 
Lemmas \ref{|v|>} and \ref{|v|<},
applied with 
\begin{align}\label{MAINAPPLPROP1PF1}
B=\tfrac12\, v^{-\frac12}
\Bigl(\tfrac14\,\delta_{6,\vecxi_2}\bigl(v^{-\frac12}\bigr)^{\frac12}\Bigr)^{\frac{\kappa}{1+61\kappa}},
\end{align}
as long as this number satisfies $B\geq1$. %
Indeed, from the observations below \eqref{|v|>PF3} we see that the number $B$ in \eqref{MAINAPPLPROP1PF1} 
satisfies $B\leq v^{-\frac12}$, as is required in Lemma \ref{|v|<}.
Furthermore, using $\delta_{6,\vecxi_2}(T)\geq T^{-1}$ $\forall T\geq1$
(as noted in the proof of Lemma \ref{|v|>})
and $\frac{\kappa}{2(1+61\kappa)}+\frac1{127\kappa}<1$,
it follows that
$B^{-1}\ll\delta_{6,\vecxi_2}(v^{-\frac12})^{\,\frac{1}{127\,\kappa}}$,
so that we indeed obtain the error bound in the last line of \eqref{MAINAPPLPROP1RES}.

It remains to consider the case when 
the number $B$ in \eqref{MAINAPPLPROP1PF1} is less than $1$. %
Using $\delta_{6,\vecxi_2}(v^{-\frac12})\geq v^{\frac12}$ it then follows that
$v$ is bounded below by some positive absolute constant;
also from bounds discussed previously it follows that each of
$\int_{\R}\Theta_f(\cdots)\overline{\Theta_g(\cdots)}\,h(u)\,du$,
$\int_{\R^2} f(\vecx)\overline{g(\vecx)}\,d\vecx\int_{\R}h\,du$
and $\lambda_{f,\og}\, h(0)$ are $\ll\|f\|_{\L_5^2}\|g\|_{\L_5^2}S_{\infty,0,2}(h)$;
hence \eqref{MAINAPPLPROP1RES} holds trivially in this case.
\end{proof}

With Proposition \ref{MAINAPPLPROP1} established,
the proof of Theorem \ref{MAINAPPLTHM}
can now be completed 
by a sequence of approximation steps.
\begin{proof}[Proof of Theorem \ref{MAINAPPLTHM}]
Let $(\alpha,\beta)\in\R^2$ be given as in the statement of the theorem,
and set $\vecxi_2=\cmatr{\alpha}{\beta}$.
By \eqref{MARKLOFkeyid} and Proposition \ref{MAINAPPLPROP1},
writing $g_1$ and $\overline{g_2}$ in place of $f$ and $g$, respectively,
we have for any
$g_1,g_2\in\scrS(\R^2)$, $h\in\C^2(\R)$ with $S_{\infty,3,2}(h)<\infty$,
and 
$T\geq1$:
\begin{align}\notag
\frac1{T^2}\sum_{\vecm_1\in\Z^2}\sum_{\vecm_2\in\Z^2}
g_1\bigl(T^{-1}(\vecm_1-\vecxi_2)\bigr)\,g_2\bigl(T^{-1}(\vecm_2-\vecxi_2)\bigr)
\,\,\hh\biggl(-\tfrac12\, Q\hspace{-3pt}\cmatr{\vecm_1}{\vecm_2}\biggr)\hspace{80pt}
\\\label{MAINAPPLTHMpfusefulrel}
=\int_{\R^2} g_1(\vecx)g_2(\vecx)\,d\vecx\cdot\hh(0)
+\lambda_{g_1,g_2} h(0)
\hspace{160pt}
\\\notag
+O_\ve\biggl( \|g_1\|_{\L_{167}^2}\|g_2\|_{\L_{167}^2}
S_{\infty,3,2}(h)\,\kappa c^{-\frac1{\kappa}}
\,\delta^{\frac{1}{127\,\kappa}}\biggr),
\end{align}
where we use the short-hand notation $\delta:=\delta_{6,\vecxi_2}(T)$.
Let us consider the contribution from all terms with $\vecm_2=\vecm_1$ in sum in the left hand side.
Set $G:=g_1g_2\in\scrS(\R^2)$.
Note that 
\begin{align*}
G(T^{-1}\vecx)=\int_{\vecx+[0,1]^2}G(T^{-1}\vecy)\,d\vecy
+O\biggl(\frac{S_{\infty,3,1}(G)}{T(1+T^{-1}\|\vecx\|)^3}\biggr),
\qquad\forall\vecx\in\R^2.
\end{align*}
Adding this relation over all $\vecx=\vecm_1-\vecxi_2$ ($\vecm_1\in\Z^2$),
and noticing $\sum %
(1+T^{-1}\|\vecm_1-\vecxi_2\|)^{-3}\ll T^2$,
we get
\begin{align*}
\frac1{T^2}\sum_{\vecm_1\in\Z^2}(g_1g_2)(T^{-1}(\vecm_1-\vecxi_2))
=\int_{\R^2}g_1g_2\,d\vecx+O\bigl(S_{\infty,3,1}(g_1g_2)\, T^{-1}\bigr).
\end{align*}
We multiply this identity by $\hh(0)$,
and note that the error term is then subsumed by the error term in 
\eqref{MAINAPPLTHMpfusefulrel},
since 
$S_{\infty,3,1}(g_1g_2)\leq\sum_{m=0}^1S_{\infty,\frac32,m}(g_1)S_{\infty,\frac32,1-m}(g_2)
\ll\|g_1\|_{L_3^2}\|g_2\|_{L_3^2}$.
Hence, subtracting the resulting identity from \eqref{MAINAPPLTHMpfusefulrel},
we obtain:
\begin{align}\notag
\frac1{T^2}\sum_{\vecm_1\in\Z^2}\sum_{\substack{\vecm_2\in\Z^2\\\vecm_2\neq\vecm_1}}
g_1\bigl(T^{-1}(\vecm_1-\vecxi_2)\bigr)\,g_2\bigl(T^{-1}(\vecm_2-\vecxi_2)\bigr)
\,\,\hh\biggl(-\tfrac12\, Q\hspace{-3pt}\cmatr{\vecm_1}{\vecm_2}\biggr)\hspace{80pt}
\\\label{MAINAPPLTHMpfusefulrel2}
=\lambda_{g_1,g_2} h(0)
+O_\ve\biggl( \|g_1\|_{\L_{167}^2}\|g_2\|_{\L_{167}^2}
S_{\infty,3,2}(h)\,\kappa c^{-\frac1{\kappa}}
\,\delta^{\frac{1}{127\,\kappa}}\biggr).
\end{align}

Next, we take $g_1,g_2$ in \eqref{MAINAPPLTHMpfusefulrel2} to be given by
$g_j(\vecx):=f_j(\vecx+T^{-1}\vecxi_2)$ 
for some $f_1,f_2\in\scrS(\R^2)$.
Recall that $\vecxi_2=\cmatr{\alpha}{\beta}$,
and by assumption in Theorem~\ref{MAINAPPLTHM}
this vector lies in $[-1,1]^2$.
Hence $\|g_j\|_{\L_{167}^2}\ll \|f_j\|_{\L_{167}^2}$
and $\|g_j-f_j\|_{\L^\infty}\ll  S_{\infty,0,1}(f_j)\,T^{-1}$,
for $j=1,2$.
Inspecting the definition of $\lambda_{f,g}$ in \eqref{lambdafgdef} it follows that
\begin{align*}
\bigl|\lambda_{g_1,g_2}-\lambda_{f_1,f_2}\bigr|
\ll \Bigl(S_{\infty,3,0}(g_1)S_{\infty,0,1}(f_2)+S_{\infty,0,1}(f_1)S_{\infty,3,0}(f_2)\Bigr)
\,T^{-1}\int_0^\infty (1+r)^{-3}r\,dr
\hspace{20pt}
\\
\ll\|f_1\|_{\L_3^2}\|f_2\|_{\L_3^2}\, T^{-1}.
\end{align*}
Hence
\begin{align}\label{MAINAPPLTHMpf1}
\frac1{T^2}\sum_{\vecm_1\in\Z^2}\sum_{\substack{\vecm_2\in\Z^2\\\vecm_2\neq\vecm_1}}
&f_1(T^{-1}\vecm_1)\,f_2(T^{-1}\vecm_2)
\,\hh\biggl(-\tfrac12\, Q\hspace{-3pt}\cmatr{\vecm_1}{\vecm_2}\biggr)    %
\\\notag
&=\lambda_{f_1,f_2} h(0)
+O_\ve\biggl( \|f_1\|_{\L_{167}^2}\|f_2\|_{\L_{167}^2}
S_{\infty,3,2}(h)\,\kappa c^{-\frac1{\kappa}}
\,\delta^{\frac1{127\kappa}}\biggr).
\end{align}
Next, take $h$ to be given by
$h(u)=\frac12\hg(\frac12u)$, where $g$ is any function in $\C^3(\R)$ with $S_{1,2,3}(g)<\infty$.
Then $S_{\infty,3,2}(h)\ll S_{1,2,3}(g)$,
and by Fourier inversion,
$g(s)=\hh(-\frac12s)$.
Let us also write $f_1\otimes f_2$ for the function in $\scrS(\R^4)$
given by $(f_1\otimes f_2)\cmatr{\vecx_1}{\vecx_2}=f_1(\vecx_1)f_2(\vecx_2)$.
Comparing \eqref{lambdafdef} and \eqref{lambdafgdef} we then have
$\lambda_{f_1,f_2}=2\lambda_{f_1\otimes f_2}$.
Comparing also with \eqref{NabfgTdef},
we obtain:
\begin{align}\label{MAINAPPLTHMpf3}
N_{\alpha,\beta}(f_1\otimes f_2,g,T)=\lambda_{f_1\otimes f_2}\, \hg(0) %
+O\biggl( \|f_1\|_{\L_{167}^2}\|f_2\|_{\L_{167}^2}
S_{1,2,3}(g)\,\kappa c^{-\frac1{\kappa}}
\,\delta^{\,\frac1{127\kappa}}\biggr).
\end{align}

It will be useful to note the following consequence of \eqref{MAINAPPLTHMpf3}:
\begin{lem}\label{BASICUPPERBDLEM}
For any $[\kappa;c]$-Diophantine vector $(\alpha,\beta)\in[-1,1]^2$,
any $g\in \C^1(\R)$ with $S_{1,2,1}(g)<\infty$,
and any $R\geq1$,
\begin{align}\label{BASICUPPERBDLEMres}
\sum_{\substack{\vecm\in\Z^4\setminus\Delta\\\|\vecm\|\leq R}}\bigl|g(Q(\vecm))\bigr|
\ll S_{1,2,1}(g)\bigl(1+\kappa c^{-\frac1{\kappa}}\,\delta_{6,\vecxi_2}(R)^{\,\frac1{127\kappa}}\bigr)R^2,
\end{align}
where the implied constant is absolute.
\end{lem}
\begin{proof}
A standard construction shows that there exists a function $\tg\in\C^\infty(\R)$ 
satisfying $\tg\geq|g|$ and $S_{1,2,3}(\tg)\ll S_{1,2,1}(g)$, with an absolute implied constant.
Fix a choice of a non-negative function $f\in\C_c^\infty(\R^2)$ satisfying $f(\vecx)=1$ for all
$\vecx$ with $\|\vecx\|\leq1$.
Applying \eqref{MAINAPPLTHMpf3} with $T=R$, $f_1=f_2=f$ and $\tg$ in place of $g$ gives (cf.\ \eqref{NabfgTdef})
\begin{align*}
\frac1{R^2}\sum_{\vecm\in\Z^4\setminus\Delta}[f\otimes f](R^{-1}\vecm)\tg(Q(\vecm))
\ll \|\tg\|_{\L^1}+S_{1,2,3}(\tg)\kappa c^{-\frac1{\kappa}}\,
\delta_{6,\vecxi_2}(R)^{\,\frac1{127\kappa}}.
\end{align*}
Using $\|\tg\|_{\L^1}\leq S_{1,2,3}(\tg)\ll S_{1,2,1}(g)$ and the fact that
$[f\otimes f](R^{-1}\vecm)\geq1$ whenever $\|\vecm\|\leq R$,
we obtain \eqref{BASICUPPERBDLEMres}.
\end{proof}
\begin{remark}\label{BASICUPPERBDLEMrem}
By contrast, if $(\alpha,\beta)$ is not Diophantine then the left hand
side of \eqref{BASICUPPERBDLEMres} may grow more rapidly than $R^2$ as $R\to\infty$;
cf.\ \cite[Sec.\ 9]{MarklofpaircorrI}.
\end{remark}
We now continue with the proof of Theorem \ref{MAINAPPLTHM}.
Take $f\in\C_c^1(\R^4)$ with support contained in the unit ball centered at the origin.
We wish to go from \eqref{MAINAPPLTHMpf3}
to an asymptotic formula for $N_{\alpha,\beta}(f,g,T)$.
Fix, once and for all, a function 
$\phi\in\C_c^\infty(\R^2)$ 
with support contained in the unit ball centered at the origin
and satisfying $\int_{\R^2}\phi(\vecx)\,d\vecx=1$.
Then for an appropriate number $0<\eta<1$ (to be fixed below) 
we define $\phi_\eta\in\C_c^\infty(\R^2)$ by 
$\phi_\eta(\vecx):=\eta^{-2}\phi(\eta^{-1}\vecx)$,
and set
\begin{align*}
\tf:=f*(\phi_\eta\otimes\phi_\eta).
\end{align*}
Note that for any $\vecx,\vecy\in\R^4$ with $(\phi_\eta\otimes\phi_\eta)(\vecx-\vecy)\neq0$
one has $\|\vecy-\vecx\|\leq\sqrt2\,\eta$ and thus
$|f(\vecy)-f(\vecx)|\ll S_{\infty,0,1}(f)\cdot\eta$;
hence
\begin{align}\label{MAINAPPLTHMpf11}
\bigl|f(\vecx)-\tf(\vecx)\bigr|\ll S_{\infty,0,1}(f)\cdot\eta,\qquad\forall\vecx\in\R^4.
\end{align}
Therefore, by \eqref{lambdafdef} and using the fact that the supports of both $f$ and $\tf$ are contained in the ball
$\{\|\vecx\|\leq3\}$,
\begin{align*}
|\lambda_f-\lambda_{\tf}| %
\ll S_{\infty,0,1}(f)\cdot\eta,
\end{align*}
and also, by \eqref{NabfgTdef},
\begin{align*}
\bigl|N_{\alpha,\beta}(f,g,T)-N_{\alpha,\beta}(\tf,g,T)\bigr|
&\ll \frac{S_{\infty,0,1}(f)\,\eta}{T^2}\sum_{\substack{\vecm\in\Z^4\setminus\Delta\\\|\vecm\|\leq3T}}\bigl|g(Q(\vecm))\bigr|
\\
&\ll S_{\infty,0,1}(f)S_{1,2,1}(g)\Bigl(1+\kappa c^{-\frac1\kappa}\,\delta_{6,\vecxi_2}(T)^{\,\frac1{127\kappa}}\Bigr)\eta,
\end{align*}
where the last bound follows from Lemma \ref{BASICUPPERBDLEM} 
and the fact that 
$\delta_{\beta,\vecxi}(T)$ is essentially a decreasing fuction of $T$,
in the sense that
\begin{align}\label{DELTAESSDECR}
\delta_{\beta,\vecxi}(T')<2\delta_{\beta,\vecxi}(T)\qquad\text{for any }\:0<T\leq T'
\end{align}
(this follows from \eqref{deltabetaxiBASICBOUND} and the fact that 
$\frac{1+\log^+y}{1+y}<2\frac{1+\log^+x}{1+x}$
whenever $0<x\leq y$).

Next, using $\tf:=f*(\phi_\eta\otimes\phi_\eta)$ we have
\begin{align*}
N_{\alpha,\beta}(\tf,g,T)=\int_{\R^4}f(\vecy) N_{\alpha,\beta}\bigl(\phi_{\eta,\vecy_1}\otimes\phi_{\eta,\vecy_2},g,T\bigr)\,d\vecy,
\end{align*}
where 
$\phi_{\eta,\veca}(\vecx)=\phi_\eta(\vecx-\veca)$ for $\vecx,\veca\in\R^2$,
and as usual we write $\vecy=\cmatr{\vecy_1}{\vecy_2}\in\R^4$ with $\vecy_1,\vecy_2\in\R^2$.
Hence by \eqref{MAINAPPLTHMpf3},
\begin{align}\label{MAINAPPLTHMpf200}
&N_{\alpha,\beta}(\tf,g,T)=\lambda_{\tf}\,\hg(0)
+O\biggl(\int_{\R^4}\bigl|f(\vecy)\bigr|\|\phi_{\eta,\vecy_1}\|_{\L_{167}^2}
\|\phi_{\eta,\vecy_2}\|_{\L_{167}^2} S_{1,2,3}(g)\,\kappa c^{-\frac1{\kappa}}
\,\delta^{\,\frac1{127\kappa}}\,d\vecy\biggr).
\end{align}
Here we have
$\|\phi_{\eta,\vecb}\|_{L_a^2}\asymp_a (1+\|\vecb\|)^a\eta^{-a-1}$ ($\forall \vecb\in\R^2$);
hence, using also the assumption about the support of $f$,
\begin{align*}
\int_{\R^4}\bigl|f(\vecy)\bigr|\|\phi_{\eta,\vecy_1}\|_{\L_{167}^2}
\|\phi_{\eta,\vecy_2}\|_{\L_{167}^2}\,d\vecy
\ll \|f\|_{\L^\infty}\eta^{-336}.
\end{align*}
Combining the above bounds we obtain:
\begin{align*}
\bigl|N_{\alpha,\beta}(f,g,T)-\lambda_f\,\hg(0)\bigr|
&\ll \|f\|_{\L^\infty} S_{1,2,3}(g)\kappa c^{-\frac1{\kappa}}\,\delta^{\,\frac1{127\kappa}}\eta^{-336}
\\
&\hspace{30pt}
+S_{\infty,0,1}(f) S_{1,2,3}(g)\Bigl(1+\kappa c^{-\frac1{\kappa}}\,\delta^{\,\frac1{127\kappa}}\Bigr)\eta
\\
&\ll S_{\infty,0,1}(f) S_{1,2,3}(g)\kappa c^{-\frac1{\kappa}}\Bigl(\delta^{\,\frac1{127\kappa}}\eta^{-336}
+\eta\Bigr).
\end{align*}
Note also that $S_{\infty,0,1}(f)\ll\sum_{j=1}^4\bigl\|\partial_{x_j}f\bigr\|_{\L^\infty}$,
since we assume that the support of $f$ is contained in the unit ball.
Choosing $\eta=(\frac19\,\delta)^{\,\frac1{337\cdot 127\kappa}}$
(this number satisfies $0<\eta<1$, by an observation which we made below \eqref{|v|>PF3}),
we now obtain the bound in Theorem \ref{MAINAPPLTHM} with $B=42799$. %
\end{proof}

\subsection{Consequences of Theorem \ref*{MAINAPPLTHM}}
\label{MAINAPPLTHMcorproofsec}

Let us start by showing that the assumptions in Theorem \ref{MAINAPPLTHM}
on $f$ having a fixed compact support 
and $(\alpha,\beta)$ satisfying $|\alpha|,|\beta|\leq1$,
can both be weakened by simple aposteriori arguments:

\begin{cor}\label{MAINAPPLTHMfnoncptsuppCOR}
Let $B>0$ be as in Theorem \ref{MAINAPPLTHM}, and let $\ve>0$.
Then for any $[\kappa;c]$-Diophantine vector $\vecxi=(\alpha,\beta)\in\R^2$,
any $f\in\C^1(\R^4)$ with $S_{\infty,3+\ve,1}(f)<\infty$,
any $g\in\C^3(\R)$ with $S_{1,2,3}(g)<\infty$,
and any $T\geq\max(1,\ve\|\vecxi\|)$,
\begin{align}\label{MAINAPPLTHMfnoncptsuppCORres}
\biggl|N_{\alpha,\beta}(f,g,T)-\lambda_f\int_{\R}g(s)\,ds\biggr|
\ll_\ve S_{\infty,3+\ve,1}(f)\,S_{1,2,3}(g)\kappa c^{-\frac1{\kappa}}
\,\Bigl(\delta_{6,\vecxi}(T)^{1/(B\kappa)}
+\frac{\|\vecxi\|}T\Bigr)
. %
\end{align}
\end{cor}
We stress that the implied constant in \eqref{MAINAPPLTHMfnoncptsuppCORres}
depends only on $\ve$.

\begin{proof}
Let us first keep $f$ as in Theorem \ref{MAINAPPLTHM},
but allow $\vecxi=(\alpha,\beta)\in\R^2$ to be outside $[-1,1]^2$.
Choose $\veck\in\Z^2$ so that the vector $(\alpha',\beta'):=\vecxi-\veck$
lies in $[-1,1]^2$,
and so that $\veck=\bn$ if already $\vecxi\in[-1,1]^2$.
Of course $(\alpha',\beta')$ is $[\kappa;c]$-Diophantine just like $\vecxi$,
and $\delta_{6,(\alpha',\beta')}(T)=\delta_{6,\vecxi}(T)$ for all $T$.
Recall that the inhomogeneous form $Q$ is defined by \eqref{Qdef};
let $Q'$ be the corresponding form coming from $(\alpha',\beta')$,
i.e.\ $Q'(\vecx)\equiv Q(\vecx+(\veck,\veck))$. 
Then
\begin{align*}
N_{\alpha,\beta}(f,g,T)-N_{\alpha',\beta'}(f,g,T)
=\frac1{T^2}\sum_{\vecm\in\Z^4\setminus\Delta}\Bigl(f\bigl(T^{-1}(\vecm+(\veck,\veck))\bigr)-f(T^{-1}\vecm)\Bigr)g(Q'(\vecm)).
\end{align*}
Here
\begin{align}\label{MAINAPPLTHMfnoncptsuppCORpf50}
\Bigl|f\bigl(T^{-1}(\vecm+(\veck,\veck))\bigr)-f(T^{-1}\vecm)\Bigr|\ll \sum_{j=1}^4\bigl\|\partial_{x_j}f\bigr\|_{\L^\infty}
\frac{\|\veck\|}{T}.
\end{align}
Furthermore, since $f$ is supported in the unit ball,
the difference in \eqref{MAINAPPLTHMfnoncptsuppCORpf50} vanishes whenever
$\|\vecm\|\geq T+\sqrt2\|\veck\|$.
Hence by Lemma \ref{BASICUPPERBDLEM},
\begin{align*}
\bigl|N_{\alpha,\beta}(f,g,T)-N_{\alpha',\beta'}(f,g,T)\bigr|
\ll \sum_{j=1}^4\bigl\|\partial_{x_j}f\bigr\|_{\L^\infty}S_{1,2,1}(g)
\kappa c^{-\frac1{\kappa}}\Bigl(\frac{T+\|\veck\|}{T}\Bigr)^2\frac{\|\veck\|}{T}.
\end{align*}
Note that $\|\veck\|\ll\|\vecxi\|$,
and $\|\vecxi\|\leq\ve^{-1}T$ by assumption;
thus $\bigl(\frac{T+\|\veck\|}{T}\bigr)^2\frac{\|\veck\|}{T}\ll_{\ve}\frac{\|\vecxi\|}{T}.$
Combining the above with Theorem \ref{MAINAPPLTHM} applied to $(\alpha',\beta')$, we conclude that
\begin{align}\label{MAINAPPLTHMgenxi}
\biggl|N_{\alpha,\beta}(f,g,T)-\lambda_f\int_{\R}g(s)\,ds\biggr|
\ll_{\ve} \sum_{j=1}^4\bigl\|\partial_{x_j}f\bigr\|_{\L^\infty}\,S_{1,2,3}(g)\kappa c^{-\frac1{\kappa}}
\,\Bigl(\delta_{6,(\alpha,\beta)}(T)^{1/(B\kappa)}+\frac{\|\vecxi\|}{T}\Bigr),
\end{align}
for all $T\geq\max(1,\ve\|\vecxi\|)$.

We next wish to extend the bound to more general functions $f$,
as in the statement of the corollary.
To achieve this, we will use the fact %
that both $N_{\alpha,\beta}(f,g,T)$ and $\lambda_f$ transform
in an obvious manner under scaling of the function $f$.
Indeed, introducing the scaling operator $\delta_R$ (for any $R>0$) acting on $\C_c(\R^4)$
through $[\delta_Rf](\vecx):=f(R\vecx)$,
we have by immediate inspection in \eqref{NabfgTdef} and \eqref{lambdafdef}:
\begin{align}\label{REDSOBFpf101}
N_{\alpha,\beta}(\delta_Rf,g,T)=R^{-2}N_{\alpha,\beta}(f,g,T/R)
\qquad (T>0)
\end{align}
and
\begin{align}\label{REDSOBFpf102}
\lambda_{\delta_Rf}=R^{-2}\lambda_f.
\end{align}

Now let $f\in\C^1(\R^4)$ with $S_{\infty,3+\ve,1}(f)<\infty$ be given.
We will decompose $f$ dyadically radially, using a partition of unity.
Fix a $\C^\infty$ function $\varphi:\R\to[0,1]$ satisfying $\varphi(r)=0$ for $r\leq0.1$
and $\varphi(r)=1$ for $r\geq0.9$,
and then define the $\C^\infty$-functions $\varphi_0,\varphi_1,\ldots:\R\to[0,1]$
through $\varphi_0(r)=1-\varphi(r-1)$
and 
\begin{align*}
\varphi_j(r)=\begin{cases}
\varphi(r-2^{j-1})&\text{if }\: r< 2^{j}
\\
1-\varphi(r-2^j)&\text{if }\: r\geq 2^{j}
\end{cases}
\qquad (j=1,2,3,\ldots,\: r\in\R).
\end{align*}
Then
$\supp\varphi_0\subset(-\infty,2)$
and $\supp\varphi_j\subset(2^{j-1},2^{j}+1)$ for all $j\geq1$;
furthermore
\begin{align*}
\sum_{j=0}^\infty\varphi_j(r)=1\quad(\forall r\in\R),
\text{ and }\: \|\varphi_j'\|_{\L^\infty}=\|\varphi'\|_{\L^\infty}\quad(j=1,2,3,\ldots).
\end{align*}

Then define $f_j\in\C_c^1(\R^4)$ through
$f_j(\vecx):=\varphi_j(\|\vecx\|) f(\vecx)$.
Then $f(\vecx)=\sum_{j=0}^\infty f_j(\vecx)$, and it follows that
\begin{align}\label{MAINAPPLTHMfnoncptsuppCORpf1}
N_{\alpha,\beta}(f,g,T)=\sum_{j=0}^\infty N_{\alpha,\beta}(f_j,g,T)
\qquad
\text{and}\qquad
\lambda_f=\sum_{j=1}^\infty\lambda_{f_j}.
\end{align}
(To prove the first relation one uses \eqref{NabfgTdef};
the change of order of summation %
is justified since
we have absolute convergence;
$\sum_{j=0}^\infty\sum_{\vecm\in\Z^4\setminus\Delta}\bigl|f_j(T^{-1}\vecm)g(Q(\vecm))\bigr|<\infty$.
This absolute convergence follows from
$\|f_j\|_{\L^\infty}\ll S_{\infty,3,0}(f) 2^{-3j}$ and the fact that the support of $f_j$ 
is contained in the ball of radius $2^{j+1}$ about the origin,
combined with the bound
$\sum_{\substack{\vecm\in\Z^4\setminus\Delta\\\|\vecm\|\leq S}}\bigl|g(Q(\vecm))\bigr|\ll S^2$ 
for $S$ large,
which follows from \eqref{MAINAPPLTHMgenxi} by the argument in the proof of
Lemma \ref{BASICUPPERBDLEM}.
The justification of the second relation in \eqref{MAINAPPLTHMfnoncptsuppCORpf1}
is similar but easier.)

We also set
\begin{align*}
\tf_j:=\delta_{2^{j+1}}f_j.
\end{align*}
Then $\tf_j\in \C_c^1(\R^4)$ and the support of $\tf_j$ is contained in the unit ball centered at $\bn$.
Hence \eqref{MAINAPPLTHMgenxi} applies to $\tf_j$, yielding
\begin{align}\label{MAINAPPLTHMres2}
\biggl|N_{\alpha,\beta}(\tf_j,g,T)-\lambda_{\tf_j}\int_{\R}g(s)\,ds\biggr|
\ll_{\ve} \sum_{k=1}^4\bigl\|\partial_{x_k}\tf_j\bigr\|_{\L^\infty}\,S_{1,2,3}(g)\kappa c^{-\frac1{\kappa}}
\,\Bigl(\delta_{6,(\alpha,\beta)}(T)^{1/(B\kappa)}+\frac{\|\vecxi\|}{T}\Bigr)
\end{align}
for all $T\geq\max(1,\ve\|\vecxi\|)$.
Here
\begin{align}\label{REDSOBFpf1}
\sum_{k=1}^4\bigl\|\partial_{x_k}\tf_j\bigr\|_{\L^\infty}
=2^{j+1}\sum_{k=1}^4\bigl\|\partial_{x_k}f_j\bigr\|_{\L^\infty}
\ll S_{\infty,3+\ve,1}(f)\cdot 2^{-(2+\ve)j}. %
\end{align}
Indeed, for $\vecx\neq\bn$ we have 
$\bigl|\partial_{x_k}\varphi_j(\|\vecx\|)\bigr|
=\bigl|\varphi_j'(\|\vecx\|)x_k\bigr|/\|\vecx\|
\leq\|\varphi'\|_{\L^\infty}$,
while at $\vecx=\bn$
$\partial_{x_k}\varphi_j(\|\vecx\|)$ vanishes;
hence
\begin{align*}
\sum_{k=1}^4\bigl\|\partial_{x_k}f_j\bigr\|_{\L^\infty}
\ll\sup\Bigl\{|f(\vecx)|+\sum_{k=1}^4|\partial_{x_k}f(\vecx)|\col
\vecx\in\R^4,\: \|\vecx\|\in\supp\varphi_j\Bigr\},
\end{align*}
and \eqref{REDSOBFpf1} follows since $\|\vecx\|\in\supp\varphi_j$ implies
$1+\|\vecx\|\asymp 2^j$.
Combining 
\eqref{MAINAPPLTHMres2}, \eqref{REDSOBFpf1}
with
\eqref{REDSOBFpf101}, \eqref{REDSOBFpf102},
we obtain:
\begin{align}\notag
\biggl|N_{\alpha,\beta}(f_j,g,2^{-j-1}T)-\lambda_{f_j}\int_{\R}g(s)\,ds\biggr|
\hspace{150pt}
\\\label{REDSOBFpf103}
\ll 2^{-\ve j} S_{\infty,3+\ve,1}(f)\, S_{1,2,3}(g)\kappa c^{-\frac1{\kappa}}
\,\Bigl(\delta_{6,(\alpha,\beta)}(T)^{1/(B\kappa)}+\frac{\|\vecxi\|}{T}\Bigr).
\end{align}
This holds for all $T\geq\max(1,\ve\|\vecxi\|)$.
We replace $T$ by $2^{j+1}T$ in 
\eqref{REDSOBFpf103}, use \eqref{DELTAESSDECR},
and finally add over all $j$, %
using \eqref{MAINAPPLTHMfnoncptsuppCORpf1}.
This gives \eqref{MAINAPPLTHMfnoncptsuppCORres}.
\end{proof}

Finally we give the proofs of Corollaries \ref{MAINAPPLTHMCOR1} and \ref{BTcor} stated in the introduction.
\begin{proof}[Proof of Corollary \ref{MAINAPPLTHMCOR1}]
Let $\chi:\R^4\to\{0,1\}$ be the characteristic function of the unit ball
and let $\chi_{(a,b)}:\R\to\{0,1\}$ be the characteristic function of the interval $(a,b)$.
For $\eta,\eta'>0$ (two constants which we will fix below) we choose
$f_{\pm}\in\C_c^\infty(\R^4)$ so that $0\leq f_-\leq\chi\leq f_+\leq 1$
and $f_-(\vecx)=1$ whenever $\|\vecx\|\leq1-\eta$ and $f_+(\vecx)=0$ whenever $\|\vecx\|\geq1+\eta$,
and we choose $g_{\pm}\in\C_c^\infty(\R)$ so that
$0\leq g_-\leq\chi_{(a,b)}\leq g_+\leq1$ and $g_-(s)=1$ whenever $a+\eta'\leq s\leq b-\eta'$
and $g_+(s)=0$ whenever $s\leq a-\eta'$ or $s\geq b+\eta'$.
(Thus if $\eta>1$, we may take $f_-\equiv0$ and if $\eta'>\frac12(b-a)$ we may take $g_-\equiv0$.)
These functions can be chosen so that
$S_{\infty,4,1}(f_{\pm})\ll \eta^{-1}$
and $S_{1,2,3}(g_{\pm})\ll \bigl(1+|a|+|b|\bigr)^2\bigl(b-a+{\eta'}^{-2}\bigr)$,
so long as $\eta,\eta'\ll1$.
By construction, we have
\begin{align*}
N_{\alpha,\beta}(f_-,g_-,T)\leq N_{\alpha,\beta}(a,b,T)\leq N_{\alpha,\beta}(f_+,g_+,T).
\end{align*}
We also have $\bigl|\lambda_{f_{\pm}}-\lambda_\chi\bigr|\ll\eta$ and $\lambda_{\chi}=\frac{\pi^2}{2}$,
thus $\lambda_{f_{\pm}}=\frac{\pi^2}{2}+O(\eta)$.
Hence by Corollary \ref{MAINAPPLTHMfnoncptsuppCOR},
\begin{align}\notag
&\bigl|N_{\alpha,\beta}(a,b,T)-\tfrac{\pi^2}2(b-a)\bigr|
\\\notag
&\ll(b-a)\eta+\eta'
+\eta^{-1}\bigl(1+|a|+|b|\bigr)^2\bigl(b-a+{\eta'}^{-2}\bigr)\kappa c^{-\frac1{\kappa}}\delta_{6,(\alpha,\beta)}(T)^{1/(B\kappa)}
\\\label{MAINAPPLTHMCOR1pf1}
&\ll(1+|a|+|b|)^3\kappa c^{-\frac1\kappa}\bigl(\eta+\eta'+\eta^{-1}{\eta'}^{-2}\delta_{6,(\alpha,\beta)}(T)^{1/(B\kappa)}\bigr).
\end{align}
Choosing $\eta=\eta'=\delta_{6,(\alpha,\beta)}(T)^{1/(4B\kappa)}$
we obtain \eqref{MAINAPPLTHMCOR1RES}, with $B'=4B$.
\end{proof}
\begin{remark}
Of course, the bound in \eqref{MAINAPPLTHMCOR1pf1} is often wasteful regarding the dependence on $a,b$.
However, recall that we have to keep $\eta,\eta'\ll1$ in order for the first bound in
\eqref{MAINAPPLTHMCOR1pf1} to be valid,
and our main aim in Corollary \ref{MAINAPPLTHMCOR1} was to give 
a reasonably simple statement of a general bound with an absolute implied constant,
and with a power rate decay with respect to $T$ for any fixed $(\alpha,\beta)$
subject to a Diophantine condition.
\end{remark}

\begin{proof}[Proof of Corollary \ref{BTcor}]
This can again be derived from Theorem \ref{MAINAPPLTHM} by an approximation argument;
however it is easier to argue directly from \eqref{MAINAPPLTHMpfusefulrel2},
since there $\vecm_1$ and $\vecm_2$ appear shifted by $\vecxi_2$, 
which is exactly what we need.
Indeed, let $\chi:\R^2\to\{0,1\}$ be the characteristic function of the open unit ball
centered at the origin and let $\chi_{(-b/2,-a/2)}$ be the characteristic function of the interval $(-b/2,-a/2)$;
then for $g_1=g_2=\chi$ and $\hh=\chi_{(-b/2,-a/2)}$,
the left hand side of \eqref{MAINAPPLTHMpfusefulrel2} is exactly equal to
$\pi R_2[a,b](T^2)$ (cf.\ \eqref{R2def}).
Now the corollary follows by a similar approximation argument as in the proof of
Corollary \ref{MAINAPPLTHMCOR1}.
\end{proof}


\begin{thebibliography}{99}

\bibitem{BT}
M.\ V.\ Berry and M.\ Tabor,
Level clustering in the regular spectrum,
Proc.\ Royal Soc.\ A \textbf{356} (1977), 375--394.

%
%
%
%
%

\bibitem{BD}
J.\ D.\ Bovey and M.\ M.\ Dodson, The Hausdorff dimension of systems of linear forms,
Acta Arith.\ \textbf{45} (1986), 337--358.

\bibitem{BV}
T.\ Browning and I.\ Vinogradov, Effective Ratner theorem for $\SL(2,\R)\ltimes\R^2$ and gaps in $\sqrt n$ modulo $1$,
\linebreak
J.\ London Math.\ Soc.\ \textbf{94} (2016), 61--84.

\bibitem{Burger}
M. Burger, Horocycle flow on geometrically finite surfaces,
Duke Math. J. \textbf{61} (1990), 779--803.

\bibitem{DMS}
C.\ Dettmann, J.\ Marklof, A.\ Str\"ombergsson,
Universal hitting time statistics for integrable flows,
J.\ Stat.\ Phys.\ \textbf{166} (2017), 714--749.

\bibitem{EMV}
M.\ Einsiedler, G.\ Margulis, A.\ Venkatesh,
Effective equidistribution for closed orbits of semisimple groups
on homogeneous spaces,
Invent.\ Math.\ \textbf{177} (2009), 137--212.

%
%
%

\bibitem{EMM98}
A.\ Eskin, G.\ Margulis and S.\ Mozes,
Upper bounds and asymptotics in a quantitative version of the Oppenheim conjecture,
Ann.\ of Math.\ \textbf{147} (1998), 93--141.

\bibitem{EMM05}
A.\ Eskin, G.\ Margulis and S.\ Mozes,
Quadratic forms of signature {$(2,2)$} and eigenvalue spacings on rectangular 2-tori,
Ann.\ of Math.\ \textbf{161} (2005), 679--725.

\bibitem{FF}
L. Flaminio and G. Forni,
Invariant distributions and time averages for horocycle flows,
Duke Math.\ J.\ \textbf{119} (2003), 465--526.

\bibitem{gF99}
G.\ Folland,
\textit{Real analysis},
John Wiley \& Sons Inc., New York, 1999.

\bibitem{grafakos}
L.\ Grafakos, \textit{Classical Fourier Analysis},
Springer-Verlag, 2008.

\bibitem{GreenTao}
B.\ Green and T.\ Tao,
The quantitative behaviour of polynomial orbits on nilmanifolds,
Ann.\ of Math.\ \textbf{175} (2012), 465--540.

%
%
%

%
%
%

\bibitem{GM}
F.\ G\"otze and G.\ Margulis,
Distribution of values of quadratic forms at integral points,
preprint 2010, arXiv:1004.5123
%

\bibitem{HW}
G.\ H.\ Hardy and E.\ M.\ Wright,
\textit{An Introduction to the Theory of Numbers},
Clarendon Press, Oxford, 1938.

%
%
%

\bibitem{IW}
H.\ Iwaniec, 
\textit{Introduction to the Spectral Theory of Automorphic Forms,}
Biblioteca de la Revista Matem\'atica Iberoamericana,
Madrid, 1995.

\bibitem{IK}
H.\ Iwaniec and E.\ Kowalski, \textit{Analytic Number Theory},
American Mathematical Society, 2004.

%
%

\bibitem{Khintchine25}
A.\ Khintchine,
Zwei {B}emerkungen zu einer {A}rbeit des {H}errn {P}erron,
Math.\ Z.\ \textbf{22} (1925), 274--284.

\bibitem{Khintchine26}
A.\ Khintchine,
Zur metrischen {T}heorie der diophantischen {A}pproximationen,
Math.\ Z.\ \textbf{24} (1926), 706--714.

\bibitem{KS}
H.\ H.\ Kim,
\newblock Functoriality for the exterior square of {${\rm GL}\sb 4$} and the
  symmetric fourth of {${\rm GL}\sb 2$}.
\newblock {\em J. Amer. Math. Soc.}, 16(1):139--183 (electronic), 2003.
\newblock With appendix 1 by D.Ramakrishnan and appendix 2 by Kim and
  P. Sarnak.

\bibitem{Knapp}
A.\ W.\ Knapp, \textit{Lie groups beyond an introduction},
Progress in Mathematics \textbf{140},
Birkh\"auser Boston Inc., 2002.

%
%
%
%

\bibitem{LM}
E.\ Lindenstrauss and G.\ Margulis, 
Effective estimates on indefinite ternary forms,
Israel J.\ Math.\ \textbf{203} (2014), 445--499.

\bibitem{Margulis89}
G.\ A.\ Margulis, Indefinite quadratic forms and unipotent flows on homogeneous spaces,
in \textit{Dynamical systems and ergodic theory}, vol.\ 23, pages 399--409,
Banach Center Publ., PWN, Warsaw, 1989.

%
%
%
%

%
%
%

\bibitem{MM}
G.\ A.\ Margulis and A.\ Mohammadi,
Quantitative version of the Oppenheim conjecture for inhomogeneous quadratic forms,
Duke Math. J.\ \textbf{158} (2011), 121--160.

%
%
%
%

\bibitem{MarklofpaircorrI}
J.\ Marklof, Pair correlation densities of inhomogeneous quadratic forms,
Ann.\ of Math.\ \textbf{158} (2003), 419-–471. 

\bibitem{MarklofpaircorrII}
J.\ Marklof, Pair correlation densities of inhomogeneous quadratic forms. II. 
Duke Math.\ J.\ \textbf{115} (2002), 409–434. 

\bibitem{Marklofmeansquare}
J.\ Marklof, Mean square value of exponential sums related to the representation of integers as sums of squares,
Acta Arithmetica \textbf{117} (2005), 353--370.

%
%
%
%

%
%
%
%

\bibitem{tM89} Miyake T, {\em Modular forms}, Springer-Verlag, 1989.

\bibitem{mohammadi2012}
A.\ Mohammadi, A special case of effective equidistribution with explicit constants,
Ergodic Theory Dynam.\ Systems \textbf{32} (2012), 237--247.

\bibitem{Nathanson}
M.\ B.\ Nathanson, \textit{Additive number theory},
Graduate Texts in Mathematics \textbf{164}, Springer-Verlag, 1996.

\bibitem{Perron}
O.\ Perron, \"Uber diophantische {A}pproximationen,
Math.\ Ann.\ \textbf{83} (1921), 77--84.

\bibitem{Prinyasart}
T.\ Prinyasart,
An Effective Equidistribution of Diagonal Translates of Certain Orbits in ASL$(3,\mathbb{Z})\backslash$ASL$(3,\mathbb{R})$,
PhD Thesis, University of California San Diego, 2018.

\bibitem{Ratner91}
M.\,Ratner, On Raghunathan's measure conjecture,
Ann.\ of Math. {\bf 134} (1991) 545-607.

\bibitem{mR91}
M.\ Ratner, Raghunathan's topological conjecture and distributions of unipotent flows,
Duke Math.\ J.\ \textbf{63} (1991), 235--280.

\bibitem{wS67}
W.\ M.\ Schmidt,
On simultaneous approximations of two algebraic numbers by rationals,
Acta Math.\ \textbf{119} (1967), 27--50.

\bibitem{wS70}
W.\ M.\ Schmidt,
Simultaneous approximation to algebraic numbers by rationals,
Acta Math.\ \textbf{125} (1970), 189--201.

\bibitem{Selberg65}
A.\ Selberg,
On the estimation of Fourier coefficients of modular forms,
Proc.\ Sympos.\ Pure Math., Amer.\ Math.\ Soc., vol. VIII, 1965, 1-­15.

\bibitem{Shah}
N. Shah, Limit distributions of expanding translates of certain orbits
on homogeneous spaces,
Proc. Indian Acad. Sci. (Math. Sci.) 
\textbf{106} (1996), 105--125.

\bibitem{iha}
A.\ Str\"ombergsson,
On the deviation of ergodic averages for horocycle flows,
Journal of Modern Dynamics, \textbf{7} (2013), 291--328.
%

\bibitem{SASL}
A.\ Str\"ombergsson,
Effective Ratner equidistribution for $\SL(2,\R)\ltimes\R^2$,
Duke Math.\ J., \textbf{164} (2015), 843--902.

%
%
%
%

%
%
%
%

\bibitem{Weil}
A.\ Weil, On some exponential sums,
Proc.\ Nat.\ Acad.\ Sci.\ U.S.A.\ \textbf{34} (1948), 204--207.

%
%
%

\end{thebibliography}
\end{document}